\documentclass[a4paper, twoside, 11pt, english]{article}

\usepackage{amsmath, amsthm, amsfonts, amssymb}
\usepackage[T1]{fontenc}
\usepackage[utf8]{inputenc}
\usepackage{babel}
\usepackage{times, euler, euscript, stmaryrd, mathabx}
\usepackage{fancyhdr, titlesec, url, enumerate}
\usepackage{xy4}
\usepackage{catex}
\usepackage{hyperref}

\DeclareMathAlphabet{\mathscr}{T1}{pzc}{m}{it} 

\newcommand{\pdf}[1]{\texorpdfstring{$#1$}{1}}
\newcommand{\sm}{\scriptstyle}

\titleformat{\section}[block]{\scshape\filcenter\Large}{\thesection.}{.5em}{}
\titleformat{\subsection}[block]{\bfseries\filcenter\large}{\thesubsection.}{.5em}{\medskip}
\titleformat{\subsubsection}[runin]{\bfseries}{\thesubsubsection.}{.5em}{}[.]
\titlespacing{\subsubsection}{0pt}{10pt}{.5em}

\newtheoremstyle{ntheorem}%
	{\topsep}{\topsep}{\itshape}{0pt}{\bfseries}{.}{.5em}%
	{\thmnumber{#2.\hspace{.5em}}\thmname{#1}\thmnote{ (#3)}}
	
\newtheoremstyle{ndefinition}%
	{\topsep}{\topsep}{\normalfont}{0pt}{\bfseries}{.}{.5em}%
	{\thmnumber{#2.\hspace{.5em}}\thmname{#1}\thmnote{ (#3)}}
	
\newtheoremstyle{nremark}%
	{\topsep}{\topsep}{\normalfont}{0pt}{\itshape}{.}{.5em}%
	{\thmnumber{}\thmname{#1}\thmnote{ (#3)}}

\theoremstyle{ntheorem}
  	\newtheorem{theorem}[subsubsection]{Theorem}
  	\newtheorem{proposition}[subsubsection]{Proposition}
	\newtheorem{lemma}[subsubsection]{Lemma}
  	\newtheorem{corollary}[subsubsection]{Corollary}

\theoremstyle{ndefinition}

	\newtheorem{example}[subsubsection]{Example}

\theoremstyle{nremark}

\fancypagestyle{plain}{
\fancyhf{}\fancyfoot[LC,RC]{}
\fancyfoot[LE,RO]{$\thepage$}

}

\newcommand{\fl}{\rightarrow}
\newcommand{\fll}{\longrightarrow}

\newcommand{\ofll}[1]{\overset{\displaystyle #1}{\fll}}
\newcommand{\ifl}{\rightarrowtail}
\newcommand{\pfl}{\twoheadrightarrow}

\newcommand{\dfl}{\Rightarrow}
\newcommand{\dfll}{\Longrightarrow}

\newcommand{\odfll}[1]{\overset{\displaystyle #1}{\dfll}}

\newcommand{\tfl}{\Rrightarrow}

\newcommand{\qfl}{\xymatrix@1@C=10pt{\ar@4 [r] &}}

\newcommand{\ens}[1]{\left\{ #1 \right\}}
\newcommand{\pres}[2]{\langle #1 \;\vert\; #2 \rangle}
\newcommand{\bigpres}[2]{\big\langle\; #1 \;\;\big\vert\;\; #2 \;\big\rangle}
\newcommand{\cohpres}[3]{\langle #1 \;\vert\; #2 \;\vert\; #3 \rangle}

\newcommand{\cl}[1]{\overline{#1}}
\newcommand{\rep}[1]{\widehat{#1}}
\renewcommand{\tilde}[1]{\widetilde{#1}}
\newcommand{\tck}[1]{{#1}{}^{\top}}

\DeclareMathOperator{\id}{Id}
\DeclareMathOperator{\im}{im}
\DeclareMathOperator{\FDT}{FDT}
\DeclareMathOperator{\FP}{FP}
\DeclareMathOperator{\Std}{Std}
\DeclareMathOperator{\Sq}{Sq}
\DeclareMathOperator{\LP}{LP}

\renewcommand{\phi}{\varphi}
\renewcommand{\epsilon}{\varepsilon}

\newcommand{\Nb}{\mathbb{N}}
\newcommand{\Zb}{\mathbb{Z}}

\newcommand{\Br}{\EuScript{B}}
\newcommand{\Cr}{\EuScript{C}}
\newcommand{\Sr}{\EuScript{S}}

\newcommand{\B}{\mathbf{B}}
\newcommand{\C}{\mathbf{C}}
\newcommand{\D}{\mathbf{D}}

\newcommand{\M}{\mathbf{M}}
\renewcommand{\S}{\mathbf{S}}

\deftwocell[gray]{mu : 2 -> 1}
\deftwocell[black]{asso : 3 -> 1}
\deftwocell[white]{penta : 4 -> 1}

\begin{document}

\thispagestyle{empty}
\begin{center}

\begin{LARGE}
\textsc{Polygraphs of finite derivation type}
\end{LARGE}

\bigskip\hrule height 1.5pt \bigskip

\begin{large}\begin{uppercase}
{Yves Guiraud \qquad Philippe Malbos}
\end{uppercase}\end{large}

\vspace*{\stretch{2}}

\begin{small}\begin{minipage}{14cm}
\noindent\textbf{Abstract --} 
Craig Squier proved that, if a monoid can be presented by a finite convergent string rewriting system, then it satisfies the homological finiteness condition left-$\FP_3$. Using this result, he constructed finitely presentable monoids with a decidable word problem, but that cannot be presented by finite convergent rewriting systems. Later, he introduced the condition of finite derivation type, which is a homotopical finiteness property on the presentation complex associated to a monoid presentation. He showed that this condition is an invariant of finite presentations and he gave a constructive way to prove this finiteness property based on the computation of the critical branchings: being of finite derivation type is a necessary condition for a finitely presented monoid to admit a finite convergent presentation. This survey presents Squier's results in the contemporary language of polygraphs and higher-dimensional categories, with new proofs and relations between them.


\smallskip\noindent\textbf{M.S.C. 2000 --} 68Q42, 03D05, 18D05.
\end{minipage}\end{small}

\vspace*{\stretch{1}}

\begin{small}\begin{minipage}{12cm}
\renewcommand{\contentsname}{\large\MakeUppercase{}}
\setcounter{tocdepth}{2}
\tableofcontents
\end{minipage}\end{small}
\end{center}

\newpage
\section{Introduction}

Given a monoid~$\M$, a generating set~$\Sigma_1$ for~$\M$ provides a way to represent the elements of~$\M$ in the free monoid~$\Sigma_1^*$, i.e.\ as finite words written with the elements of~$\Sigma_1$. But, in general, an element of~$\M$ has several representatives in~$\Sigma_1^*$. The \emph{word problem} for~$\M$ consists in finding a generating set~$\Sigma_1$ and a procedure that can determine whether or not any two elements of~$\Sigma_1^*$ represent the same element in the monoid~$\M$.

\subsubsection*{The word problem and convergent presentations}

One way to solve the word problem is to exhibit a finite presentation $\Sigma=(\Sigma_1,\Sigma_2)$ of~$\M$, made of a generating set~$\Sigma_1$ and a set~$\Sigma_2$ of directed relations with a good computational property: \emph{convergence}. Indeed, in rewriting theory, one studies presentations where the relations in~$\Sigma_2$ are not seen as equalities between the words in~$\Sigma_1^*$, such as~$u=v$, but, instead, as \emph{rewriting rules} that can only be applied in one direction, like $u\dfl v$,
thus simulating a non-reversible computational process \emph{reducing} the word~$u$ into the word~$v$. 

In rewriting theory, such a presentation~$\Sigma$ of a monoid is called a \emph{string rewriting system} or, historically, a \emph{semi-Thue system}; in that case, the directed relations of~$\Sigma_2$ are called \emph{rewriting rules}. A presentation~$\Sigma$ is \emph{convergent} if it has the two properties of
\begin{itemize}
\item \emph{termination}, i.e.\ all the computations end eventually, and
\item \emph{confluence}, i.e.\ different computations on the same input lead to the same result.
\end{itemize}
The \emph{monoid presented by~$\Sigma$} is defined as the quotient, denoted by~$\cl{\Sigma}$, of the free monoid~$\Sigma_1^*$ over~$\Sigma_1$ by the congruence generated by~$\Sigma_2$. By extension, we say that~$\Sigma$ presents any monoid isomorphic to~$\cl{\Sigma}$.

A finite and convergent presentation~$\Sigma$ of a monoid~$\M$ gives a solution to the word problem, called the \emph{normal-form procedure} and defined as follows. Given an element~$u$ of the free monoid~$\Sigma_1^*$, convergence ensures that all the applications of (directed) relations to~$u$, in every possible manner, will eventually produce a unique result: an element~$\rep{u}$ of~$\Sigma_1^*$ where no relation applies anymore. The word~$\rep{u}$ is called the \emph{normal form of~$u$}. By construction, two elements~$u$ and~$v$ of~$\Sigma_1^*$ represent the same element of~$\M$ if, and only if, their normal forms are equal in~$\Sigma_1^*$. Finiteness ensures that one can determine whether an element of~$\Sigma_1^*$ is a normal form or not, by examining all the relations. 

\subsubsection*{Rewriting and polygraphs}

The notion of string rewriting system comes from combinatorial algebra. It was introduced by Axel Thue in 1914 in order to solve the word problem for finitely presented semigroups~\cite{Thue14}.
It is only in 1947 that the problem was shown to be undecidable, independently by Emil Post~\cite{Post47} and Andrei Markov~\cite{Markov47a,Markov47b}. Then in 1943, Maxwell Newman gave a general setting, the abstract rewriting theory, to describe the properties of termination and confluence, and to show the first fundamental result of rewriting: Newman's lemma~\cite{Newman42}. Since then, rewriting theory has been mainly developed in theoretical computer science, producing numerous variants corresponding to different syntaxes of the formulas being transformed: string, terms, terms modulo, $\lambda$-terms, term-graphs, etc. Rewriting is also present in other computational formalisms such as Petri nets or logical systems.

More recently, higher-dimensional rewriting has unified several paradigms of rewriting. This approach is based on presentations by generators and relations of higher-dimensional categories, independently introduced by Albert Burroni and Ross Street under the respective names of \emph{polygraphs} in~\cite{Burroni93} and \emph{computads} in~\cite{Street76,Street87}. Those algebraic objects have been subsequently developed in rewriting theory, fixing the terminology to polygraph in that field~\cite{Metayer03, Guiraud06jpaa, Lafont07, Metayer08, GuiraudMalbos09, Mimram10, GuiraudMalbos11, GuiraudMalbos12mscs, GuiraudMalbos12advances, GuiraudMalbosMimram13, GaussentGuiraudMalbos}. 

The main useful property of polygraphs is to encapsulate, in the same globular object, an algebraic structure corresponding to the syntax of the terms and to the computations on the terms, together with a homotopical structure describing the properties of the computations. As a consequence, polygraphs provide a natural setting to formulate Squier's theory, based on the discovery of deep relations between the computational, the homological and the homotopical properties of presentations by generators and relations.

\subsubsection*{From computational to homological properties}

The normal-form procedure proves that, if a monoid admits a finite convergent presentation, then it has a decidable word problem. The converse implication was still an open problem in the middle of the eighties:

\begin{quote}
\noindent{\bf Question. \cite{Jantzen82,Jantzen85}}
\begin{em}
Does every finitely presented monoid with a decidable word problem admit a finite convergent presentation?
\end{em}
\end{quote}
 
In~\cite{KapurNarendran85}, Deepak Kapur and Paliath Narendran consider Artin's presentation of the monoid~$\B_3^+$ of positive braids on three strands
\[
\pres{s,t}{sts \dfl tst}.
\]
Kapur and Narendran proved that~$\B_3^+$ admits no finite convergent presentation on the two generators~$s$ and~$t$. However, they also proved one gets a finite convergent presentation of~$\B_3^+$ by adjunction of a new generator~$a$ standing for the product~$st$:
\begin{equation}
\label{E:KN}
\bigpres
	{s,t,a}
	{ta \odfll{\alpha} as,\; st\odfll{\beta}a,\; sas \odfll{\gamma} aa,\; saa \odfll{\delta} aat}.
\end{equation}
As a consequence, the word problem for~$\B_3^+$ is solvable. The result of Kapur and Narendran shows that the existence of a finite convergent presentation depends on the chosen generators. Thus, to provide the awaited negative answer to the open question, one would have to exhibit a monoid with a decidable word problem but with no finite convergent presentation \emph{for any possible set of generators}: new methods had to be introduced.

And, indeed, Craig Squier answered the question by linking the existence of a finite convergent presentation for a given monoid~$\M$ to an invariant of the monoid: the \emph{homological type left-$\FP_3$}. Here, invariant is to be taken in the sense that it is independent of the choice of a presentation of~$\M$ and, in particular, of a generating set. A monoid~$\M$ is of homological type left-$\FP_3$ if there exists an exact sequence
\[
\xymatrix{
P_3
	\ar [r]  
& P_2
	\ar [r] 
& P_1
	\ar [r] 
& P_0
	\ar [r] 
& \Zb
	\ar [r]
& 0
}
\]
of projective and finitely generated (left) $\Zb\M$-modules, where~$\Zb$ denotes the trivial $\Zb\M$-module. From a presentation~$\Sigma$ of a monoid~$\M$, one can build an exact sequence of free $\Zb\M$-modules
\begin{equation}
\label{resolution2}
\xymatrix{
\Zb\M[\Sigma_2] 
	\ar [r] ^-{d_2} 
& \Zb\M[\Sigma_1]
	\ar [r] ^-{d_1}
& \Zb\M
	\ar [r] ^-{\epsilon}
& \Zb
	\ar [r]
& 0,
}
\end{equation}
where $\Zb\M[\Sigma_k]$ is the free $\Zb\M$-module over~$\Sigma_k$. In~\cite{Squier87}, Squier proved that, when~$\Sigma$ is convergent, its \emph{critical branchings} form a generating set of the kernel of~$d_2$, where a critical branching of~$\Sigma$ is a minimal overlapping application of two relations on the same word of~$\Sigma_1^*$. For example, the relations $\alpha:ta \dfl as$ and $\beta:st \dfl a$ generate a critical branching $(\beta a, s\alpha)$ on~$sta$:
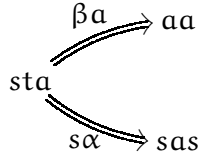
\[
\xymatrix @C=3em @R=1em {
& aa
\\
sta
	\ar@2@/^/ [ur] ^-{\beta a}
	\ar@2@/_/ [dr] _-{s\alpha}
\\
& sas
}	
\]
The convergence of~$\Sigma$ ensures that every critical branching~$(f,g)$ is confluent, that is, it can be completed by rewriting sequences~$f'$ and~$g'$ as in
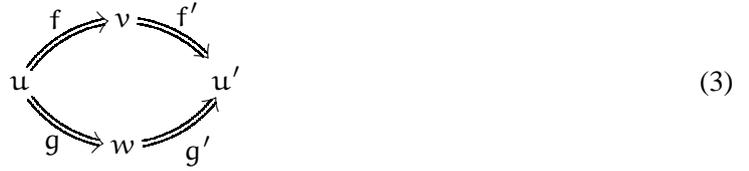
\begin{equation}
\label{Equation:Confluence}
\vcenter{\xymatrix @!C @C=2em @R=1em {
& v
	\ar@2@/^/ [dr] ^-{f'}
\\
u 
	\ar@2@/^/ [ur] ^-{f}
	\ar@2@/_/ [dr] _-{g}
&& u'
\\
& w
	\ar@2@/_/ [ur] _-{g'}
}}	
\end{equation}
For example, the presentation~\eqref{E:KN} of~$\B_3^+$ has four critical branchings, and all of them are confluent:
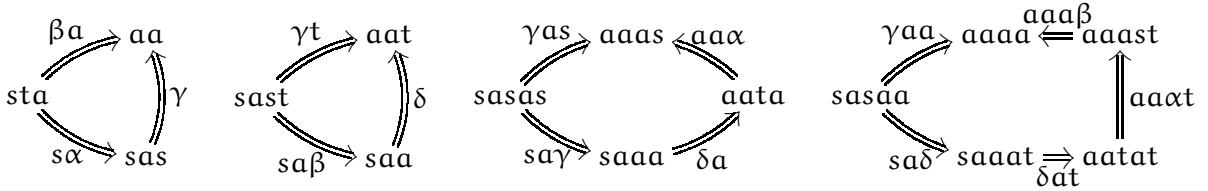
\[
\xymatrix@R=1em@C=2em{
& aa
\\
sta
	\ar@2@/^/ [ur] ^{\beta a} 
	\ar@2@/_/ [dr] _{s\alpha} 
\\
& sas
	\ar@2@/_/ [uu] _{\gamma}
}
\quad
\xymatrix@R=1em@C=2em{
& aat
\\
sast
	\ar@2@/^/ [ur] ^{\gamma t} _(0.66){}="src"
	\ar@2@/_/ [dr] _{sa\beta} ^(0.658){}="tgt"
\\
& saa
	\ar@2@/_/ [uu] _{\delta}
}
\quad
\xymatrix@R=1em@C=1em@!C{
& aaas
\\
sasas
	\ar@2@/^/ [ur] ^{\gamma as}
	\ar@2@/_/ [dr] _{sa\gamma}
&& aata
	\ar@2@/_/ [ul] _{aa\alpha}
\\
& saaa
	\ar@2@/_/ [ur] _{\delta a}
}
\quad
\xymatrix@R=1em@C=1em{
& aaaa
& aaast
	\ar@2 [l] _{aaa\beta}
\\
sasaa
	\ar@2@/^/ [ur] ^{\gamma aa}
	\ar@2@/_/ [dr] _{sa\delta}
\\
& saaat
	\ar@2 [r] _{\delta at}
& aatat
	\ar@2 [uu] _{aa\alpha t}
}
\]
Squier proved that the set~$\Sigma_3$ of critical branchings of a convergent presentation~$\Sigma$ extends the exact sequence~\eqref{resolution2} by one step:
\begin{equation}
\label{resolution3}
\xymatrix{
\Zb\M[\Sigma_3] 
	\ar [r] ^-{d_3} 
& \Zb\M[\Sigma_2] 
	\ar [r] ^-{d_2} 
& \Zb\M[\Sigma_1]
	\ar [r] ^-{d_1}
& \Zb\M
	\ar [r] ^-{\epsilon}
& \Zb
	\ar [r]
& 0,
}
\end{equation}
where the boundary map~$d_3$ is defined on the generic branching~\eqref{Equation:Confluence} by 
\[
d_3(f,g) \:=\: [f] - [g] + [f'] - [g'],
\]
where~$[\cdot]$ satisfies
\[
[ \: \odfll{\displaystyle u_1\alpha_1 v_1} \cdots \odfll{\displaystyle u_n\alpha_n v_n} \: ] 
\:=\: 
\cl{u}_1 [\alpha_1] + \cdots + \cl{u}_n [\alpha_n].
\]
Moreover, when the presentation~$\Sigma$ is finite, then its set of critical branchings is finite, yielding \emph{Squier's homological theorem}.

\begin{quote}
\noindent{\bf \ref{Theorem:SquierAb1}. Theorem {\cite[Theorem~4.1]{Squier87}}}
\emph{
If a monoid admits a finite convergent presentation, then it is of homological type left-$\FP_3$.
}
\end{quote}
Finally, Squier considers in \cite{Squier87} the family~$\S_k$ of monoids, for~$k\geq 2$, presented by
\[
\bigpres
	{a, b, t, (x_i)_{1\leq i\leq k}, (y_i)_{1\leq i\leq k}} 
	{
		\alpha,\;
		(\beta_i)_{1\leq i\leq k},\;
		(\gamma_i)_{1\leq i\leq k},\;
		(\delta_i)_{1\leq i\leq k},\;
		(\epsilon_i)_{1\leq i\leq k}
	}
\]
with 
\[
ab \odfll{\alpha} 1 \,,
\qquad
x_ia \odfll{\beta_i} atx_i \,,
\qquad
x_it \odfll{\gamma_i} tx_i \,,
\qquad
x_ib \odfll{\delta_i} bx_i \,,
\qquad
x_iy_i \odfll{\epsilon_i} 1.
\]
Each~$\S_k$ is a finitely generated monoid with a decidable word problem. But, for~$k\geq 2$, the monoid~$\S_k$ is not of homological type left-$\FP_3$ and, as a consequence, it does not admit a finite convergent presentation. Thus, Squier gave a negative answer to the open question: there exist finitely generated monoids with a decidable word problem that do not admit a finite convergent presentation (for any possible finite set of generators). 

\subsubsection*{Finite derivation type}

Given a monoid~$\M$ with a presentation~$\Sigma$, Squier considers in~\cite{Squier94} the presentation complex of~$\Sigma$, that is a cellular complex with one $0$-cell, whose $1$-cells are the elements of the free monoid~$\Sigma_1^*$ and whose $2$-cells are generated by the relations of~$\Sigma_2$. More precisely, the $2$-cells of the presentation complex are constructed as follows. We denote by~$\Sigma_2^-$ the set obtained by reversing the relations of~$\Sigma_2$:
\[
\Sigma_2^- \:=\: \big\{\; v\odfll{\alpha^-} u\;\;\text{such that}\;\; u\odfll{\alpha} v \text{ is in }\Sigma_2\;\big\}.
\] 
There is a $2$-cell in the presentation complex between each pair of words with shape~$wuw'$ and~$wvw'$ such that $\Sigma_2 \amalg \Sigma_2^-$ contains a relation~$u\dfl v$. Then, Squier extends this $2$-dimension complex with $3$-cells filling all the squares formed by independent applications of relations, such as the following one, where~$u_1\dfl v_1$ and~$u_2 \dfl v_2$ are in $\Sigma_2\amalg \Sigma_2^-$:
\[
\xymatrix @!C @C=2.5em @R=1.5em {
& {wv_1w'u_2w''}
	\ar@2 @/^/ [dr]
		\ar@{} [dd] |-{\equiv}
\\
{wu_1w'u_2w''} 
	\ar@{=>}@/^/ [ur]
	\ar@{=>}@/_/ [dr]
&& {wv_1w'v_2w''}
\\
& {wu_1w'v_2w''}
	\ar@2 @/_/ [ur]
}	
\]
We will see that, in the $2$-categorical formulation of this complex that we consider, this $3$-cell corresponds to the so-called exchange relations.

A \emph{homotopy basis of~$\Sigma$} is a set~$\Sigma_3$ of additional $3$-cells that makes Squier's complex aspherical, i.e.\ such that every $2$-dimensional sphere can be ``filled'' by the $3$-cells of~$\Sigma_3$. The presentation~$\Sigma$ is of \emph{finite derivation type} if it admits a finite homotopy basis. Squier proved that the finite derivation type property is an intrinsic property of the presented monoid:

\begin{quote}
\noindent{\bf \ref{Theorem:Squier2}. Theorem {\cite[Theorem~4.3]{Squier94}}}
\emph{
Let~$\Sigma$ and~$\Xi$ be two finite presentations of the same monoid. Then~$\Sigma$ has finite derivation type if and only if~$\Xi$ has finite derivation type.
}
\end{quote}
The proof given by Squier is based on Tietze transformations. In these notes, we give another proof based on a homotopy bases transfer theorem, Theorem~\ref{Theorem:HomotopyBasesTransfer}. As a consequence of Theorem~\ref{Theorem:Squier2}, we can define the condition of finite derivation type for monoids independently of a considered presentation: a monoid is of \emph{finite derivation type} if its finite presentations are of finite derivation type.

\subsubsection*{From computational to homotopical properties}

In~\cite{Squier94}, Squier links the existence of a finite convergent presentation to the homotopical property of finite derivation type. He proves that, given a convergent presentation~$\Sigma$, it is sufficient to consider one $3$-cell filling the diagram~\eqref{Equation:Confluence} for each critical branching to get a homotopy basis of~$\Sigma$. Such a set of $3$-cells is called a \emph{family of generating confluences of~$\Sigma$}.

\begin{quote}
\noindent{\bf \ref{Theorem:Squier1}. Theorem {\cite[Theorem~5.2]{Squier94}}}
\emph{
Let~$\Sigma$ be a convergent presentation of a monoid. Every family of generating confluences of~$\Sigma$ is a homotopy basis.
}
\end{quote}
Moreover, if~$\Sigma$ is finite, the presentation~$\Sigma$ has finitely many critical branchings.

\begin{quote}
\noindent{\bf \ref{Theorem:Squier}. Theorem {\cite[Theorem~5.3]{Squier94}}}
\emph{
If a monoid admits a finite convergent presentation, then it is of finite derivation type.
}
\end{quote}

In \cite{Squier94}, Squier used Theorem~\ref{Theorem:Squier} to give another proof that there exist finitely generated monoids with a decidable word problem that do not admit a finite convergent presentation. 
Moreover, he showed that the homological finiteness condition left-$\FP_3$
is not sufficient for a finitely presented monoid with a decidable word problem to admit a finite convergent presentation.
Indeed, the monoid~$\S_1$ given by the presentation
\[
\bigpres {a,b,t,x,y} {ab \dfl 1,\; xa\dfl atx,\; xt \dfl tx,\;xb\dfl bx,\;xy\dfl 1}
\]
has a decidable word problem, admits a finite presentation and is of homological type left-$\FP_3$, yet it is not of a finite derivation type (and, thus, it does not admit a finite convergent presentation). This example is entirely developed in Section~\ref{Example:MonoidSk}.

\subsubsection*{Refinements of Squier's conditions}

Squier's homological and homotopical finiteness conditions are related: finite derivation type implies homological type left-$\FP_3$, as proved by several authors~\cite{CremannsOtto94,Pride95,Lafont95}. The converse implication is false in general, as already noted above with the monoid~$\S_1$, but it is true in the special case of groups~\cite{CremannsOtto96}. Squier has proved in~\cite{Squier94} that the invariant homological type left-$\FP_3$ is a necessary, but not sufficient condition for a monoid to admit a finite convergent presentation, the same question being still open for the homotopical invariant finite derivation type. After Squier, various refinements of both invariants have been explored, in the quest for a complete characterisation of the existence of finite convergent presentations of monoids. 

In the homological direction, thanks to the notion of abelian resolution, one defines the more restrictive conditions homological type left-$\FP_n$, for every natural number~$n>3$, and homological type left-$\FP_{\infty}$: a monoid~$\M$ has homological type left-$\FP_{\infty}$ if there exists a resolution of the trivial $\Zb\M$-module by finitely generated and projective $\Zb\M$-modules. In~\cite{Kobayashi90}, a notion of $n$-fold critical branching is used to complete the exact sequence~\eqref{resolution3} into a resolution, obtaining the following implication: if a monoid admits a finite convergent presentation, then it is of homological type left-$\FP_{\infty}$, the converse implication still being false in general. 
The same results are also known for associative algebras presented by a finite Gr\"obner basis~\cite{Anick86} and for groups~\cite{Cohen92,Brown92,Groves90}. One can obtain similar implications with the properties right-$\FP_{\infty}$ and bi-$\FP_{\infty}$, defined with resolutions by right modules and bimodules, respectively. 
In~\cite{GuiraudMalbos12advances}, the authors give a construction of a resolution involving $n$-fold critical branchings and based on the notion of normalisation strategy.

In the homotopical direction, the condition finite derivation type has been refined into~$\FDT_4$, a property about the existence of a finite presentation with a finite homotopy basis, itself satisfying a homotopical finiteness property~\cite{PrideGlashanPasku05}. The condition~$\FDT_4$ is also necessary for a monoid to admit a finite convergent presentation and it is sufficient, but not necessary, for having the conditions homological type left/right/bi-$\FP_4$. Higher-dimensional finite derivation type properties, called~$\FDT_n$, are defined in~\cite{GuiraudMalbos12advances}, as a generalisation in any dimension of Squier's finite derivation type. A monoid with a finite convergent presentation is~$\FDT_{\infty}$ and, for any~$n$, the property~$\FDT_n$ implies the homological type~$\FP_n$~\cite{GuiraudMalbos12advances}.

\subsubsection*{Organisation and prerequisites}

In Section~\ref{Section:twoDimensionalPolygraphs}, we consider presentations of monoids (and, more generally, of categories) by $2$-polygraphs, and we explain their main rewriting properties in Section~\ref{Section:Rewriting}. In Section~\ref{Section:threeDimensionalPolygraphs}, we introduce the property of finite derivation type for categories using the structure of $3$-polygraph and, in Section~\ref{Section:SquierCompletion}, we relate convergence and finite derivation type. This survey is rather self-contained, but wider categorical notions are covered in more detail by Mac Lane in~\cite{MacLane98} and by Barr and Wells in~\cite{BarrWells}. For notions of rewriting theory, one can refer to Baader and Nipkow~\cite{BaaderNipkow98}, Terese~\cite{Terese03} and Book and Otto~\cite{BookOtto93} for the special case of string rewriting. For extension of the finite derivation type property to higher-dimensional rewriting systems, we refer the reader to~\cite{GuiraudMalbos09}.

\section{Low-dimensional categories and polygraphs}
\label{Section:twoDimensionalPolygraphs}

\subsection{Categories and functors}

\subsubsection{Categories}
\label{Categories:OriginalDefinition}

A \emph{category} (or $1$-category) is a data~$\C$ made of a set~$\C_0$, whose elements are called the \emph{$0$-cells of~$\C$}, and, for all $0$-cells~$x$ and~$y$ of~$\C$, of a set~$\C(x,y)$, whose elements are called the \emph{$1$-cells from~$x$ to~$y$ of~$\C$}. Those sets are equipped with the following algebraic structure:
\begin{itemize}
\item for all $0$-cells~$x$, $y$ and~$z$ of~$\C$, a map, called the \emph{composition} (or $0$-composition) of~$\C$,
\[
\gamma_{x,y,z} \::\: \C(x,y)\times\C(y,z) \:\fl\: \C(x,z),
\]
\item for every $0$-cell~$x$ of~$\C$, a specified element~$1_x$ of~$\C(x,x)$, called the \emph{identity of~$x$}.
\end{itemize}
The following relations are required to hold, where we write $u:x\fl y$ to mean that~$u$ belongs to~$\C(x,y)$:
\begin{itemize}
\item the composition is associative, i.e.\ for all $0$-cells~$x$, $y$, $z$ and~$t$ and for all $1$-cells $u:x\fl y$, $v:y\fl z$ and $w:z\fl t$ of~$\C$ ,
\[
\gamma_{x,z,t}(\gamma_{x,y,z}(u,v),w) \:=\: \gamma_{x,y,t}(u,\gamma_{y,z,t}(v,w)),
\]
\item the identities are local units for the composition, i.e.\ for all $0$-cells~$x$ and~$y$ and for every $1$-cell $u:x\fl y$ of~$\C$, 
\[
\gamma_{x,x,y}(1_x,u) \:=\: u \:=\: \gamma_{x,y,y}(u,1_y).
\]
\end{itemize}
A \emph{groupoid} is a category where every $1$-cell has an inverse, that is, for every $1$-cell $u:x\fl y$, there exists a (necessarily unique) $1$-cell $u^-:y\fl x$ such that
\[
\gamma_{x,y,x} (u,u^-) \:=\: 1_x 
\qquad\text{and}\qquad
\gamma_{y,x,y} (u^-,u) \:=\: 1_y.
\]
Monoids (resp.\ groups) correspond exactly to categories (resp.\ groupoids) with only one $0$-cell.

\subsubsection{A globular point of view}

A category~$\C$ can also be seen as a graph
\[
\xymatrix{
\C_0 
& 
\C_1
\ar@<+0.5ex>[l]^{t_0}
\ar@<-0.5ex>[l]_{s_0}
}
\]
where~$\C_1$ is the disjoint union of all the hom-sets~$\C(x,y)$, and the maps~$s_0$ and~$t_0$ send a $1$-cell $u:x\fl y$ to its \emph{source}~$x$ and its \emph{target}~$y$, respectively. We usually simply write~$s(u)$ and~$t(u)$ instead of~$s_0(u)$ and~$t_0(u)$. The composition of~$\C$ equips this graph with a partial function
\[
\star_0 : \C_1 \times \C_1 \fl \C_1
\] 
mapping a pair~$(u,v)$ of \emph{composable} $1$-cells (i.e.\ such that $t(u)=s(v)$) to $u\star_0 v$ (we often simply write~$uv$). By definition, the source and target of a composite $1$-cell are given by
\[
s(u\star_0 v) = s(u)  \qquad\text{and}\qquad  t(u\star_0 v) = t(v).
\]
Moreover, the associativity axiom is written as
\[
(u\star_0 v)\star_0 w \:=\: u\star_0(v\star_0 w).
\]
The identities define an inclusion
\[
\C_0 \ifl \C_1
\]
that maps each $0$-cell~$x$ to the $1$-cell~$1_x$. By definition, the source and target of an identity $1$-cell are
\[
s(1_x) \:=\: x \qquad\text{and}\qquad t(1_x) \:=\:  x.
\]
Finally, the unit axioms become
\[
1_x \star_0 u \:=\: u \:=\: u \star_0 1_y.
\]
This globular definition of categories is equivalent to the original definition given in \ref{Categories:OriginalDefinition}.

\subsubsection{Functors}

Let~$\C$ and~$\D$ be categories. A \emph{functor $F:\C\fl\D$} is a data made of a map $F_0:\C_0\fl\D_0$ and, for all $0$-cells~$x$ and~$y$ of~$\C$, a map
\[
F_{x,y} \::\: \C(x,y) \:\fl\: \D(F(x),F(y)),
\]
such that the following relations are satisfied:
\begin{itemize}
\item for all $0$-cells~$x$, $y$ and~$z$ and all $1$-cells $u:x\fl y$ and $v:y\fl z$ of~$\C$,
\[
F_{x,z}(u\star_0 v) \:=\: F_{x,y}(u) \star_0 F_{y,z}(v),
\]
\item for every $0$-cell~$x$ of~$\C$,
\[
F_{x,x}(1_x) \:=\: 1_{F(x)}.
\]
\end{itemize}
We often just write~$F(x)$ for~$F_0(x)$ and~$F(u)$ for $F_{x,y}(u)$. A functor~$F$ is an \emph{isomorphism} if the map~$F_0$ and each map~$F_{x,y}$ is a bijection. We denote by~$\id_{\C}$ the identity functor on~$\C$. Morphisms of monoids are exactly the functors between the corresponding categories with one $0$-cell. 

\subsection{The word problem}

\subsubsection{\pdf{1}-polygraphs}

A \emph{$1$-polygraph} is a graph~$\Sigma$, i.e.\ a diagram of sets and maps
\[
\xymatrix{
\Sigma_0 
& 
\Sigma_1.
\ar@<+0.5ex>[l]^{t_0}
\ar@<-0.5ex>[l]_{s_0}
}
\]
The elements of~$\Sigma_0$ and~$\Sigma_1$ are called the \emph{$0$-cells} and the \emph{$1$-cells} of~$\Sigma$, respectively. If there is no confusion, we just write $\Sigma=(\Sigma_0,\Sigma_1)$. A $1$-polygraph is \emph{finite} if it has finitely many $0$-cells and $1$-cells.

\subsubsection{Free categories}

If~$\Sigma$ is a $1$-polygraph, the \emph{free category over~$\Sigma$} is the category denoted by~$\Sigma^*$ and defined as follows:
\begin{itemize}
\item the $0$-cells of~$\Sigma^*$ are the ones of~$\Sigma$,
\item the $1$-cells of~$\Sigma^*$ from~$x$ to~$y$ are the finite paths of~$\Sigma$, i.e.\ the finite sequences
\[
x \ofll{u_1} x_1 \ofll{u_2} x_2 \ofll{u_3} \cdots \ofll{u_{n-1}} x_{n-1} \ofll{u_n} y
\]
of $1$-cells of~$\Sigma$,
\item the composition is given by concatenation,
\item the identities are the empty paths.
\end{itemize}
If~$\Sigma$ has only one $0$-cell, then the $1$-cells of the free category~$\Sigma^*$ form the free monoid over the set~$\Sigma_1$.

\subsubsection{Generating \pdf{1}-polygraph}

Let~$\C$ be a category. A $1$-polygraph~$\Sigma$ \emph{generates~$\C$} if there exists a functor
\[
\pi \::\: \Sigma^* \:\pfl\: \C
\]
that is the identity on $0$-cells and such that, for all $0$-cells~$x$ and~$y$ of~$\C$, the map 
\[
\pi_{x,y} \::\: \Sigma^*(x,y) \:\pfl\: \C(x,y)
\]
is surjective. We usually consider that~$\pi$ is implicitly specified for a given generating $1$-polygraph~$\Sigma$ and, if~$u$ is a $1$-cell of~$\Sigma^*$, we just write~$\cl{u}$ instead of~$\pi(u)$. A category is \emph{finitely generated} if it admits a finite generating $1$-polygraph (in particular, the category must have finitely many $0$-cells).

\subsubsection{The word problem for categories}

Let~$\C$ be a category. The \emph{word problem for~$\C$} is the problem of finding a generating $1$-polygraph~$\Sigma$ for~$\C$ together with an algorithm that decides, for any two $1$-cells~$u$ and~$v$ of~$\Sigma^*$, whether or not~$\cl{u}=\cl{v}$ holds in~$\C$ (that is, whether or not the $1$-cells~$u$ and~$v$ represent the same $1$-cell of~$\C$). We note that, to have~$\cl{u}=\cl{v}$, it is necessary for~$u$ and~$v$ to be \emph{parallel}, i.e.\ they must have the same source and the same target. The word problem is undecidable in general for a given category~$\C$, even if it is finitely generated. However, a finite convergent presentation of~$\C$, see~\ref{Section:NormalFormAlgorithm}, provides a solution to the word problem for~$\C$.

\subsection{Presentations of categories}

\subsubsection{Spheres and cellular extensions of categories}

A \emph{$1$-sphere} of a category~$\C$ is a pair $\gamma=(u,v)$ of parallel $1$-cells of~$\C$, that is, with the same source and the same target; such a $1$-sphere is denoted by $\gamma:u\dfl v$. The $1$-cell~$u$ is the \emph{source of~$\gamma$} and~$v$ is its \emph{target}. A \emph{cellular extension} of~$\C$ is a set~$\Gamma$ equipped with a map from~$\Gamma$ to the set of $1$-spheres of~$\C$.

\subsubsection{Congruences}

A \emph{congruence} on a category~$\C$ is an equivalence relation~$\equiv$ on the parallel $1$-cells of~$\C$ that is compatible with the composition of~$\C$, that is, for all $1$-cells 
\[
\xymatrix@C=4em{
x
   \ar[r]^{w} 
& y
   \ar@/^3ex/ [r]^{u}
   \ar@/_3ex/ [r]_{v}
& z
   \ar[r]^{w'}
& t
}
\]
of~$\C$ such that~$u\equiv v$, we have $wuw' \equiv wvw'$. If~$\Gamma$ is a cellular extension of~$\C$, the \emph{congruence generated by~$\Gamma$} is denoted by~$\equiv_{\Gamma}$ and defined as the smallest congruence relation such that, if $\gamma : u\dfl v$ is in~$\Gamma$, then $u\equiv_{\Gamma} v$. The literature also calls~$\equiv_{\Gamma}$ the \emph{Thue congruence} generated by~$\Gamma$.

\subsubsection{Quotient categories}

If~$\C$ is a category and~$\Gamma$ is a cellular extension of~$\C$, the \emph{quotient of~$\C$ by~$\Gamma$} is the category denoted by~$\C/\Gamma$ and defined as follows:
\begin{itemize}
\item the $0$-cells of~$\C/\Gamma$ are the ones of~$\C$,
\item for all $0$-cells~$x$ and~$y$ of~$\C$, the hom-set~$\C/\Gamma(x,y)$ is the quotient of~$\C(x,y)$ by the restriction of~$\equiv_{\Gamma}$.
\end{itemize}
We denote by $\pi_{\Gamma}:\C\pfl\C/\Gamma$ the canonical projection. When the context is clear, we just write~$\pi$ for~$\pi_{\Gamma}$ and~$\cl{u}$ for the image through~$\pi$ of a $1$-cell~$u$ of~$\C$.

\subsubsection{\pdf{2}-polygraphs}

A \emph{$2$-polygraph} is a triple $\Sigma=(\Sigma_0,\Sigma_1,\Sigma_2)$ made of a $1$-polygraph~$(\Sigma_0,\Sigma_1)$, often simply denoted by~$\Sigma_1$, and a cellular extension~$\Sigma_2$ of the free category~$\Sigma_1^*$. In other terms, a $2$-polygraph~$\Sigma$ is a $2$-graph
\[
\xymatrix{
\Sigma_0 
& 
\Sigma_1^\ast
\ar@<+0.5ex>[l]^{t_0}
\ar@<-0.5ex>[l]_{s_0}
&
\Sigma_2
\ar@<+0.5ex>[l]^{t_1}
\ar@<-0.5ex>[l]_{s_1}
}
\]
whose $0$-cells and $1$-cells form a free category. The elements of~$\Sigma_k$ are called the \emph{$k$-cells of~$\Sigma$} and~$\Sigma$ is \emph{finite} if it has finitely many cells in every dimension.

\subsubsection{Presentations of categories}

If~$\Sigma$ is a $2$-polygraph, the \emph{category presented by~$\Sigma$} is the category denoted by~$\cl{\Sigma}$ and defined by
\[
\cl{\Sigma} \:=\: \Sigma_1^*/\Sigma_2.
\]
If~$\C$ is a category, a \emph{presentation of~$\C$} is a $2$-polygraph~$\Sigma$ such that~$\C$ is isomorphic to~$\cl{\Sigma}$. In that case, the $1$-cells of~$\Sigma$ are called \emph{the generating $1$-cells of~$\C$}, or \emph{the generators of~$\C$}, and the $2$-cells of~$\Sigma$ are called \emph{the generating $2$-cells of~$\C$}, or \emph{the relations of~$\C$}. 

\subsubsection{Tietze equivalence} 

Two $2$-polygraphs are \emph{Tietze-equivalent} if they present the same category. It is a standard result that two (finite) $2$-polygraphs are Tietze-equivalent if, and only if, they are related by a (finite) sequence of \emph{elementary Tietze transformations}. On a $2$-polygraph~$\Sigma$, the elementary Tietze transformations are the following operations:
\begin{itemize}
\item adjunction or elimination of a generating $1$-cell~$x$ and of a generating $2$-cell $\alpha:u\dfl x$, where~$u$ is a $1$-cell of the free category over $\Sigma_1\setminus\ens{x}$,
\item adjunction or elimination of a generating $2$-cell $\alpha:u\dfl v$ such that $u\equiv_{\Sigma_2\setminus\ens{\alpha}}v$.
\end{itemize}

\subsection{Two-dimensional categories}

\subsubsection{\pdf{2}-categories}

A \emph{$2$-category} is a data~$\Cr$ made of a set~$\Cr_0$, whose elements are called the \emph{$0$-cells of~$\Cr$}, and, for all $0$-cells~$x$ and~$y$ of~$\Cr$, a category~$\Cr(x,y)$, whose $0$-cells and $1$-cells are respectively called the \emph{$1$-cells} and the \emph{$2$-cells from~$x$ to~$y$ of~$\Cr$}. This data is equipped with the following algebraic structure:
\begin{itemize}
\item for all $0$-cells~$x$, $y$ and~$z$ of~$\Cr$, a functor
\[
\gamma_{x,y,z} \::\: \Cr(x,y)\times\Cr(y,z) \:\fl\: \Cr(x,z),
\]
\item for every $0$-cell~$x$ of~$\Cr$, a specified $0$-cell~$1_x$ of the category~$\Cr(x,x)$.
\end{itemize}
The following relations are required to hold:
\begin{itemize}
\item the composition is associative, i.e.\ for all $0$-cells~$x$, $y$, $z$ and~$t$ of~$\Cr$, 
\[
\gamma_{x,z,t} \circ (\gamma_{x,y,z} \times \id_{\Cr(z,t)}) \:=\: \gamma_{x,y,t} \circ (\id_{\Cr(x,y)}\times \gamma_{y,z,t}),
\]
\item the identities are local units for the composition, i.e.\ for all $0$-cells~$x$ and~$y$ of~$\Cr$,
\[
\gamma_{x,x,y}\circ (1_x \times \id_{\Cr(x,y)}) \:=\: \id_{\Cr(x,y)} \:=\: \gamma_{x,y,y} \circ (\id_{\Cr(x,y)} \times 1_y).
\]
\end{itemize}
This definition of $2$-categories is usually stated as follows: a $2$-category is a category enriched in categories. A \emph{$(2,1)$-category} is a $2$-category whose $2$-cells are invertible for the $1$-composition: in other terms, it is a $2$-category whose hom-categories are groupoids.

\subsubsection{The globular point of view}

A $2$-category can, equivalently, be defined as a $2$-graph
\[
\xymatrix{
\Cr_0 
& 
\Cr_1
\ar@<+0.5ex>[l]^{t_0}
\ar@<-0.5ex>[l]_{s_0}
& 
\Cr_2
\ar@<+0.5ex>[l]^{t_1}
\ar@<-0.5ex>[l]_{s_1}
}
\]
equipped with an additional algebraic structure. The definition of $2$-graph requires that the source and target maps satisfy the globular relations:
\[
s_0\circ s_1 \:=\: s_0\circ t_1 
\qquad\text{and}\qquad 
t_0\circ s_1 \:=\: t_0\circ t_1.
\]
The $2$-graph is equipped with two compositions, the \emph{$0$-composition}~$\star_0$ and the \emph{$1$-composition}~$\star_1$, respectively defined on $0$-composable $1$-cells and $2$-cells, and on $1$-composable $2$-cells. We also have an inclusion of~$\Cr_0$ into~$\Cr_1$ given by the identities of the $2$-category, and an inclusion of~$\Cr_1$ into~$\Cr_2$ induced by the identities of the hom-categories. In details, we have the following operations:
\begin{itemize}
\item for all $1$-cells $x\ofll{u} y\ofll{v} z$, a $0$-composite $1$-cell $x \ofll{u\star_0 v} z$,
\item for all $2$-cells  
$
\xymatrix@C=4em{
x
    \ar@/^4ex/ [r] ^-{u} _{}="src"
    \ar@/_4ex/ [r] _-{u'} ^{}="tgt"
     \ar@2 "src"!<0pt,-10pt>;"tgt"!<0pt,10pt> ^-{f}
& y
    \ar@/^4ex/ [r] ^-{v} _{}="src"
    \ar@/_4ex/ [r] _-{v'} ^{}="tgt"
     \ar@2 "src"!<0pt,-10pt>;"tgt"!<0pt,10pt> ^-{g}
& z
}
$, 
a $0$-composite $2$-cell
$
\xymatrix@C=5em{
x
    \ar@/^4ex/ [r] ^-{u\star_0 v} _{}="src"
     \ar@/_4ex/ [r] _-{u'\star_0 v'} ^{}="tgt"
     \ar@2 "src"!<-10pt,-10pt>;"tgt"!<-10pt,10pt> ^-{f\star_0 g}
& z
}
$, 
\item for all $2$-cells 
$
\xymatrix@C=4em{
x
     \ar@/^5ex/ [r] ^-{u} _{}="src1"
    \ar[r] |(.3){v} ^{}="tgt1" _{}="src2"
    \ar@/_5ex/ [r] _-{w} ^{}="tgt2"
     \ar@2 "src1"!<0pt,-5pt>;"tgt1"!<0pt,5pt> ^-{f}
     \ar@2 "src2"!<0pt,-5pt>;"tgt2"!<0pt,5pt> ^-{g}
& y
}
$, 
a $1$-composite $2$-cell
$
\xymatrix@C=5em{
x
    \ar@/^4ex/ [r] ^-{u} _{}="src"
     \ar@/_4ex/ [r] _-{w} ^{}="tgt"
     \ar@2 "src"!<-10pt,-10pt>;"tgt"!<-10pt,10pt> ^-{f\star_1 g}
& y
}
$,
\item for every $0$-cell~$x$, an identity $1$-cell~$x \ofll{1_x} x$,
\item for every $1$-cell~$x \ofll{u} y$, an identity $2$-cell~$u \odfll{1_u} u$.
\end{itemize}
The following relations hold:
\begin{itemize}
\item for all $1$-cells $x \ofll{u} y \ofll{v} z \ofll{w} t$, \: $(u\star_0 v) \star_0 w = u \star_0 ( v\star_0 w)$,
\item for every $1$-cell $x \ofll{u} y$, \: $1_x \star_0 u = u = u \star_0 1_y$,
\item for all $1$-cells $x \ofll{u} y \ofll{v} z$, \: $1_{u \star_0 v} =1_u\star_0 1_v$, 
\item for all $2$-cells $u \odfll{f} v \odfll{g} w \odfll{h} x$, \: $(f\star_1 g) \star_1 h = f \star_1 (g \star_1 h)$,
\item for all $2$-cells 
$
\xymatrix@C=4em{
x
    \ar@/^4ex/ [r] ^-{u} _{}="src"
    \ar@/_4ex/ [r] _-{u'} ^{}="tgt"
     \ar@2 "src"!<0pt,-10pt>;"tgt"!<0pt,10pt> ^-{f}
& y
    \ar@/^4ex/ [r] ^-{v} _{}="src"
    \ar@/_4ex/ [r] _-{v'} ^{}="tgt"
     \ar@2 "src"!<0pt,-10pt>;"tgt"!<0pt,10pt> ^-{g}
& z
	\ar@/^4ex/ [r] ^-{w} _{}="src"
    \ar@/_4ex/ [r] _-{w'} ^{}="tgt"
     \ar@2 "src"!<0pt,-10pt>;"tgt"!<0pt,10pt> ^-{h}
& t
}
$, \:
$(f\star_0 g) \star_0 h = f\star_0 (g\star_0 h)$,
\item for every $2$-cell
$
\xymatrix@C=4em{
x
    \ar@/^4ex/ [r] ^-{u} _{}="src"
     \ar@/_4ex/ [r] _-{v} ^{}="tgt"
     \ar@2 "src"!<0pt,-10pt>;"tgt"!<0pt,10pt> ^-{f}
&y
}
$, \:
$1_x\star_0 f = f = f\star_0 1_y$,
\item for every $2$-cell $u \odfll{f} v$, \: $1_u\star_1 f = f = f \star_1 1_v$,
\item for all $2$-cells 
$
\xymatrix{
x
    \ar@/^5ex/ [rr] ^-{u} _{}="src1"
    \ar[rr] |(.25){u'} ^{}="tgt1" _{}="src2"
    \ar@/_5ex/ [rr] _-{u''} ^{}="tgt2"
     \ar@2 "src1"!<0pt,-5pt>;"tgt1"!<0pt,-5pt> ^-{f}
     \ar@2 "src2"!<0pt,-5pt>;"tgt2"!<0pt,5pt> ^-{f'}
&& y
    \ar@/^5ex/ [rr] ^-{v} _{}="src3"
    \ar[rr] |(.25){v'} ^{}="tgt3" _{}="src4"
    \ar@/_5ex/ [rr] _-{v''} ^{}="tgt4"
     \ar@2 "src3"!<0pt,-5pt>;"tgt3"!<0pt,-5pt> ^-{g}
     \ar@2 "src4"!<0pt,-5pt>;"tgt4"!<0pt,5pt> ^-{g'}
&& z
}
$, \: 
$(f\star_1 f')\star_0(g\star_1 g') = (f\star_0 g)\star_1 (f'\star_0 g')$.
\end{itemize}
The last relation is usually called the \emph{exchange relation} or the \emph{interchange law} for the compositions~$\star_0$ and~$\star_1$. This globular definition of $2$-categories is equivalent to the enriched one. In particular, the $0$-composition of $2$-cells with identity $2$-cells defines the \emph{whiskering} operations
\begin{itemize}
\item for all cells
$
\xymatrix@C=2em{
x \ar[r]^{w} &
y
    \ar@/^4ex/ [rr] ^-{u} _{}="src"
     \ar@/_4ex/ [rr] _-{v} ^{}="tgt"
     \ar@2 "src"!<0pt,-10pt>;"tgt"!<0pt,10pt> ^-{f}
&&
z
}
$, \:
the \emph{left whiskering}
$
\xymatrix@C=5em{
x
    \ar@/^4ex/ [r] ^-{w\star_0 u} _{}="src"
     \ar@/_4ex/ [r] _-{w\star_0 v} ^{}="tgt"
     \ar@2 "src"!<-11pt,-10pt>;"tgt"!<-11pt,10pt> ^-{w\star_0 f}
& z
}
$,

\item for all cells
$
\xymatrix@C=2em{
x
    \ar@/^4ex/ [rr] ^-{u} _{}="src"
     \ar@/_4ex/ [rr] _-{v} ^{}="tgt"
     \ar@2 "src"!<0pt,-10pt>;"tgt"!<0pt,10pt> ^-{f}
&& y
\ar[r]^{w}
& z
}
$, \:
the \emph{right whiskering}
$
\xymatrix@C=5em{
x
    \ar@/^4ex/ [r] ^-{u\star_0 w} ^{}="src"
     \ar@/_4ex/ [r] _-{v\star_0 w} ^{}="tgt"
     \ar@2 "src"!<-11pt,-10pt>;"tgt"!<-11pt,10pt> ^-{f\star_0 w}
& z
}
$,
\end{itemize}
that satisfy the following relations, implied by the exchange and associativity relations:
\begin{itemize}
\item for all cells
$
\xymatrix@C=2em{
x \ar[r]^u 
& y
    \ar@/^5ex/ [rr] ^-{v} _{}="src1"
    \ar[rr] |(.25){v'} ^{}="tgt1" _{}="src2"
    \ar@/_5ex/ [rr] _-{v''} ^{}="tgt2"
     \ar@2 "src1"!<0pt,-5pt>;"tgt1"!<0pt,5pt> ^-{f}
     \ar@2 "src2"!<0pt,-5pt>;"tgt2"!<0pt,5pt> ^-{f'}
&& z
}
$, \:
$u \star_0 (f\star_1 f') = (u\star_0 f) \star_1 (u\star_0 f')$,
\item for all cells
$
\xymatrix@C=2em{
x
    \ar@/^5ex/ [rr] ^-{u} _{}="src1"
    \ar[rr] |(.25){u'} ^{}="tgt1" _{}="src2"
    \ar@/_5ex/ [rr] _-{u''} ^{}="tgt2"
     \ar@2 "src1"!<0pt,-5pt>;"tgt1"!<0pt,5pt> ^-{f}
     \ar@2 "src2"!<0pt,-5pt>;"tgt2"!<0pt,5pt> ^-{f'}
&& y
\ar[r]^{v}
& z
}
$, \:
$(f\star_1 f') \star_0 v = (f\star_0 v) \star_1 (f'\star_0 v)$,
\item for all cells 
$
\xymatrix@C=2em{
x \ar[r]^u & y  \ar[r]^v 
& z
    \ar@/^4ex/ [rr] ^-{w} _{}="src"
    \ar@/_4ex/ [rr] _-{w'} ^{}="tgt"
     \ar@2 "src"!<0pt,-10pt>;"tgt"!<0pt,10pt> ^-{f}
&& t
}
$, \:
$(u\star_0 v) \star_0 f = u \star_0 (v \star_0 f)$,
\item for all cells 
$
\xymatrix@C=2em{
x
  \ar[r]^u  
& y
    \ar@/^4ex/ [rr] ^-{v} _{}="src"
    \ar@/_4ex/ [rr] _-{v'} ^{}="tgt"
     \ar@2 "src"!<0pt,-10pt>;"tgt"!<0pt,10pt> ^-{f}
&& z
  \ar[r]^{w} 
& t
}
$, \:
$(u\star_0 f) \star_0 w = u \star_0 (f \star_0 w)$,
\item for all cells 
$
\xymatrix@C=2em{
x
    \ar@/^4ex/ [rr] ^-{u} _{}="src"
    \ar@/_4ex/ [rr] _-{u'} ^{}="tgt"
     \ar@2 "src"!<0pt,-10pt>;"tgt"!<0pt,10pt> ^-{f}
&& y 
  \ar[r]^{v} 
& z
  \ar[r]^{w}
& t
}
$, \:
$(f\star_0 v) \star_0 w = f \star_0 (v \star_0 w)$,
\end{itemize}
As for categories, we usually omit the~$\star_0$ notation. For $2$-cells, we write~$s$ and~$t$ instead of~$s_1$ and~$t_1$.

\subsubsection{Free $2$-categories}

Let~$\Sigma$ be a $2$-polygraph. The \emph{free $2$-category over~$\Sigma$} is denoted by~$\Sigma^*$ and defined as follows:
\begin{itemize}
\item the $0$-cells of~$\Sigma^*$ are the ones of~$\Sigma$,
\item for all $0$-cells~$x$ and~$y$ of~$\Sigma$, the hom-category~$\Sigma^*(x,y)$ is presented by the $2$-polygraph 
\begin{itemize}
\item whose $0$-cells are the $1$-cells from~$x$ to~$y$ of~$\Sigma^*$,
\item whose $1$-cells are the 
\[
\xymatrix@C=4em{
x
   \ar[r]^{w} 
& y
   \ar@/^3ex/ [r]^{u} _{}="src"
   \ar@/_3ex/ [r]_{v} ^{}="tgt"
   \ar@2 "src"!<0pt,-10pt>;"tgt"!<0pt,10pt> ^{\alpha}
& z
   \ar[r]^{w'}
& t
}
\]
with $\alpha:u\dfl v$ in~$\Sigma_2$ and~$w$ and~$w'$ in~$\Sigma_1^*$,
\item with one $2$-cell with source $\alpha w v \star_1 u'w \beta$ and target $uw \beta \star_1 \alpha wv'$, for all $\alpha:u\dfl u'$ and $\beta:v\dfl v'$ in~$\Sigma_2$ and~$w$ in~$\Sigma_1^*$,
\end{itemize}
\item for all $0$-cells~$x$, $y$ and~$z$ of~$\Sigma$ the composition functor is given by the concatenation on $1$-cells and, on $2$-cells, by
\begin{align*}
\big( u_1\alpha_1 u'_1 \star_1 \cdots \star_1 u_m \alpha_m u'_m \big)
	\:\star_0\: \big( v_1\beta_1 v'_1 \star_1 \cdots \star_1 v_n \beta_n v'_n \big) 
\quad &{=}
\\
\qquad\qquad u_1\alpha_1 u'_1 v_1 s(\beta_1) v'_1 \star_1 \cdots \star_1 u_m\alpha_m u'_m v_1 s(\beta_1) v'_1 
\quad &{\star_1}
\\
\qquad\qquad  
	u_m t(\alpha_m) u'_m v_1\beta_1 v'_1 \star_1 \cdots \star_1 u_m t(\alpha_m) u'_m v_n\beta_n v'_n \,,
\end{align*}
\item for every $0$-cell~$x$ of~$\Sigma$, the identity $1$-cell~$1_x$ is the one of~$\Sigma_1^*$.
\end{itemize}
By definition of the $2$-category~$\Sigma^*$, for all $1$-cells~$u$ and~$v$ of~$\Sigma^*$, we have~$\cl{u}=\cl{v}$ in the quotient category~$\cl{\Sigma}$ if, and only if, there exists a zigzag sequence of $2$-cells of~$\Sigma^*$ between them:
\[
\xymatrix @C=1.5em {
u 
	\ar@2 [r] ^-{f_1}
& u_1
& v_1 
	\ar@2 [l] _-{g_1}
	\ar@2 [r] ^-{f_2}
& u_2
& (\cdots)
	\ar@2 [l] 
	\ar@2 [r]
& u_{n-1}
& v_{n-1}
	\ar@2 [l] _-{g_{n-1}}
	\ar@2 [r] ^-{f_n}
& u_n
& v.
	\ar@2 [l] _-{g_n}
}
\]

\subsubsection{Free $(2,1)$-categories}

If~$\Sigma$ is a $2$-polygraph, the \emph{free $(2,1)$-category over~$\Sigma$} is denoted by~$\tck{\Sigma}$ and is defined as the $2$-category whose $0$-cells are the ones of~$\Sigma$ and, for all $0$-cells~$x$ and~$y$, the hom-category~$\tck{\Sigma}(x,y)$ is given as the quotient
\[
\tck{\Sigma}(x,y) \:=\: \big(\Sigma\amalg\Sigma^-)^*(x,y) \big/ \text{Inv}(\Sigma_2),
\]
where:
\begin{itemize}
\item the $2$-polygraph~$\Sigma^-$ is obtained from~$\Sigma$ by reversing its $2$-cells,
\item the cellular extension $\text{Inv}(\Sigma_2)$ contains the following two relations for every $2$-cell~$\alpha$ of~$\Sigma$ and all possible $1$-cells~$u$ and~$v$ of~$\Sigma^*$ such that~$s(u)=x$ and~$t(v)=y$:
\[
u\alpha v \star_1 u\alpha^-v \:\equiv\: 1_{u s(\alpha) v} 
\qquad\text{and}\qquad
u\alpha^-v \star_1 u\alpha v \:\equiv\: 1_{u t(\alpha) v}.
\]
\end{itemize}
By definition of the $(2,1)$-category~$\tck{\Sigma}$, for all $1$-cells~$u$ and~$v$ of~$\Sigma^*$, we have~$\cl{u}=\cl{v}$ in the quotient category~$\cl{\Sigma}$ if, and only if, there exists a $2$-cell $f:u\dfl v$ in the $(2,1)$-category~$\tck{\Sigma}$.

\section{Rewriting properties of \pdf{2}-polygraphs}
\label{Section:Rewriting}

\subsection{Convergent presentations of categories}

Let us fix a $2$-polygraph~$\Sigma$.

\subsubsection{Rewriting and normal forms}

A \emph{rewriting step of~$\Sigma$} is a $2$-cell of the free $2$-category~$\Sigma^*$ with shape
\[
\xymatrix@C=4em{
{x}
	\ar [r] ^-{w} 
& {y}
	\ar@/^3ex/ [r] ^-{u} ^{}="src"
	\ar@/_3ex/ [r] _-{v} ^{}="tgt"
	\ar@2 "src"!<0pt,-10pt>;"tgt"!<0pt,10pt> ^-{\phi}
& {z}
	\ar [r] ^-{w'}
& {t}
}
\]
where $\phi:u\dfl v$ is a $2$-cell of~$\Sigma$ and~$w$ and~$w'$ are $1$-cells of~$\Sigma^*$. A \emph{rewriting sequence of~$\Sigma$} is a finite or infinite sequence 
\[
\xymatrix@C=2.5em{
{u_1}
	\ar@2 [r] ^-{f_1}
& {u_2}
	\ar@2 [r] ^-{f_2}
& (\cdots)
	\ar@2 [r] ^-{f_{n-1}}
& {u_n}
	\ar@2 [r] ^-{f_n}
& (\cdots)
}
\]
of rewriting steps. If~$\Sigma$ has a rewriting sequence from~$u$ to~$v$, we say that \emph{$u$ rewrites into~$v$}. Let us note that every $2$-cell~$f$ of~$\Sigma^*$ decomposes into a finite rewriting sequence of~$\Sigma$, this decomposition being unique up to exchange relations. A $1$-cell~$u$ of~$\Sigma^*$ is a \emph{normal form} if~$\Sigma$ has no rewriting step with source~$u$, and a \emph{normal form of~$u$} is a $1$-cell~$v$ of~$\Sigma^*$ that is a normal form and such that~$u$ rewrites into~$v$.

\subsubsection{Termination}

We say that~$\Sigma$ \emph{terminates} if it has no infinite rewriting sequence. In that case, every $1$-cell of~$\Sigma^*$ has at least one normal form and \emph{noetherian induction} allows definitions and proofs of properties of $1$-cells of~$\Sigma^*$ by induction on the size of the $2$-cells leading to normal forms. A \emph{termination order on~$\Sigma$} is an order relation~$\leq$ on parallel $1$-cells of~$\Sigma^*$ such that the following properties are satisfied:
\begin{itemize}
\item the composition of $1$-cells of~$\Sigma^*$ is strictly monotone in both arguments,
\item every decreasing family $(u_n)_{n\in\Nb}$ of parallel $1$-cells of~$\Sigma^*$ is stationary,
\item for every $2$-cell~$\alpha$ of~$\Sigma$, the strict inequality $s(\alpha)>t(\alpha)$ holds. 
\end{itemize}
As a direct consequence of the definition, if~$\Sigma$ admits a termination order, then~$\Sigma$ terminates. A useful example of termination order is the \emph{left degree-wise lexicographic order} (or \emph{deglex} for short) generated by a given order on the $1$-cells of~$\Sigma$. It is defined by the following strict inequalities, where the~$x_i$s and~$y_j$s are $1$-cells of~$\Sigma$:
\[
x_1\cdots x_p \:<\: y_1\cdots y_q, 
	\qquad \text{if } p<q,
\]
\[
x_1\cdots x_{k-1} x_k \cdots x_p \:<\: x_1\cdots x_{k-1} y_k \cdots y_p, 
	\qquad \text{if } x_k < y_k.
\]
The deglex order is total if, and only if, the original order on $1$-cells of~$\Sigma$ is total.

\subsubsection{Branchings}

A \emph{branching} of~$\Sigma$ is a pair~$(f,g)$ of $2$-cells of~$\Sigma_2^*$ with a common source, as in 
\[
\xymatrix @R=1em @C=3em {
& {v}
\\
{u}
	\ar@2@/^/ [ur] ^-{f}
	\ar@2@/_/ [dr] _-{g}
\\
& {w}
}
\]
The $1$-cell~$u$ is the \emph{source} of this branching and the pair~$(v,w)$ is its \emph{target}. We do not distinguish the branchings~$(f,g)$ and~$(g,f)$. A branching~$(f,g)$ is \emph{local} if~$f$ and~$g$ are rewriting steps. Local branchings belong to one of the following three families:
\begin{itemize}
\item \emph{aspherical} branchings, for a rewriting step $u\odfll{f} v$:
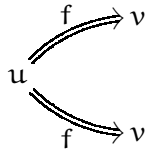
\[
\xymatrix @R=1em @C=3em {
& {v}
\\
{u}
	\ar@2@/^/ [ur] ^-{f}
	\ar@2@/_/ [dr] _-{f}
\\
& {v}
}
\]

\item \emph{Peiffer} branchings, for rewriting steps $u\odfll{f} u'$ and~$v\odfll{g} v'$:
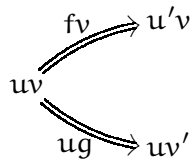
\[
\xymatrix @R=1em @C=3em {
& {u'v}
\\
{uv}
	\ar@2@/^/ [ur] ^-{fv}
	\ar@2@/_/ [dr] _-{ug}
\\
& {uv'}
}
\]

\item \emph{overlapping} branchings are the remaining local branchings.
\end{itemize}
Local branchings are compared by ``inclusion'', i.e.\ by the order~$\preccurlyeq$ generated by the relations
\[
(f,g) \:\preccurlyeq\: \big( u f v, u g v)
\]
given for any local branching~$(f,g)$ and any possible $1$-cells~$u$ and~$v$ of~$\Sigma_1^*$. An overlapping local branching that is minimal for the order~$\preccurlyeq$ is called a \emph{critical branching} (or a \emph{critical pair}). The terms ``aspherical'' and ``Peiffer'' come from the corresponding notions for spherical diagrams in Cayley complexes associated to presentations of groups,~\cite{LyndonSchupp77}, while ``critical'' is used in rewriting theory,~\cite{BookOtto93,BaaderNipkow98}. 

\subsubsection{Confluence}

A branching
\[
\xymatrix @R=1em @C=3em {
& {v}
\\
{u}
	\ar@2@/^/ [ur] ^-{f}
	\ar@2@/_/ [dr] _-{g}
\\
& {w}
}
\]
is \emph{confluent} if there exist $2$-cells~$f'$ and~$g'$ in~$\Sigma_2^*$, as in the following diagram: 
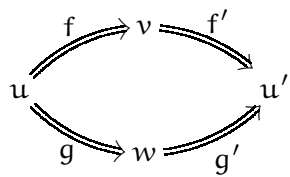
\[
\xymatrix @R=1em @C=3em{
& {v}
	\ar@2@/^/ [dr] ^-{f'}
\\
{u}
	\ar@2@/^/ [ur] ^-{f}
	\ar@2@/_/ [dr] _-{g}
&& {u'}
\\
& {w}
	\ar@2@/_/ [ur] _-{g'}
}
\]
We say that~$\Sigma$ is \emph{confluent} (resp.\ \emph{locally confluent}) if all of its branchings (resp.\ local branchings) are confluent.  If~$\Sigma$ is confluent, every $1$-cell of~$\Sigma^*$ has at most one normal form. 

\begin{lemma}
A $2$-polygraph is locally confluent if, and only if, all its critical branchings are confluent. 
\end{lemma}

\begin{proof}
Every aspherical branching is confluent:
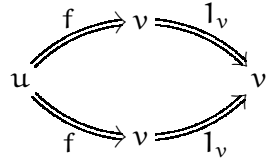
\[
\xymatrix @R=1em @C=3em{
& {v}
	\ar@2@/^/ [dr] ^-{1_v}
\\
{u}
	\ar@2@/^/ [ur] ^-{f}
	\ar@2@/_/ [dr] _-{f}
&& {v}
\\
& {v}
	\ar@2@/_/ [ur] _-{1_v}
}
\]
We also have confluence of every Peiffer local branching:
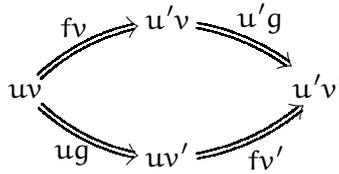
\[
\xymatrix @R=1em @C=3em{
& {u'v}
	\ar@2@/^/ [dr] ^-{u'g}
\\
{uv}
	\ar@2@/^/ [ur] ^-{fv}
	\ar@2@/_/ [dr] _-{ug}
&& {u'v'}
\\
& {uv'}
	\ar@2@/_/ [ur] _-{fv'}
}
\]
We note that, in the aspherical and Peiffer cases, the $2$-cells~$f'$ and~$g'$ can be chosen in such a way that $f\star_1 f'=g\star_1 g'$ holds. Finally, in the case of an overlapping but not minimal local branching~$(f,g)$, there exist  factorisations~$f=uhv$ and~$g=ukv$ with 
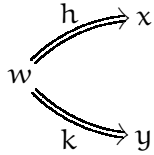
\[
\xymatrix @R=1em @C=3em {
& x
\\
w
	\ar@2@/^/ [ur] ^-{h}
	\ar@2@/_/ [dr] _-{k}
\\
& y
}
\]
a critical branching of~$\Sigma$. Moreover, if the branching~$(h,k)$ is confluent, then so is~$(f,g)$:
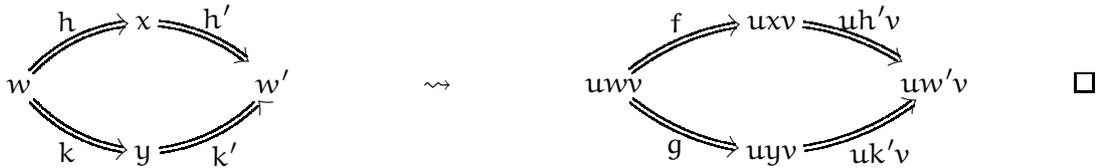
\[
\vcenter{\xymatrix @R=1em @C=3em{
& {x}
	\ar@2@/^/ [dr] ^-{h'}
\\
{w}
	\ar@2@/^/ [ur] ^-{h}
	\ar@2@/_/ [dr] _-{k}
&& {w'}
\\
& {y}
	\ar@2@/_/ [ur] _-{k'}
}}
\qquad\qquad\leadsto\qquad\qquad
\vcenter{\xymatrix @R=1em @C=3em{
& {uxv}
	\ar@2@/^/ [dr] ^-{uh'v}
\\
{uwv}
	\ar@2@/^/ [ur] ^-{f}
	\ar@2@/_/ [dr] _-{g}
&& {uw'v}
\\
& {uyv}
	\ar@2@/_/ [ur] _-{uk'v}
}}
\qedhere
\]
\end{proof}

The following result, also called the diamond lemma, is implied by Theorem~\ref{Theorem:Squier1}.

\begin{theorem}[Newman's lemma~{\cite[Theorem~3]{Newman42}}]
For terminating $2$-polygraphs, local confluence and confluence are equivalent properties. 
\end{theorem}

\subsubsection{Convergent polygraphs}
\label{Section:NormalFormAlgorithm}

We say that~$\Sigma$ is \emph{convergent} if it terminates and it is confluent. Such a~$\Sigma$ is called a \emph{convergent presentation of~$\cl{\Sigma}$}, and of any category that is isomorphic to~$\cl{\Sigma}$. In that case, every $1$-cell~$u$ of~$\Sigma_1^*$ has a unique normal form, denoted by~$\rep{u}$, so that we have~$\cl{u}=\cl{v}$ in~$\cl{\Sigma}$ if, and only if,~$\widehat{u}=\widehat{v}$ holds in~$\Sigma_1^*$. This defines a section $\cl{\Sigma}\ifl\Sigma_1^*$ of the canonical projection $\Sigma_1^\ast \pfl \cl{\Sigma}$, mapping a $1$-cell~$u$ of~$\cl{\Sigma}$ to the unique normal form of its representative $1$-cells in~$\Sigma^*$, still denoted by~$\rep{u}$. 

As a consequence, a finite and convergent $2$-polygraph~$\Sigma$ yields a decision procedure for the word problem of the category~$\cl{\Sigma}$ it presents: the \emph{normal-form procedure}, which takes, as input, two $1$-cells~$u$ and~$v$ of~$\Sigma^*$, and decides whether~$\cl{u}=\cl{v}$ holds in~$\cl{\Sigma}$ or not. For that, the procedure computes the respective normal forms~$\rep{u}$ and~$\rep{v}$ of~$u$ and~$v$. Finiteness is used to test whether a given $1$-cell~$u$ is a normal form or not, by examination of all the relations and their possible applications on~$u$. Then, the equality~$\cl{u}=\cl{v}$ holds in~$\cl{\Sigma}$ if, and only if, the equality~$\rep{u}=\rep{v}$ holds in~$\Sigma^*$.

\subsection{Transformations of \pdf{2}-polygraphs}

\subsubsection{Knuth-Bendix completion}

Let~$\Sigma$ be a terminating $2$-polygraph, equipped with a total termination order~$\leq$. A \emph{Knuth-Bendix completion of~$\Sigma$} is a $2$-polygraph~$\check{\Sigma}$ obtained by the following process. We start with~$\check{\Sigma}$ equal to~$\Sigma$ and with~$\Br$ equal to the set of critical branchings of~$\Sigma$. If~$\Br$ is empty, then the procedure stops. Otherwise, it picks a branching 
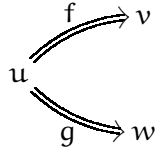
\[
\xymatrix @C=3em @R=1em {
& {v}
\\
{u}
	\ar@2@/^/ [ur] ^-{f}
	\ar@2@/_/ [dr] _-{g}
\\
& {w}
}
\]
in~$\Br$ and it performs the following operations:
\begin{enumerate}
\item It computes $2$-cells $f':v\dfl\rep{v}$ and $g':w\dfl\rep{w}$ of~$\check{\Sigma}^*$, where~$\rep{v}$ and~$\rep{w}$ are normal forms for~$v$ and~$w$, respectively, as in the following diagram:
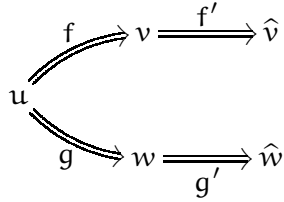
\[
\xymatrix @C=3em @R=1em {
& {v}
	\ar@2 [r] ^-{f'}
& {\rep{v}}
\\
{u}
	\ar@2@/^/ [ur] ^-{f}
	\ar@2@/_/ [dr] _-{g}
\\
& {w}
	\ar@2 [r] _-{g'}
& {\rep{w}}
}
\]
\item It tests which (in)equality~$\rep{v}=\rep{w}$ or~$\rep{v}>\rep{w}$ or~$\rep{v}<\rep{w}$ holds, corresponding to the following three situations, respectively:
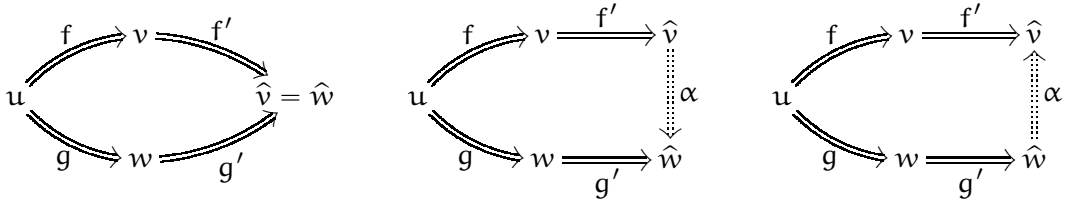
\[
\xymatrix @C=3em @R=1em {
& {v}
	\ar@2@/^/ [dr] ^-{f'}
\\
{u}
	\ar@2@/^/ [ur] ^-{f}
	\ar@2@/_/ [dr] _-{g}
&& {\rep{v}=\rep{w}}
\\
& {w}
	\ar@2@/_/ [ur] _-{g'}
}
\qquad
\xymatrix @C=3em @R=1em {
& {v}
	\ar@2 [r] ^-{f'}
& {\rep{v}}
	\ar@2{.>} [dd] ^-{\alpha}
\\
{u}
	\ar@2@/^/ [ur] ^-{f}
	\ar@2@/_/ [dr] _-{g}
\\
& {w}
	\ar@2 [r] _-{g'}
& {\rep{w}}	
}
\qquad
\xymatrix @C=3em @R=1em {
& {v}
	\ar@2 [r] ^-{f'}
& {\rep{v}}
\\
{u}
	\ar@2@/^/ [ur] ^-{f}
	\ar@2@/_/ [dr] _-{g}
\\
& {w}
	\ar@2 [r] _-{g'}
& {\rep{w}}
	\ar@2{.>} [uu] _-{\alpha}
}
\]
If~$\rep{v}\neq\rep{w}$, the procedure adds the dotted $2$-cell~$\alpha$ of the corresponding situation to~$\check{\Sigma}$ and all the new critical branchings created by~$\alpha$ to~$\Br$.

\item It removes~$(f,g)$ from~$\Br$ and restarts from the beginning.
\end{enumerate}
If the procedure stops, it returns the $2$-polygraph~$\check{\Sigma}$. Otherwise, it builds an increasing sequence of $2$-polygraphs, whose limit is denoted by~$\check{\Sigma}$. Note that the resulting $2$-polygraph may depend on the order of examination of the critical branchings. Also, if the starting $2$-polygraph~$\Sigma$ is already convergent, then the Knuth-Bendix completion of~$\Sigma$ is~$\Sigma$. 
By construction, the $2$-polygraph~$\check{\Sigma}$ is convergent and, since all the operations performed by the procedure are Tietze transformations, it is Tietze-equivalent to~$\Sigma$:

\begin{theorem}[\cite{KnuthBendix70}]
Any Knuth-Bendix completion~$\check{\Sigma}$ of a $2$-polygraph~$\Sigma$, equipped with a total termination order, is a convergent presentation of the category~$\cl{\Sigma}$. Moreover, the $2$-polygraph~$\check{\Sigma}$ is finite if, and only if, the $2$-polygraph~$\Sigma$ is finite and the Knuth-Bendix completion procedure halts.
\end{theorem}

\subsubsection{Métivier-Squier reduction}

A $2$-polygraph~$\Sigma$ is \emph{reduced} if, for every $2$-cell $\alpha:u\dfl v$ of~$\Sigma$, we have that~$u$ is a normal form for $\Sigma_2\setminus\ens{\alpha}$ and that~$v$ is a normal form for~$\Sigma_2$. Given a convergent $2$-polygraph~$\Sigma$, the \emph{Métivier-Squier reduction of~$\Sigma$} is the $2$-polygraph obtained by the procedure that successively performs the following operations:
\begin{enumerate}
\item The procedure replaces every generating $2$-cell $\alpha:u\dfl v$ by $\alpha:u\dfl\rep{u}$:
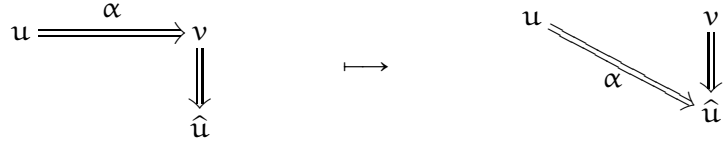
\[
\vcenter{\xymatrix{
{u} 
	\ar@2 [rr] ^-*+{\alpha} 
&& {v}
	\ar@2 [d]
\\
&& {\rep{u}}
}}
\qquad\qquad\longmapsto\qquad\qquad
\vcenter{\xymatrix{
{u} 
	\ar@2 [drr] _-{\alpha} 
&& {v}
	\ar@2 [d]
\\
&& {\rep{u}}
}}
\]

\item Next, if the resulting $2$-polygraph contains parallel generating $2$-cells, the procedure removes all but one:
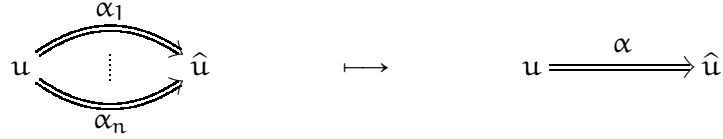
\[
\xymatrix{
{u} 
	\ar@2 @/^3ex/ [rr] ^-{\alpha_1} ^{}="src"
	\ar@2 @/_3ex/ [rr] _-{\alpha_n} _{}="tgt"
&& {\rep{u}}
\ar@{.} "src"!<0pt,-10pt>;"tgt"!<0pt,10pt>
}
\qquad\qquad\longmapsto\qquad\qquad
\xymatrix{
{u}
	\ar@2 [rr] ^-*+{\alpha}
&& {\rep{u}}
}
\]

\item Finally, the procedure removes every generating $2$-cell~$\alpha$ with source~$vs(\beta)v'$, where~$\beta$ is another generating $2$-cell:
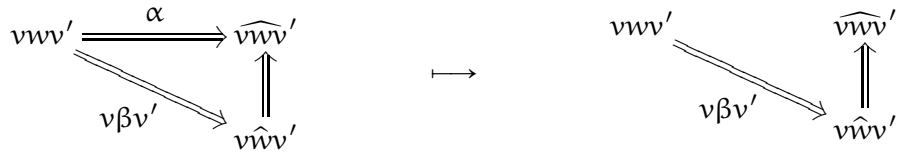
\[
\vcenter{\xymatrix{
{vwv'}
	\ar@2 [rr] ^-*+{\alpha}
	\ar@2 [drr] _-*+{v\beta v'} 
&& {\rep{vwv}'}
\\
&& {v\rep{w}v'}
	\ar@2 [u] 
}}
\qquad\qquad\longmapsto\qquad\qquad
\vcenter{\xymatrix{
{vwv'}
	\ar@2 [drr] _-*+{v\beta v'} 
&& {\rep{vwv}'}
\\
&& {v\rep{w}v'}
	\ar@2 [u] 
}}
\]
\end{enumerate}

By construction, we get the following result, originally obtained by Métivier for term rewriting and by Squier for string rewriting: 

\begin{theorem}[\cite{Metivier83}, {\cite[Theorem~2.4]{Squier87}}]
Every (finite) convergent $2$-polygraph is Tietze-equivalent to a (finite) reduced convergent
$2$-polygraph. 
\end{theorem}

\subsection{Normalisation strategies}
\label{Section:NormalisationStrategies}

\subsubsection{Normalisation strategies}
\label{SS:NormalisationStrategies}

Let~$\Sigma$ be a $2$-polygraph and let~$\C$ denote the category presented by~$\Sigma$. We consider a section $\C \ifl \Sigma_1^\ast$ of the canonical projection $\pi : \Sigma_1^\ast \pfl \C$, i.e.\ we choose, for every $1$-cell~$u$ of~$\C$, a $1$-cell~$\rep{u}$ of~$\Sigma^*$ such that $\pi(\rep{u})=u$. In general, we cannot assume that the chosen section is functorial, i.e.\ that $\rep{uv}=\rep{u}\rep{v}$ holds in~$\Sigma^*$. However, we assume that~$\rep{1}_x=1_x$ holds for every $0$-cell~$x$ of~$\C$. Given a $1$-cell~$u$ of~$\Sigma^*$, we simply write~$\rep{u}$ for~$\rep{\cl{u}}$.

Such a section being fixed, a \emph{normalisation strategy for~$\Sigma$} is a map
\[
\sigma : \Sigma_1^\ast \fl \Sigma_2^\ast
\]
that sends every $1$-cell~$u$ of~$\Sigma^*$ to a $2$-cell
\[
u\odfll{\sigma_u}\rep{u}
\]
of~$\tck{\Sigma}$, such that $\sigma_{\rep{u}} = 1_{\rep{u}}$ holds for every $1$-cell~$u$ of~$\Sigma^\ast$. 

\subsubsection{Left and right normalisation strategies}
\label{SS:LeftRight}

Let~$\Sigma$ be a $2$-polygraph, with a chosen section. A normalisation strategy~$\sigma$ for~$\Sigma$ is a \emph{left} one (resp.\ a \emph{right} one) if it also satisfies 
\[
\sigma_{uv} \:=\quad
\vcenter{\xymatrix @C=4em @R=2.5em {
& \cdot  
	\ar@/^/ [dr] ^-{v}
	\ar@2 []!<2.5pt,-20pt>;[d]!<2.5pt,7.5pt> ^-*+{\sigma_{\rep{u}v}}
\\
\cdot
	\ar@/^4ex/ [ur] ^-{u} _-{}="src"
	\ar@/_/ [ur] |-{\rep{u}} ^-{}="tgt"
	\ar@2 "src"!<5pt,-7.5pt>;"tgt"!<-5pt,7.5pt> ^-{\sigma_{u}}
	\ar@/_/ [rr] _-{\rep{uv}}
&& \cdot
}}
\qquad\qquad
\big( \text{resp.\ } \sigma_{uv} \:=\quad
\vcenter{\xymatrix @C=4em @R=2.5em {
& \cdot  
	\ar@/^4ex/ [dr] ^-{v} _-{}="src"
	\ar@/_/ [dr] |-{\rep{v}} ^-{}="tgt"
	\ar@2 "src"!<-5pt,-7.5pt>;"tgt"!<5pt,7.5pt> ^-{\sigma_{v}}
	\ar@2 []!<-10pt,-20pt>;[d]!<-10pt,7.5pt> ^-*+{\sigma_{u\rep{v}}}
\\
\cdot
	\ar@/^/ [ur] ^-{u}
	\ar@/_/ [rr] _-{\rep{uv}}
&& \cdot
}}
\quad\big).
\]
The $2$-polygraph~$\Sigma$ always admits left and right normalisation strategies. For example, in the left case, let us arbitrarily choose a $2$-cell $\sigma_{ua}:ua\dfl\rep{ua}$ in~$\tck{\Sigma}$, for every~$1$-cell~$u$ of~$\Sigma^*$ and every $1$-cell~$a$ of~$\Sigma$, such that~$\rep{u}=u$ and~$ua$ is defined, with $\rep{ua}\neq ua$. Then we extend~$\sigma$ into a left normalisation strategy for~$\Sigma$ by putting $\sigma_u=1_u$ if $\rep{u}=u$ (which implies $\sigma_{1_x}=1_{1_x}$), and
\[
\sigma_{u} = \sigma_{v}a \star_1 \sigma_{\rep{v}a}
\]
if $\rep{u}\neq u$ and $u=va$ with~$v$ in~$\Sigma^*$ and~$a$ in~$\Sigma$.

\subsubsection{Leftmost and rightmost normalisation strategies}
\label{SS:LeftmostRightmost}

If~$\Sigma$ is a reduced $2$-polygraph, then, for every $1$-cell~$u$ of~$\Sigma^*$, the set of rewriting steps with source~$u$ can be ordered from left to right: for two rewriting steps $f=v\alpha v'$ and $g=w\beta w'$ with source~$u$, we have~$f\prec g$ if the length of~$v$ is strictly smaller than the length of~$w$. If~$\Sigma$ is finite, then the order~$\prec$ is total and the set of rewriting steps of source~$u$ is finite. Hence, this set contains a smallest element~$\lambda_u$ and a greatest element~$\rho_u$, respectively called the \emph{leftmost} and the \emph{rightmost rewriting steps on~$u$}. If, moreover, the $2$-polygraph~$\Sigma$ terminates, the iteration of~$\lambda$ (resp.~$\rho$) yields a normalisation strategy~$\sigma$ called the \emph{leftmost} (resp.\ \emph{rightmost}) \emph{normalisation strategy of~$\Sigma$}: 
\[
\sigma_{u} = \lambda_{u} \star_1 \sigma_{t(\lambda_u)}
\qquad
(\text{resp.\ }
\sigma_u = \rho_u \star_1 \sigma_{t(\rho_u)}).
\]
We prove, by noetherian induction, that the leftmost (resp.\ rightmost) normalisation strategy of~$\Sigma$ is a left (resp.\ right) normalisation strategy. The leftmost and rightmost normalisation strategies give a way to make constructive some of the results we present here. For example, they provide a deterministic choice of a confluence diagram
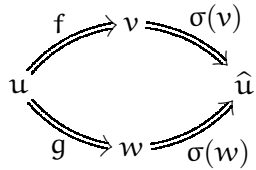
\[
\xymatrix @!C @R=1em @C=2.5em {
& v
	\ar@2@/^/ [dr] ^{\sigma(v)}
\\
u 
	\ar@2@/^/ [ur] ^{f}
	\ar@2@/_/ [dr] _{g}
&& {\rep{u}}
\\
& w
	\ar@2@/_/ [ur] _{\sigma(w)}
}
\]
for every branching~$(f,g)$ of~$\Sigma$.

\section{Finite derivation type}
\label{Section:threeDimensionalPolygraphs}

\subsection{Coherent presentations of categories}

\subsubsection{Cellular extensions and homotopy bases of \pdf{2}-categories}

Let~$\Cr$ be a $2$-category. A \emph{$2$-sphere of~$\Cr$} is a pair $\gamma=(f,g)$ of parallel $2$-cells of~$\Cr$, i.e.\ with~$s(f)=s(g)$ and~$t(f)=t(g)$. We call~$f$ the \emph{source} of~$\gamma$ and~$g$ its \emph{target} and we denote such a $2$-sphere by $\gamma:f\tfl g$. A \emph{cellular extension} of the $2$-category~$\Cr$ is a set~$\Gamma$ equipped with a map from~$\Gamma$ to the set of $2$-spheres of~$\Cr$.

A \emph{congruence} on a $2$-category~$\Cr$ is an equivalence relation~$\equiv$ on the parallel $2$-cells of~$\Cr$ that is compatible with the two compositions of~$\Cr$, that is, for all cells
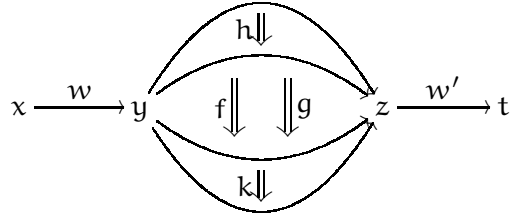
\[
\xymatrix @C=3em @!C {
x
  \ar[r]^{w}
&
y
   \ar @/^8ex/ [rr] _-{}="src1"
   \ar @/^4ex/ [rr] ^{}="tgt1" _{}="src2"
   \ar @/_4ex/ [rr] ^{}="src3" _{}="tgt2"
   \ar @/_8ex/ [rr] _{}="tgt3"
&& z
  \ar[r]^{w'}
&
t
\ar@2 "src1"!<0pt,-4pt>;"tgt1"!<0pt,4pt> _-{h}
\ar@2 "src2"!<-10pt,-9pt>;"tgt2"!<-10pt,9pt> _-{f}
\ar@2 "src2"!<10pt,-9pt>;"tgt2"!<10pt,9pt> ^-{g}
\ar@2 "src3"!<0pt,-4pt>;"tgt3"!<0pt,4pt> _-{k}
}
\]
of~$\Cr$ such that~$f\equiv g$, we have
\[
w\star_0(h\star_1 f \star_1 k)\star_0 w' \;\equiv\; w\star_0(h\star_1 g \star_1 k)\star_0 w'.
\]
If~$\Gamma$ is a cellular extension of~$\Cr$, the \emph{congruence generated by~$\Gamma$} is denoted by~$\equiv_{\Gamma}$ and defined as the smallest congruence such that, if~$\Gamma$ contains a $3$-cell $\gamma:f\tfl g$, then~$f\equiv_{\Gamma}g$. The \emph{quotient $2$-category} of a $2$-category~$\Cr$ by a congruence relation~$\equiv$ is the $2$-category, denoted by~$\Cr/\equiv$, whose $0$-cells and $1$-cells are those of~$\Cr$ and whose $2$-cells are the equivalence classes of $2$-cells of~$\Cr$ modulo the congruence~$\equiv$. 

A \emph{homotopy basis of~$\Cr$} is a cellular extension~$\Gamma$ of~$\Cr$ such that, for every parallel $2$-cells~$f$ and~$g$ of~$\Cr$, we have~$f\equiv_{\Gamma} g$, that is, the equality~$\cl{f}=\cl{g}$ holds in the quotient $2$-category~$\Cr/\equiv_{\Gamma}$. For instance, the set of $2$-spheres of~$\Cr$ forms a homotopy basis.

\subsubsection{\pdf{(3,1)}-polygraphs and coherent presentations}

A \emph{$(3,1)$-polygraph} is a pair $\Sigma=(\Sigma_2,\Sigma_3)$ made of a $2$-polygraph~$\Sigma_2$ and a cellular extension~$\Sigma_3$ of the free $(2,1)$-category~$\tck{\Sigma}_2$ over~$\Sigma_2$, as summarised in
\[
\xymatrix @C=3em{
\Sigma_0
& \Sigma_1^*
	\ar@<.5ex> [l]^{t_0}
	\ar@<-.5ex> [l]_{s_0}
&  \tck{\Sigma}_2
	\ar@<.5ex> [l]^{t_1}
	\ar@<-.5ex> [l]_{s_1}
&  \Sigma_3
	\ar@<.5ex> [l]^{t_2}
	\ar@<-.5ex> [l]_{s_2}
}
\]
If~$\C$ is a category, a \emph{coherent presentation of~$\C$} is a $(3,1)$-polygraph $\Sigma=(\Sigma_2,\Sigma_3)$ such that~$\Sigma_2$ is a presentation of~$\C$ and~$\Sigma_3$ is a homotopy basis of~$\tck{\Sigma}_2$.

\subsubsection{Finite derivation type}

A $2$-polygraph~$\Sigma$ is of \emph{finite derivation type} if it is finite and if the $(2,1)$-category~$\tck{\Sigma}$ admits a finite homotopy basis. A category~$\C$ is of \emph{finite derivation type} if it admits a finite coherent presentation.

\subsubsection{\pdf{3}-categories}

The definition of $3$-category is adapted from the one of $2$-category by replacement of the hom-categories and the composition functors by hom-$2$-categories and composition $2$-functors. In particular, in a $3$-category, the $3$-cells can be composed in three different ways:
\begin{itemize}
\item by~$\star_0$, along their $0$-dimensional boundary:
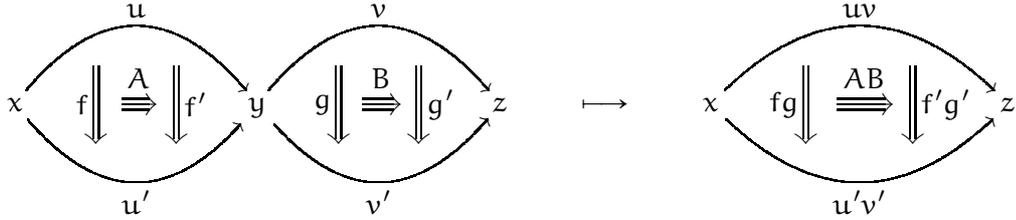
\[
\xymatrix @C=3em @!C{
x
   \ar @/^6ex/ [rr] ^{u} _{}="src1"
   \ar @/_6ex/ [rr] _{u'} ^{}="tgt1"
&& y
\ar@2 "src1"!<-15pt,-15pt>;"tgt1"!<-15pt,15pt> _{f}="srcA"
\ar@2 "src1"!<15pt,-15pt>;"tgt1"!<15pt,15pt> ^{f'}="tgtA"
\ar@3 "srcA"!<15pt,0pt>;"tgtA"!<-15pt,0pt> ^*+{A}
   \ar @/^6ex/ [rr] ^{v} _{}="src2"
   \ar @/_6ex/ [rr] _{v'} ^{}="tgt2"
&& z
\ar@2 "src2"!<-15pt,-15pt>;"tgt2"!<-15pt,15pt> _{g}="srcB"
\ar@2 "src2"!<15pt,-15pt>;"tgt2"!<15pt,15pt> ^{g'}="tgtB"
\ar@3 "srcB"!<15pt,0pt>;"tgtB"!<-15pt,0pt> ^*+{B}
}
\qquad\longmapsto\qquad
\xymatrix @C=4em @!C{
x
   \ar @/^6ex/ [rr] ^{uv} _{}="src1"
   \ar @/_6ex/ [rr] _{u'v'} _{}="tgt1"
&& z
\ar@2 "src1"!<-20pt,-15pt>;"tgt1"!<-20pt,15pt> _{fg}="srcA"
\ar@2 "src1"!<20pt,-15pt>;"tgt1"!<20pt,15pt> ^{f'g'}="tgtA"
\ar@3 "srcA"!<20pt,0pt>;"tgtA"!<-20pt,0pt> ^*+{AB}
}
\]
\item by~$\star_1$, along their $1$-dimensional boundary:
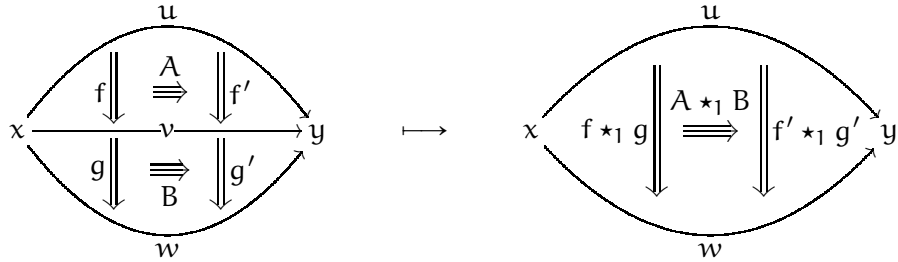
\[
\xymatrix @C=4em @!C{
x
   \ar @/^8ex/ [rr] ^{u} _{}="src1"
   \ar [rr] |{v} ^{}="tgt1" _{}="src2" 
   \ar @/_8ex/ [rr] _{w} _{}="tgt2"
&& y
\ar@2 "src1"!<-20pt,-10pt>;"tgt1"!<-20pt,0pt> _{f}="srcA"
\ar@2 "src1"!<20pt,-10pt>;"tgt1"!<20pt,0pt> ^{f'}="tgtA"
\ar@2 "src2"!<-20pt,0pt>;"tgt2"!<-20pt,10pt> _{g}="srcB"
\ar@2 "src2"!<20pt,0pt>;"tgt2"!<20pt,10pt> ^{g'}="tgtB"
\ar@3 "srcA"!<20pt,0pt>;"tgtA"!<-20pt,0pt> ^*+{A}
\ar@3 "srcB"!<20pt,0pt>;"tgtB"!<-20pt,0pt> _*+{B}
}
\qquad\longmapsto\qquad
\xymatrix @C=5em @!C{
x
   \ar @/^8ex/ [rr] ^{u} _{}="src1"
   \ar @/_8ex/ [rr] _{w} _{}="tgt1"
&& y
\ar@2 "src1"!<-20pt,-15pt>;"tgt1"!<-20pt,15pt> _{f\star_1 g} ^{}="srcA"
\ar@2 "src1"!<20pt,-15pt>;"tgt1"!<20pt,15pt> ^{f'\star_1 g'} _{}="tgtA"
\ar@3 "srcA"!<10pt,0pt>;"tgtA"!<-10pt,0pt> ^*+{A\star_1 B}
}
\]
\item by~$\star_2$, along their $2$-dimensional boundary:
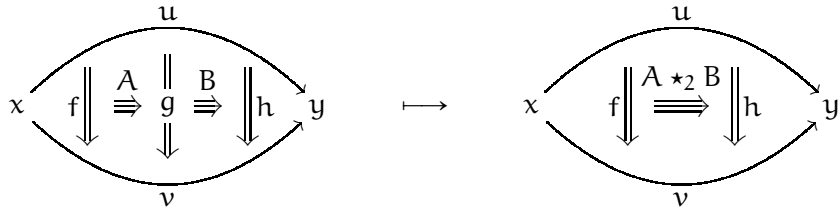
\[
\xymatrix @C=4em @!C{
x
   \ar @/^6ex/ [rr] ^{u} _{}="src"
   \ar @/_6ex/ [rr] _{v} _{}="tgt"
&& y
\ar@2 "src"!<-30pt,-15pt>;"tgt"!<-30pt,15pt> _{f} ^{}="srcA"
\ar@2 "src"!<0pt,-10pt>;"tgt"!<0pt,10pt> |*+{g} _{} ^{}="srcB" _{}="tgtA"
\ar@2 "src"!<30pt,-15pt>;"tgt"!<30pt,15pt> ^{h} _{}="tgtB"
\ar@3 "srcA"!<10pt,0pt>;"tgtA"!<-10pt,0pt> ^*+{A}
\ar@3 "srcB"!<10pt,0pt>;"tgtB"!<-10pt,0pt> ^*+{B}
}
\qquad\longmapsto\qquad
\xymatrix @C=4em @!C{
x
   \ar @/^6ex/ [rr] ^{u} _{}="src1"
   \ar @/_6ex/ [rr] _{v} _{}="tgt1"
&& y
\ar@2 "src1"!<-20pt,-15pt>;"tgt1"!<-20pt,15pt> _{f} ^{}="srcA"
\ar@2 "src1"!<20pt,-15pt>;"tgt1"!<20pt,15pt> ^{h} _{}="tgtA"
\ar@3 "srcA"!<10pt,0pt>;"tgtA"!<-10pt,0pt> ^*+{A\star_2 B}
}
\]
\end{itemize}
A \emph{$(3,1)$-category} is a $3$-category whose $2$-cells are invertible for the composition~$\star_1$ and whose $3$-cells are invertible for the composition~$\star_2$. This implies that $3$-cells are also invertible for the composition~$\star_1$.

\subsubsection{Free \pdf{(3,1)}-categories}

Given a $(3,1)$-polygraph~$\Sigma$, the \emph{free $(3,1)$-category over~$\Sigma$} is denoted by~$\tck{\Sigma}$ and defined as follows:
\begin{itemize}
\item its underlying $2$-category is the free $(2,1)$-category~$\tck{\Sigma}_2$,
\item its $3$-cells are all the formal compositions by~$\star_0$, $\star_1$ and~$\star_2$ of $3$-cells of~$\Sigma$, of their inverses and of identities of $2$-cells, up to associativity, identity, exchange and inverse relations.
\end{itemize}
In particular, we get that~$\Sigma_3$ is a homotopy basis of~$\tck{\Sigma}_2$ if, and only if, for every pair~$(f,g)$ of parallel $2$-cells of~$\tck{\Sigma}_2$, there exists a $3$-cell $A:f\tfl g$ in~$\tck{\Sigma}$.

\subsection{The homotopy basis transfer theorem}

The objective of this section is to prove Theorem~\ref{Theorem:Squier2}: given two finite presentations of the same category, both are of finite derivation type or neither is. Towards this goal, we prove Theorem~\ref{Theorem:HomotopyBasesTransfer}, that allows transfers of homotopy bases between presentations of the same category.

\begin{lemma}
\label{Lemma:TietzeEquivalence}
Let~$\C$ be a category and let~$\Sigma$ and~$\Xi$ be presentations of~$\C$. There exist $2$-functors
\[
F:\tck{\Sigma}\fl\tck{\Xi}
\qquad\text{and}\qquad
G:\tck{\Xi}\fl\tck{\Sigma}
\]
and, for every $1$-cells~$u$ of~$\tck{\Sigma}$ and~$v$ of~$\tck{\Xi}$, there exist $2$-cells
\[
\sigma_u : GF(u) \dfl u
\qquad\text{and}\qquad
\tau_v : FG(v) \dfl v
\]
in~$\tck{\Sigma}$ and~$\tck{\Xi}$, such that the following conditions are satisfied:
\begin{itemize}
\item the $2$-functors~$F$ and~$G$ induce the identity through the canonical projections onto~$\C$:
\[
\xymatrix{
{\tck{\Sigma}}
	\ar@{->>} [r]^{\pi_\Sigma}
	\ar [d] _-{F}
	\ar@{} [dr] |{=}
& \C
	\ar [d] ^-{\id_{\C}} 
\\
{\tck{\Xi}} 
	\ar@{->>} [r]_{\pi_\Xi}
& \C
}
\qquad\qquad
\xymatrix{
{\tck{\Sigma}}
	\ar@{->>} [r]^{\pi_\Sigma}
	\ar@{} [dr] |{=}
& \C
\\
{\tck{\Xi}}
	\ar@{->>} [r]_{\pi_\Xi}
	\ar [u] ^-{G}
& \C
	\ar [u] _-{\id_{\C}}
}
\]
\item the $2$-cells~$\sigma_u$ and~$\tau_v$ are functorial in~$u$ and~$v$:
\[
\sigma_{1_x} = 1_{1_x},
\qquad 
\tau_{1_y} = 1_{1_y},
\qquad
\sigma_{uu'} = \sigma_u\sigma_{u'}
\quad\text{and}\quad
\tau_{vv'} = \tau_v\tau_{v'}.
\]
\end{itemize}
\end{lemma}

\begin{proof}
Let us define~$F$, the case of~$G$ being symmetric. On a $0$-cell~$x$ of~$\Sigma$, we take~$F(x)=x$. If $a:x\fl y$ is a $1$-cell of~$\Sigma$, we choose, in an arbitrary way, a $1$-cell $F(a):x\fl y$ in~$\tck{\Xi}$ such that $\pi_{\Xi}F(a)=\pi_{\Sigma}(a)$. Then, we extend~$F$ to every $1$-cell of~$\tck{\Sigma}$ by functoriality. Let $\alpha:u\dfl u'$ be a $2$-cell of~$\Sigma$. Since~$\Sigma$ is a presentation of~$\C$, we have $\pi_{\Sigma}(u)=\pi_{\Sigma}(u')$, so that $\pi_{\Xi}F(u)=\pi_{\Xi}F(u')$ holds. Using the fact that~$\Xi$ is a presentation of~$\C$, we arbitrarily choose a $2$-cell $F(\alpha):F(u)\dfl F(u')$ in~$\tck{\Xi}$. Then, we extend~$F$ to every $2$-cell of~$\tck{\Sigma}$ by functoriality.

Now, let us define~$\sigma$, the case of~$\tau$ being symmetric. Let~$a$ be a $1$-cell of~$\Sigma$. By construction of~$F$ and~$G$, we have:
\[
\pi_{\Sigma} GF(a) \:=\: \pi_{\Xi} F(a) \:=\: \pi_{\Sigma}(a).
\]
Since~$\Sigma$ is a presentation of~$\C$, there exists a $2$-cell $\sigma_a:GF(a)\dfl a$ in~$\tck{\Sigma}$. We extend~$\sigma$ to every $1$-cell~$u$ of~$\tck{\Sigma}$ by functoriality.
\end{proof}

\begin{theorem}
\label{Theorem:HomotopyBasesTransfer}
Let~$\C$ be a category, let~$\Sigma$ and~$\Xi$ be two presentations of~$\C$ and let~$F$, $G$ and~$\tau$ be chosen as in Lemma~\ref{Lemma:TietzeEquivalence}. If~$\Gamma$ is a homotopy basis of~$\tck{\Sigma}$, then
\[
\Delta \:=\: F(\Gamma) \:\amalg\: \tau_{\Xi}
\] 
is a homotopy basis of~$\tck{\Xi}$, where:
\begin{itemize}
\item the cellular extension~$F(\Gamma)$ contains one $3$-cell
\[
\xymatrix @C=6em{
F(u)
	\ar@2 @/^4ex/ [r] ^-{F(f)} _-{}="src"
	\ar@2 @/_4ex/ [r] _-{F(g)} ^-{}="tgt"
\ar@3 "src"!<0pt,-10pt>;"tgt"!<0pt,10pt> ^-{F(\gamma)}
& F(v)
}
\]
for every $3$-cell $\gamma:f\tfl g$ of~$\Gamma$,
\item the cellular extension~$\tau_{\Xi}$ contains one $3$-cell
\[
\xymatrix @!C @C=2.5em @R=1em{
& FG(v)
	\ar@2 @/^/ [dr] ^-{\tau_v}
\ar@3 []!<0pt,-15pt>;[dd]!<0pt,15pt> ^-{\tau_{\alpha}}
\\
FG(u) 
	\ar@2 @/^/ [ur] ^-{FG(\alpha)}
	\ar@2 @/_/ [dr] _-{\tau_u}
&& v	
\\
& u
	\ar@2 @/_/ [ur] _-{\alpha}
}
\]
for every $2$-cell $\alpha:u\dfl v$ of~$\Xi$.
\end{itemize}
\end{theorem}

\begin{proof}
Let us define, for every $2$-cell~$f$ of~$\tck{\Xi}$, a $3$-cell~$\tau_{f}$ of~$\tck{\Delta}$ with shape
\[
\xymatrix @!C @C=2.5em @R=1em{
& FG(v)
	\ar@2 @/^/ [dr] ^-{\tau_v}
\ar@3 []!<0pt,-15pt>;[dd]!<0pt,15pt> ^-{\tau_f}
\\
FG(u) 
	\ar@2 @/^/ [ur] ^-{FG(f)}
	\ar@2 @/_/ [dr] _-{\tau_u}
&& v
\\
& u
	\ar@2 @/_/ [ur] _-{f}
}
\]
We extend the notation~$\tau_{\alpha}$ in a functorial way, according to the formulas $
\tau_{1_u} = 1_{\tau_u}$, $\tau_{fg} = \tau_f \tau_g$,
\[
\tau_{f\star_1 g} \:=\quad 
\vcenter{\xymatrix @!C @C=2.5em @R=1em{
&& FG(w)
	\ar@2 @/^/ [dr] ^-{\tau_w}
\ar@3 []!<0pt,-15pt>;[dd]!<0pt,15pt> ^-{\tau_g}
\\
& FG(v)
	\ar@2 @/^/ [ur] ^-{FG(g)}
	\ar@2 [dr] |-*+{\tau_v}
\ar@3 []!<0pt,-15pt>;[dd]!<0pt,15pt> ^-{\tau_f}
&& w
\\
FG(u) 
	\ar@2 @/^/ [ur] ^-{FG(f)}
	\ar@2 @/_/ [dr] _-{\tau_u}
&& v
	\ar@2 @/_/ [ur] _-{g}
\\
& u
	\ar@2 @/_/ [ur] _-{f}
}}
\]
and
\[
\tau_{f^-} \:=\quad 
\vcenter{\xymatrix @!C @C=2.5em @R=1em{
&& u
	\ar@2 @/^/ [dr] ^-{f}
\ar@3 []!<0pt,-15pt>;[dd]!<0pt,15pt> ^-{\tau_f^-}
\\
FG(v) 
	\ar@2 [r] ^-{FG(f)^-}
& FG(u) 
	\ar@2 @/^/ [ur] ^-{\tau_u}
	\ar@2 @/_/ [dr] _-{FG(f)}
&& v
	\ar@2 [r] ^-{f^-}
& u
\\
&& FG(v)
	\ar@2 @/_/ [ur] _-{\tau_v}
}}
\]

One checks that the $3$-cells~$\tau_f$ are well-defined, i.e.\ that their definition is compatible with the relations on $2$-cells of~$\tck{\Xi}$, such as the exchange relation:
\[
\tau_{fg\star_1 hk} \:=\: \tau_{(f\star_1 h)(g\star_1 k)}.
\]
Now, let us consider parallel $2$-cells $f,g:u\dfl v$ of~$\tck{\Xi}$. The $2$-cells~$G(f)$ and~$G(g)$ of~$\tck{\Sigma}$ are parallel so that, since~$\Gamma$ is a homotopy basis of~$\tck{\Sigma}$, there exists a $3$-cell 
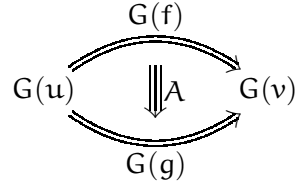
\[
\xymatrix @C=5em{
G(u)
	\ar@2 @/^4ex/ [r] ^-{G(f)} _-{}="src"
	\ar@2 @/_4ex/ [r] _-{G(g)} ^-{}="tgt"
\ar@3 "src"!<0pt,-10pt>;"tgt"!<0pt,10pt> ^-{A}
& G(v)
}
\]
in~$\tck{\Gamma}$. An application of~$F$ to~$A$ gives the $3$-cell
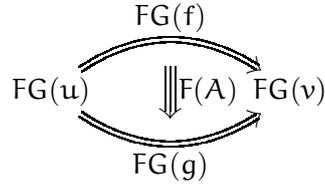
\[
\xymatrix @C=5em{
FG(u)
	\ar@2 @/^4ex/ [r] ^-{FG(f)} _-{}="src"
	\ar@2 @/_4ex/ [r] _-{FG(g)} ^-{}="tgt"
\ar@3 "src"!<0pt,-10pt>;"tgt"!<0pt,10pt> ^-{F(A)}
& FG(v)
}
\]
of~$\tck{\Xi}$, which, by definition of~$\Delta$ and functoriality of~$F$, is in~$\tck{\Delta}$. Using the $3$-cells~$F(A)$, $\tau_f$ and~$\tau_g$, we get the following $3$-cell from~$f$ to~$g$ in~$\tck{\Delta}$:
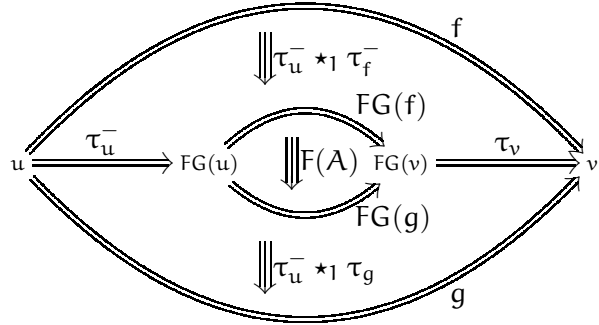
\[
\xymatrix @!C @C=4em {
{\sm u}
	\ar@2 @/^12ex/ [rrr] ^(0.75){f} _{}="src1"
	\ar@2 [r] ^-{\tau_u^-} 
	\ar@2 @/_12ex/ [rrr] _(0.75){g} ^{}="tgt3"
& {\sm FG(u)}
	\ar@2 @/^4ex/ [r] ^(0.8){FG(f)} ^{}="tgt1" _{}="src2"
	\ar@2 @/_4ex/ [r] _(0.8){FG(g)} ^{}="tgt2" _{}="src3"
& {\sm FG(v)}
	\ar@2 [r] ^-{\tau_v}
& {\sm v}
		\ar@3 "src1"!<-15pt,-10pt>;"tgt1"!<-15pt,10pt> ^-{\tau_u^-\star_1\tau_f^-}
		\ar@3 "src2"!<-5pt,-10pt>;"tgt2"!<-5pt,10pt> ^-{F(A)}
		\ar@3 "src3"!<-15pt,-10pt>;"tgt3"!<-15pt,10pt> ^-{\tau_u^-\star_1\tau_g}
}
\]
This concludes the proof that $\Delta=F(\Gamma) \amalg \tau_{\Xi}$ is a homotopy basis of the $(2,1)$-category~$\tck{\Xi}$.
\end{proof}

We deduce from Theorem~\ref{Theorem:HomotopyBasesTransfer} the following result.

\begin{theorem}[{\cite[Theorem~4.3]{Squier94}}]
\label{Theorem:Squier2}
Let~$\Sigma$ and~$\Xi$ be finite presentations of the same category. Then~$\Sigma$ is of finite derivation type if, and only if, $\Xi$ is of finite derivation type.
\end{theorem}

The following proposition is useful to prove that a presentation admits no finite homotopy basis.

\begin{proposition}
\label{Proposition:FiniteBasisExtracted}
Let~$\Sigma$ be a $2$-polygraph and let~$\Gamma$ be a homotopy basis of~$\tck{\Sigma}$. If~$\Sigma$ admits a finite homotopy basis, then there exists a finite subset of~$\Gamma$ that is a homotopy basis of~$\tck{\Sigma}$.
\end{proposition}

\begin{proof}
Let~$\Delta$ be a finite homotopy basis of~$\Sigma$ and let~$\delta$ be a $3$-cell of~$\Delta$. Since~$\Gamma$ is a homotopy basis of~$\Sigma$, there exists a $3$-cell~$A_{\delta}$ in~$\tck{\Gamma}$ with boundary~$(s(\delta),t(\delta))$. This induces a $3$-functor 
\[
F \::\: \tck{\Delta} \:\fl\: \tck{\Gamma}
\] 
that is the identity on~$\Sigma$ and such that $F(\delta)=A_{\delta}$ for every $3$-cell~$\delta$ of~$\Delta$. Let~$\Gamma_{\Delta}$ be the subset of~$\Gamma$ that contains all the generating $3$-cells that appear in the $3$-cells~$A_{\delta}$, for every~$\delta$ in~$\Delta$. Since~$\Delta$ is finite and each~$A_{\delta}$ contains finitely many $3$-cells of~$\Gamma$, we have that~$\Gamma_{\Delta}$ is finite. Finally, let us consider a $2$-sphere~$(f,g)$ of~$\tck{\Sigma}$. By hypothesis, there exists a $3$-cell~$A$ in~$\tck{\Delta}$ with boundary~$(f,g)$. By application of~$F$, one gets a $3$-cell~$F(A)$ in~$\tck{\Gamma}$ whose boundary is~$(f,g)$. Moreover, the $3$-cell~$F(A)$ is a composite of cells~$A_{\delta}$: hence, the $3$-cell~$F(A)$ is in~$\tck{\Gamma}_{\Delta}$. As a consequence, one gets~$f\equiv_{\Gamma_{\Delta}}g$, so that~$\Gamma_{\Delta}$ is a finite homotopy basis of~$\tck{\Sigma}$, which concludes the proof.
\end{proof}

\subsection{Squier completion for convergent presentations}
\label{Section:SquierCompletion}

Squier completion provides a way to extend a convergent presentation of a category~$\C$ into a coherent presentation of~$\C$. 

\subsubsection{Squier completion}

For~$\Sigma$ a $2$-polygraph, a \emph{family of generating confluences of~$\Sigma$} is a cellular extension of~$\tck{\Sigma}$ that contains exactly one $3$-cell  
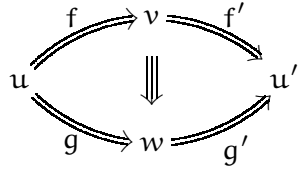
\[
\xymatrix @R=1em@C=3em @!C{
& {v}
	\ar@2 @/^/ [dr] ^{f'}
	\ar@3 []!<0pt,-15pt>;[dd]!<0pt,15pt> 
\\
{u}
	\ar@2 @/^/ [ur] ^{f}
	\ar@2 @/_/ [dr] _{g}
&& {u'}
\\
& {w}
	\ar@2 @/_/ [ur] _{g'}
}
\]
for every critical branching~$(f,g)$ of~$\Sigma$. We note that, if~$\Sigma$ is confluent, it always admits a family of generating confluences. However, such a family is not necessarily unique, since the $3$-cell can be directed in the reverse way and, for a given branching~$(f,g)$, we can have several possible $2$-cells~$f'$ and~$g'$ with the required shape. Normalisation strategies provide a deterministic way to construct a family of generating confluences, see~\cite[4.3.2]{GuiraudMalbos12advances}.

For a convergent $2$-polygraph~$\Sigma$, \emph{Squier completion of~$\Sigma$} is the $(3,1)$-polygraph denoted by~$\Sr(\Sigma)$ and defined by $\Sr(\Sigma)=(\Sigma,\Gamma)$, where~$\Gamma$ is a chosen family of generating confluences of~$\Sigma$. By the following result, if~$\Sigma$ is a convergent presentation of a category~$\C$, then Squier completion~$\Sr(\Sigma)$ is a coherent presentation of~$\C$.

\begin{theorem}[{\cite[Theorem~5.2]{Squier94}}]
\label{Theorem:Squier1}
Let $\Sigma$ be a convergent $2$-polygraph. Every family of generating confluences of~$\Sigma$ is a homotopy basis of~$\tck{\Sigma}$.
\end{theorem}

\begin{proof}
We fix a family of generating confluences of~$\Sigma$ and consider the corresponding Squier completion~$\Sr(\Sigma)$. We proceed in three steps. 

\medskip\noindent
\textbf{Step~1.} 
We prove that, for every local branching $(f,g):u\dfl(v,w)$ of~$\Sigma$, there exist $2$-cells $f':v\dfl u'$ and $g':w\dfl u'$ in~$\Sigma^*$ and a $3$-cell $A:f\star_1 f'\tfl g\star_1 g'$ in~$\tck{\Sr(\Sigma)}$, as in the following diagram:
\[
\xymatrix @R=1em@C=3em @!C{
& {v}
	\ar@2 @/^/ [dr] ^{f'}
	\ar@3 []!<0pt,-15pt>;[dd]!<0pt,15pt> ^*+{A}
\\
{u}
	\ar@2 @/^/ [ur] ^{f}
	\ar@2 @/_/ [dr] _{g}
&& {u'}
\\
& {w}
	\ar@2 @/_/ [ur] _{g'}
}
\]
As we have seen in the study of confluence of local branchings, in the case of an aspherical or Peiffer branching, we can choose~$f'$ and~$g'$ such that $f\star_1 f'=g\star_1 g'$: an identity $3$-cell is enough to link them. Moreover, if we have an overlapping branching~$(f,g)$ that is not critical, we have $(f,g)=(uhv,ukv)$ with~$(h,k)$ critical. We consider the $3$-cell $\alpha:h\star_1 h'\tfl k\star_1 k'$ of~$\Sr(\Sigma)$ corresponding to the critical branching~$(h,k)$ and we conclude that the following $2$-cells~$f'$ and~$g'$ and $3$-cell~$A$ satisfy the required conditions:
\[
f' \:=\: uh'v
\qquad
g' \:=\: uk'v
\qquad
A \:=\: u\alpha v.
\]

\medskip\noindent
\textbf{Step~2.} 
We prove that, for every parallel $2$-cells~$f$ and~$g$ of~$\Sigma^*$ whose common target is a normal form, there exists a $3$-cell from~$f$ to~$g$ in~$\tck{\Sr(\Sigma)}$. We proceed by noetherian induction on the common source~$u$ of~$f$ and~$g$, using the termination of~$\Sigma$. Let us assume that~$u$ is a normal form: then, by definition, both $2$-cells~$f$ and~$g$ must be equal to the identity of~$u$, so that $1_{1_u}:1_u\tfl 1_u$ is a $3$-cell of~$\tck{\Sr(\Sigma)}$ from~$f$ to~$g$. 

Now, let us fix a $1$-cell~$u$ of~$\Sigma^*$ with the following property: for every $1$-cell~$v$ of~$\Sigma^*$ such that~$u$ rewrites into~$v$ in one step, and for every parallel $2$-cells $f,g:v\dfl\rep{v}=\rep{u}$ of~$\Sigma^*$, there exists a $3$-cell from~$f$ to~$g$ in~$\tck{\Sr(\Sigma)}$. Let us consider parallel $2$-cells $f,g:u\dfl\rep{u}$ and let us prove the result by progressively constructing the following composite $3$-cell from~$f$ to~$g$ in~$\tck{\Sr(\Sigma)}$:
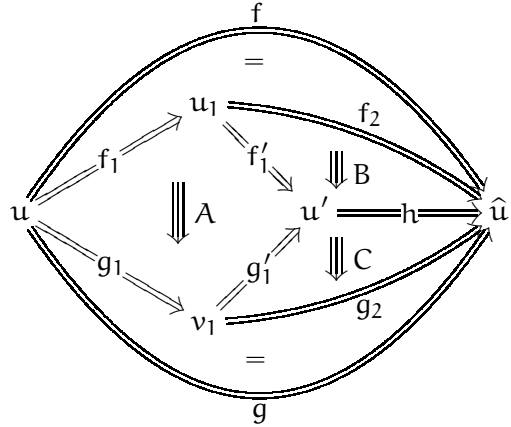
\[
\xymatrix@R=2.5em{
&& {u_1}
	\ar@2[dr] |-{f'_1}
	\ar@2 @/^2ex/[rrrd] ^-{f_2} _-{}="5"
        \ar@3 []!<-10pt,-30pt>;[dd]!<-10pt,30pt> ^*+{A}
\\
{u}
	\ar@2 @/^14ex/[rrrrr] ^-{f} ^-{}="1"
	\ar@2 @/_14ex/[rrrrr] _-{g} _-{}="3"
	\ar@2[urr] |-{f_1}
	\ar@2[drr] |-{g_1}
&&& {u'}
	\ar@2[rr] |-{h}
&& {\rep{u}}
\\
&& {v_1}
	\ar@2[ru] |-{g'_1}
	\ar@2 @/_2ex/[rrru] _-{g_2} ^-{}="8"
\ar@2{} "1";"1,3"!<17.5pt,2.5pt> |{=}
\ar@2{} "3";"3,3"!<17.5pt,-2.5pt> |{=}
\ar@3 "5"!<-9.55pt,-6.5pt>;"2,4"!<8pt,6.5pt> ^*+{B}
\ar@3 "2,4"!<8pt,-9.5pt>;"8"!<-9.2pt,4.5pt> ^*+{C}
}
\]
Since~$u$ is not a normal form, we can decompose $f=f_1\star_1 f_2$ and $g=g_1\star_1 g_2$ so that~$f_1$ and~$g_1$ are rewriting steps. They form a local branching~$(f_1,g_1)$ and we build the $2$-cells~$f_1'$ and~$g_1'$ of~$\Sigma^*$ together with the $3$-cell~$A$ of~$\tck{\Sr(\Sigma)}$, as in the first part of the proof. Then, we consider a $2$-cell~$h:u'\dfl\rep{u}$ in~$\Sigma^*$, that must exist by confluence of~$\Sigma$ and since~$\rep{u}$ is a normal form. We apply the induction hypothesis to the parallel $2$-cells~$f_2$ and~$f'_1\star_1 h$ in order to get~$B$ and, symmetrically, to the parallel $2$-cells~$g'_1\star_1 h$ and~$g_2$ to get~$C$.

\medskip\noindent
\textbf{Step 3.} We prove that every $2$-sphere of~$\tck{\Sigma}$ is the boundary of a $3$-cell of~$\tck{\Sr(\Sigma)}$. First, let us consider a $2$-cell $f:u\dfl v$ in~$\Sigma^*$. Using the confluence of~$\Sigma$, we choose $2$-cells
\[
\sigma_u \::\: u \:\dfl\: \rep{u}
\qquad\text{and}\qquad
\sigma_v \::\: v\:\dfl\: \rep{v} = \rep{u}
\]
in~$\Sigma^*$. By construction, the $2$-cells~$f\star_1\sigma_v$ and~$\sigma_u$ are parallel and their common target~$\rep{u}$ is a normal form. Thus, there exists a $3$-cell in~$\tck{\Sr(\Sigma)}$ from $f\star_1\sigma_v$ to~$\sigma_u$ or, equivalently, a $3$-cell~$\sigma_f$ from~$f$ to $\sigma_u\star_1\sigma_v^-$ in~$\tck{\Sr(\Sigma)}$, as in the following diagram:
\[
\xymatrix@C=2em @!C @R=1.5em {
u 
	\ar@2@/^3ex/ [rr] ^*+{f} _{}="src"
	\ar@2@/_/ [dr] _{\sigma_u}
&& v
\\	
& {\rep{u}}
	\ar@2@/_/ [ur] _{\sigma_v^-} 
		\ar@3 "src"!<0pt,-10pt>;[]!<0pt,15pt> ^*+{\sigma_f}
}
\]
Moreover, the $(3,1)$-category~$\tck{\Sr(\Sigma)}$ contains a $3$-cell~$\sigma_{f^-}$ from~$f^-$ to $\sigma_v\star_1\sigma_u^-$, given as the following composite:
\[
\xymatrix@C=2em @!C @R=1.5em {
&& {\rep{u}}
	\ar@2@/^/ [dr] ^{\sigma_v^-} 
\\
v 
	\ar@2[r] ^{f^-}
& u
	\ar@2@/_3ex/ [rr] _*+{f} _{}="tgt"
	\ar@2@/^/ [ur] ^{\sigma_u}
&& v
	\ar@2[r] ^{\sigma_v}
& {\rep{u}}
	\ar@2[r] ^{\sigma_u^-}
& u
\ar@3 "1,3"!<0pt,-15pt>;"tgt"!<0pt,10pt> ^*+{\sigma_f^-}
}
\]
Now, let us consider a general $2$-cell $f:u\dfl v$ of~$\tck{\Sigma}$. By construction of~$\tck{\Sigma}$, the $2$-cell~$f$ can be decomposed (in general in a non-unique way) into a zigzag
\[
\xymatrix @C=2em {
u 
	\ar@2[r] ^-{f_1}
& v_1
	\ar@2[r] ^-{g_1^-}
& u_2 
	\ar@2[r] ^-{f_2}
& (\cdots)
	\ar@2[r] ^-{g_{n-1}^-}
& u_n
	\ar@2[r] ^-{f_n}
& v_n 
	\ar@2[r] ^-{g_n^-}
& v
}
\]
where each~$f_i$ and~$g_i$ is a $2$-cell of~$\Sigma^*$. We define~$\sigma_f$ as the following composite $3$-cell of~$\tck{\Sr(\Sigma)}$, with source~$f$ and target $\sigma_u\star_1\sigma_v^-$: 
\[
\xymatrix @C=2em @!C{
u 
	\ar@2[rr] ^{f_1} _{}="f1"
	\ar@2@/_/ [dr] _{\sigma_{u}}
&& v_1
	\ar@2[rr] ^{g_1^-} _{}="g1"
	\ar@2[dr] |{\sigma_{v_1}}
&& (\cdots)
	\ar@2[rr] ^{f_n} _{}="fn"
	\ar@2[dr] |{\sigma_{u_n}}
&& v_n
	\ar@2[rr] ^{g_n^-} _{}="gn"
	\ar@2[dr] |{\sigma_{v_n}}
&& v
\\
& {\rep{u}}
	\ar@2[ur] |{\sigma_{v_1}^-}
	\ar@{=} [rr] ^{}="id1"
&& {\rep{u}}
	\ar@2[ur] |{\sigma_{u_2}^-}
	\ar@{=} [r] 
& (\cdots)
	\ar@{=} [r] 
& {\rep{u}}
	\ar@2[ur] |{\sigma_{v_n}^-}
	\ar@{=} [rr] ^{}="idn"
&& {\rep{u}}
	\ar@2@/_/ [ur] _{\sigma_{v}^-}
\ar@3 "f1"!<-5pt,-10pt>;"2,2"!<-5pt,15pt> ^{\sigma_{f_1}}
\ar@3 "g1"!<-5pt,-10pt>;"2,4"!<-5pt,15pt> ^{\sigma_{g_1^-}}
\ar@3 "fn"!<-5pt,-10pt>;"2,6"!<-5pt,15pt> ^{\sigma_{f_n}}
\ar@3 "gn"!<-5pt,-10pt>;"2,8"!<-5pt,15pt> ^{\sigma_{g_n^-}}
\ar@{} "1,3"!<0pt,-15pt>;"id1"!<0pt,10pt> |{=}
\ar@{} "1,7"!<0pt,-15pt>;"idn"!<0pt,10pt> |{=}
}
\]
We proceed similarly for any other $2$-cell $g:u\dfl v$ of~$\tck{\Sigma}$, to get a $3$-cell~$\sigma_g$ from~$g$ to $\sigma_u\star_1\sigma_v^-$ in~$\tck{\Sr(\Sigma)}$. Thus, the composite $\sigma_f\star_2\sigma_g^-$ is a $3$-cell of~$\tck{\Sr(\Sigma)}$ from~$f$ to~$g$, concluding the proof.
\end{proof}

Theorem~\ref{Theorem:Squier1} is extended to higher-dimensional polygraphs in~\cite[Proposition~4.3.4]{GuiraudMalbos09}. In the special case of presentations of monoids, we recover the original result of Squier.

\begin{theorem}[{\cite[Theorem~5.3]{Squier94}}]
\label{Theorem:Squier}
If a monoid admits a finite convergent presentation, then it is of finite derivation type.
\end{theorem}

\subsubsection{Generating confluences in the reduced case}

Theorem~\ref{Theorem:Squier1} holds for any choice of family of generating confluences. If~$\Sigma$ is a reduced convergent $2$-polygraph, we can construct an explicit such family as follows. Let~$\sigma$ be the leftmost normalisation strategy of~$\Sigma$. Since~$\Sigma$ is reduced, every critical branching of~$\Sigma$ has the form
\[
\xymatrix{
\strut 
	\ar[r] _-{\rep{u}}
	\ar@/^6ex/ [rrr] ^{u'} _-{}="t1"
& \strut
	\ar[rr] ^-{\rep{w}} _(0.75){}="s2" ^(0.25){}="s1"
	\ar@/_6ex/ [rrr] _{v'} ^-{}="t2"
&& \strut
	\ar[r] ^-{\rep{v}}
& \strut
\ar@2 "s1"!<0pt,7.5pt>;"t1"!<0pt,-7.5pt> ^-{\alpha}
\ar@2 "s2"!<0pt,-7.5pt>;"t2"!<0pt,7.5pt> ^-{\beta}
}
\]
where~$\alpha$ and~$\beta$ are $2$-cells of~$\Sigma$ and where~$\rep{u}$, $\rep{w}$ and~$\rep{v}$ are non-identity normal forms. Let us note that~$\alpha\rep{v}$ is the leftmost reduction step of~$\rep{u}\rep{w}\rep{v}$ and that~$\rep{u}\beta$ is its rightmost reduction step. In particular, we have 
\[
\sigma(\rep{u}\rep{w}\rep{v}) \:=\: \alpha\rep{v} \star_1 \sigma(u'\rep{v}).
\]
We define~$\Sigma_3$ as the cellular extension of~$\tck{\Sigma}$ made of one $3$-cell with the following shape, for every critical branching $b=(\alpha\rep{v},\rep{u}\beta)$ of~$\Sigma$:
\[
\xymatrix @!C @C=2em @R=2em {
& {\rep{u}v'}
	\ar@2@/^/ [dr] ^-{\sigma(\rep{u}v')}
\\
{\rep{u}\rep{w}\rep{v}}
	\ar@2@/^/ [ur] ^-{\rep{u}\beta}
	\ar@2@/_/ [rr] _-{\sigma(\rep{u}\rep{w}\rep{v})} ^-{}="tgt"
&& {\rep{uwv}}
\ar@3 "1,2"!<-5pt,-20pt>;"tgt"!<-5pt,15pt> ^-{\omega_b}
}
\]

\begin{example}

The \emph{standard presentation} of a category~$\C$ is the $2$-polygraph~$\Std_2(\C)$ defined as follows. The $0$-cells and $1$-cells of~$\Std_2(\C)$ are the ones of~$\C$, with~$\rep{u}$ denoting a $1$-cell~$u$ of~$\C$ when seen as a $1$-cell of~$\Std_2(\C)$. The $2$-polygraph~$\Std_2(\C)$ contains a $2$-cell 
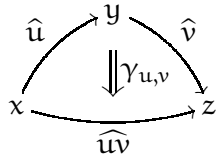
\[
\xymatrix{
& y
	\ar@/^/ [dr] ^-{\rep{v}}
\\
x
	\ar@/^/ [ur] ^-{\rep{u}}
	\ar@/_/ [rr] _-{\rep{uv}} _-{}="tgt"
&& z
\ar@2 "1,2"!<0pt,-15pt>;"tgt"!<0pt,10pt> ^-{\gamma_{u,v}}	
}
\]
for all $1$-cells $u:x\fl y$ and $v:y\fl z$ of~$\C$, and a $2$-cell
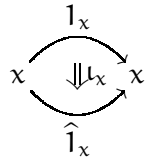
\[
\xymatrix@C=3em{
x 
	\ar@/^3ex/ [r] ^-{1_x} ^-{}="src"
	\ar@/_3ex/ [r] _-{\rep{1}_x} _-{}="tgt"
& x
\ar@2 "src"!<0pt,-10pt>;"tgt"!<0pt,10pt> ^-{\iota_x}	
}
\]
for every $0$-cell~$x$ of~$\C$. The \emph{standard coherent presentation of~$\C$} is the $(3,1)$-polygraph denoted by~$\Std_3(\C)$ and obtained by extension of~$\Std_2(\C)$ with the homotopy basis made of the following $3$-cells: 
\begin{itemize}
\item for all $1$-cells $u:x\fl y$, $v:y\fl z$ and $w:z\fl t$ of~$\C$, one $3$-cell
\[
\xymatrix @C=3em @R=1em {
& {\rep{uv}\rep{w}}
	\ar@2 @/^/ [dr] ^-{\gamma_{uv,w}}
\\
{\rep{u}\rep{v}\rep{w}}
	\ar@2 @/^/ [ur] ^-{\gamma_{u,v}\rep{w}}
	\ar@2 @/_/ [dr] _-{\rep{u}\gamma_{v,w}}
&& {\rep{uvw}}
\\
& {\rep{u}\rep{vw}}
	\ar@2 @/_/ [ur] _-{\gamma_{u,vw}}
\ar@3 "1,2"!<0pt,-15pt>;"3,2"!<0pt,15pt> ^-{\alpha_{u,v,w}}
}
\]
\item for every $1$-cell $u:x\fl y$ of~$\C$, two $3$-cells
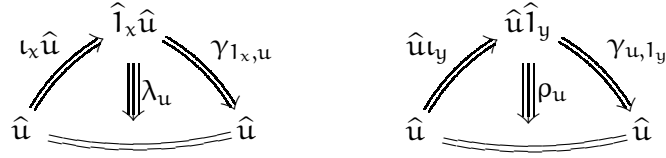
\[
\xymatrix{
& {\rep{1}_x\rep{u}}
	\ar@2@/^/ [dr] ^-{\gamma_{1_x,u}}
\\
{\rep{u}}
	\ar@2@/^/ [ur]  ^{\iota_x \rep{u}}
	\ar@2{=}@/_/ [rr] _-{}="tgt"
&& {\rep{u}}
\ar@3 "1,2"!<0pt,-15pt>;"tgt"!<0pt,10pt> ^-{\lambda_u}	
}
\qquad\qquad
\xymatrix{
& {\rep{u}\rep{1}_y}
	\ar@2@/^/ [dr] ^-{\gamma_{u,1_y}}
\\
{\rep{u}}
	\ar@2@/^/ [ur] ^-{\rep{u}\iota_y}
	\ar@2{=}@/_/ [rr] _-{}="tgt"
&& {\rep{u}}
\ar@3 "1,2"!<0pt,-15pt>;"tgt"!<0pt,10pt> ^-{\rho_u}	
}
\] 
\end{itemize}
Let us prove that~$\Std_3(\C)$ is, indeed, a coherent presentation of~$\C$. The standard presentation~$\Std_2(\C)$ is not terminating: indeed, for every $0$-cell~$x$ of~$\C$, the $2$-cell~$\iota_x$ creates infinite rewriting sequences  
\[
1_x \:\dfl\: \rep{1}_x \:\dfl\: \rep{1}_x\rep{1}_x \:\dfl\: \rep{1}_x\rep{1}_x\rep{1}_x \:\dfl\: \cdots
\]
However, we get a convergent presentation of~$\C$ by reversing all the $2$-cells~$\iota_x$ into~$\iota_x^-$. Indeed, for termination, we consider the size of the $1$-cells (the number of generators they contain) and we check that each $2$-cell~$\gamma_{u,v}$ has source of size~$2$ and target of size~$1$, while each $2$-cell~$\iota^-_x$ has source of size~$1$ and target of size~$0$. As a consequence, for every non-identity $2$-cell $f:u\dfl v$ of the free $2$-category, the size of~$u$ is strictly greater than the size of~$v$. For confluence, we study the critical branchings, divided into three families:
\begin{itemize}
\item for all $1$-cells $u:x\fl y$, $v:y\fl z$ and $w:z\fl t$ of~$\C$, one critical branching $(\gamma_{u,v}\rep{w}, \rep{u}\gamma_{v,w})$, giving the $3$-cell
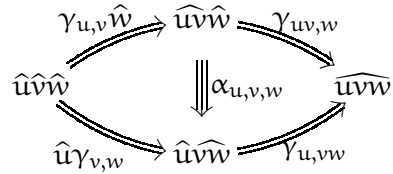
\[
\xymatrix @!C @C=3em @R=1em {
& {\rep{uv}\rep{w}}
	\ar@2 @/^/ [dr] ^-{\gamma_{uv,w}}
        \ar@3 []!<0pt,-15pt>;[dd]!<0pt,15pt> ^-{\alpha_{u,v,w}}
\\
{\rep{u}\rep{v}\rep{w}}
	\ar@2 @/^/ [ur] ^-{\gamma_{u,v}\rep{w}}
	\ar@2 @/_/ [dr] _-{\rep{u}\gamma_{v,w}}
&& {\rep{uvw}}
\\
& {\rep{u}\rep{vw}}
	\ar@2 @/_/ [ur] _-{\gamma_{u,vw}}
}
\]
\item for every $1$-cell $u:x\fl y$ of~$\C$, two critical branchings $(\gamma_{1_x,u}, \iota^-_x\rep{u})$ and $(\gamma_{u,1_y}, \rep{u} \iota^-_y)$, producing the $3$-cells
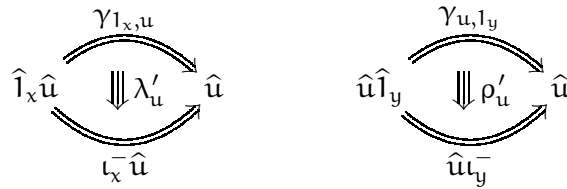
\[
\xymatrix @C=4em @!C {
{\rep{1}_x\rep{u}}
	\ar@2@/^4ex/ [r] ^{\gamma_{1_x,u}} _{}="src"
	\ar@2@/_4ex/ [r] _{\iota_x^-\rep{u}} ^{}="tgt"
	\ar@3 "src"!<-2.5pt,-12.5pt>;"tgt"!<-2.5pt,12.5pt> ^-*+{\lambda'_u}
& {\rep{u}}
}
\qquad\qquad
\xymatrix @C=4em @!C {
{\rep{u}\rep{1}_y}
	\ar@2@/^4ex/ [r] ^{\gamma_{u,1_y}} _{}="src"
	\ar@2@/_4ex/ [r] _{\rep{u}\iota_y^-} ^{}="tgt"
	\ar@3 "src"!<-2.5pt,-12.5pt>;"tgt"!<-2.5pt,12.5pt> ^-*+{\rho'_u}
& {\rep{u}}
}
\]
\end{itemize}
Since considering the $2$-cells~$\iota_x$ or~$\iota^-_x$ as generators does not change the generated $(2,1)$-category, we get that those three families of $3$-cells form a homotopy basis of~$\Std_2(\C)$. We replace~$\lambda'_u$ by $\lambda_u=\iota_x\rep{u}\star_1 \lambda_u$ and~$\rho'_u$ by $\rho_u=\rep{u}\iota_y\star_1\rho_u$ to get the result.
\end{example}

\begin{example}
\label{Example:HomotopicalCompletion}

Let us consider the monoid~$\M$ presented by the $2$-polygraph
\[
\Sigma \:=\: 
\bigpres{x, y}{xyx \odfll{\alpha} yy}.
\]
We prove that~$\Sigma$ terminates with the deglex order generated by~$x<y$. The $2$-polygraph~$\Sigma$ has one, non confluent critical branching $(\alpha yx, xy\alpha)$. Knuth-Bendix completion~$\check{\Sigma}$ of~$\Sigma$ is obtained by adjunction of the following $2$-cell $\beta:yyyx\dfl xyyy$:
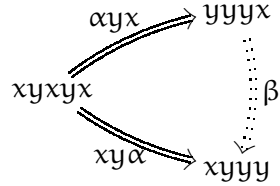
\[
\xymatrix @!C @C=3em @R=1.5em {
& {yyyx}
	\ar@2{.>}@/^/ [dd] ^-{\beta}
\\
{xyxyx}
	\ar@2@/^/ [ur] ^-{\alpha yx} 
	\ar@2@/_/ [dr] _-{xy\alpha} 
\\
& {xyyy}
}
\]
Then, Squier completion~$\Sr(\check{\Sigma})$ extends~$\check{\Sigma}$ with the following two $3$-cells:
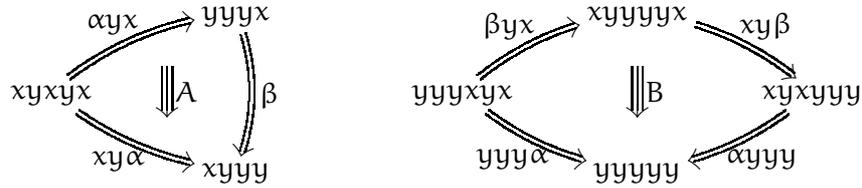
\[
\xymatrix @!C @C=3em @R=1.5em {
& {yyyx}
	\ar@2@/^/ [dd] ^-{\beta}
\\
{xyxyx}
	\ar@2@/^/ [ur] ^-{\alpha yx} _(0.66){}="src"
	\ar@2@/_/ [dr] _-{xy\alpha} ^(0.66){}="tgt"
		\ar@3 "src"!<0pt,-15pt>;"tgt"!<0pt,15pt> ^-{A}
\\
& {xyyy}
}
\qquad\qquad
\xymatrix @!C @C=2em @R=1.5em {
& {xyyyyx}
	\ar@2@/^/ [dr] ^-{xy\beta}
		\ar@3 []!<0pt,-20pt>;[dd]!<0pt,20pt> ^-{B}
\\
{yyyxyx}
	\ar@2@/^/ [ur] ^-{\beta yx}
	\ar@2@/_/ [dr] _-{yyy\alpha}
&& {xyxyyy}
	\ar@2@/^/ [dl] ^-{\alpha yyy}
\\
& {yyyyy}
}
\]
In fact, the $3$-cell~$A$ is sufficient to get a homotopy basis of~$\check{\Sigma}$, as witnessed by the following $3$-sphere of the $(3,1)$-category~$\tck{\Sr(\check{\Sigma})}$:
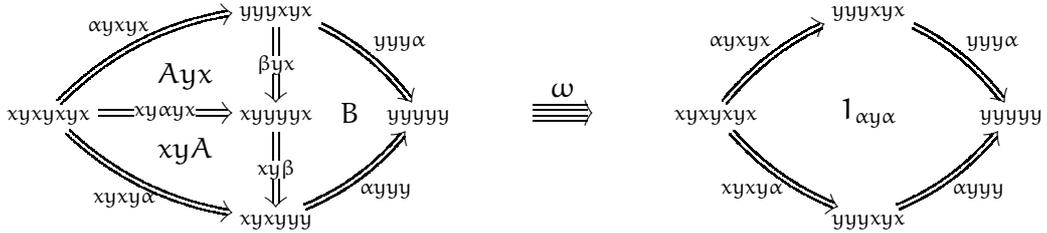
\[
\xymatrix @C=2em{
&& {\sm yyyxyx}
	\ar@2@/^/ [dr] ^-{\sm yyy\alpha}
	\ar@2 [d] |-{\sm \beta yx} _(0.4){}="tgt1"
&&&&& {\sm yyyxyx}
	\ar@2@/^/ [dr] ^-{\sm yyy\alpha}
\\
{\sm xyxyxyx}
	\ar@2@/^2ex/ [urr] ^-{\sm \alpha yxyx}
	\ar@2 [rr] |-{\sm xy\alpha yx}
	\ar@2@/_2ex/ [drr] _-{\sm xyxy\alpha}
&& {\sm xyyyyx}
	\ar@2 [d] |-{\sm xy\beta} _(0.6){}="tgt2"
	\ar@{} [r] |-{B}
	\ar@{} "2,1";"tgt1" |(0.6){Ayx}
	\ar@{} "2,1";"tgt2" |(0.6){xyA}
& {\sm yyyyy}
& \strut
	\ar@4 [r] ^-*+{\omega}
&& {\sm xyxyxyx}
	\ar@2@/^/ [ur] ^-{\sm \alpha yxyx}
	\ar@{} [rr] |-{1_{\alpha y \alpha}}
	\ar@2@/_/ [dr] _-{\sm xyxy\alpha}
&& {\sm yyyyy}
\\
&& {\sm xyxyyy}
		\ar@2@/_/ [ur] _-{\sm \alpha yyy}
&&&&& {\sm yyyxyx}
	\ar@2@/_/ [ur] _-{\sm \alpha yyy}
}
\]
Indeed, the $3$-sphere~$\omega$ proves that the boundary of~$B$ is also the boundary of a $3$-cell of the $(3,1)$-category $\tck{(\Sr(\check{\Sigma})\setminus\ens{B})}$. This elimination mechanism, based on the study of the triple critical branchings of~$\check{\Sigma}$ is part of the \emph{homotopical reduction procedure} introduced in~\cite{GaussentGuiraudMalbos}. This construction coherently eliminates pairs of redundant cells of a coherent presentation. On this particular example, it yields that the $(2,1)$-category~$\tck{\Sigma}$ admits an empty homotopy basis, i.e.\ that the $(3,1)$-polygraph~$(\Sigma,\emptyset)$ is a coherent presentation of the monoid~$\M$.
\end{example}

\section{A homological finiteness condition}
\label{Section:homologicalFinitenessCondition}

\subsection{Monoids of finite homological type}

\subsubsection{Resolutions}

Let~$\M$ be a monoid. We denote by~$\Zb\M$ the ring generated by~$\M$, that is, the free abelian group over~$\M$, equipped with the canonical extension of the product of~$\M$:
\[
\big(\sum_{u\in\M} \lambda_u u\big)
\big(\sum_{v\in\M} \lambda_v v\big)
\:=\:
\sum_{u,v\in\M} \lambda_u\lambda_v uv
\:=\:
\sum_{w\in \M}
\sum_{uv=w}
\lambda_u\lambda_v w.
\]
Given a (left) $\Zb\M$-module~$M$, a \emph{resolution of~$M$} is an exact sequence of $\Zb\M$-modules
\[
\xymatrix{
(\cdots) \ar[r]
& M_{n} \ar[r] ^-{d_n} 
& M_{n-1} \ar[r]
& (\cdots) \ar[r] 
& M_2 \ar[r] ^-{d_2} 
& M_1 \ar[r] ^-{d_1}
& M_0 \ar[r] ^-{d_0}
& M \ar[r]
& 0
}
\]
that is, a sequence $(M_n)_{n\in\Nb}$ of $\Zb\M$-modules, together with a sequence $(d_n)_{n\in\Nb}$ of morphisms of $\Zb\M$-modules, called the \emph{boundary maps}, such that~$d_0$ is surjective and
\[
\im d_{n+1} = \ker d_n
\] 
holds for every natural number~$n$. In particular, the sequence $(M_n,d_n)_{n\in\Nb}$ is a \emph{(chain) complex of $\Zb\M$-modules}, that is, we have the inclusion $\im d_{n+1} \subseteq \ker d_n$ or, equivalently, the relation $d_nd_{n+1}=0$ holds for every natural number~$n$. Such a resolution is called \emph{projective} (resp. \emph{free}) if all the modules~$M_n$ are projective (resp. free). Given a natural number~$n$, a \emph{partial resolution of length~$n$ of~$\M$} is defined in a similar way but with a bounded sequence $(M_k)_{0\leq k\leq n}$ of $\Zb\M$-modules.

\subsubsection{Contracting homotopies}

Given a complex of $\Zb\M$-modules
\[
\xymatrix{
(\cdots) \ar[r]
& M_{n+1} \ar[r] ^-{d_{n+1}} 
& M_n \ar[r] ^-{d_n}
& M_{n-1} \ar[r] 
& (\cdots) \ar[r]
& M_1 \ar[r] ^-{d_1}
& M_0 \ar[r] ^-{d_0}
& M \ar[r]
& 0
}
\]
a method to prove that such a complex is a resolution of~$\M$ is to construct a \emph{contracting homotopy}, that is a sequence of morphisms of $\Zb$-modules
\[
\xymatrix{
(\cdots)
& M_{n+1} \ar[l]
& M_{n} \ar[l] _-{i_{n+1}}
& M_{n-1} \ar[l] _-{i_n} 
& (\cdots) \ar[l] 
& M_1 \ar[l]
& M_0 \ar[l] _-{i_1}
& M \ar[l] _-{i_0}
}
\]
such that $d_0 i_0=\id_{M}$ and, for every~$n$, we have
\[
d_{n+1} i_{n+1} + i_n d_n \:=\: \id_{M_n}.
\]
Indeed, in that case, we have that~$d_0$ is surjective. Moreover, for every natural number~$n$  and every~$x$ in~$\ker d_n$, the equality $d_{n+1} i_{n+1}(x)=x$ holds, proving that~$x$ is in~$\im d_{n+1}$, so that $\ker d_n \subseteq \im d_{n+1}$ holds. As a consequence, the considered complex is a resolution of~$M$.

\subsubsection{Homological type left-\pdf{\FP_n}}
\label{Definition:LeftFPn}

If~$\M$ is a monoid, the \emph{trivial $\Zb\M$-module} is the abelian group~$\Zb$ equipped with the trivial action~$un=n$, for every~$u$ in~$\M$ and~$n$ in~$\Zb$. A monoid~$\M$ is of \emph{homological type left-$\FP_n$}, for a natural number~$n$, if there exists a partial resolution of length~$n$ of the trivial $\Zb\M$-module~$\Zb$ by projective, finitely generated $\Zb\M$-modules:
\[
\xymatrix{
P_{n} \ar[r] ^-{d_n}
& P_{n-1} \ar[r] ^-{d_{n-1}}
& (\cdots) \ar[r] ^-{d_2}
& P_1 \ar[r] ^-{d_1} 
& P_0 \ar[r] ^-{d_0} 
& \Zb \ar[r]
& 0.
}
\]
A monoid~$\M$ is of \emph{homological type left-$\FP_{\infty}$} if there exists a resolution of~$\Zb$ by projective, finitely generated $\Zb\M$-modules.

\begin{lemma}
\label{lemme_pl_n}
Let~$\M$ be a monoid and let~$n$ be a natural number. The following assertions are equivalent:  

\begin{enumerate}[\bf i)]
\item The monoid~$\M$ is of homological type left-$\FP_n$.
\item There exists a free, finitely generated partial resolution of the trivial $\Zb\M$-module~$\Zb$ of length~$n$
\[
\xymatrix{
F_n \ar[r] 
& F_{n-1} \ar[r] 
& (\cdots) \ar[r] 
& F_0 \ar[r] 
& \Zb \ar[r]
& 0.
}
\]
\item For every $0\leq k<n$ and every projective, finitely generated partial resolution of the trivial  $\Zb\M$-module~$\Zb$ of length~$k$
\[
\xymatrix{
P_k \ar[r] ^-{d_k} 
& P_{k-1} \ar[r]^{d_{k-1}}
& (\cdots) \ar[r] 
& P_0 \ar[r]^{d_0} 
& \Zb \ar[r]
& 0,
}
\]
the $\Zb\M$-module~$\ker d_k$ is finitely generated.  
\end{enumerate}
\end{lemma}

Lemma~\ref{lemme_pl_n} is a consequence of the following generalisation of Schanuel's lemma. If
\[
\xymatrix{
0\ar[r]
& Q \ar[r] 
& P_n \ar[r] 
& (\cdots) \ar[r] 
& P_0 \ar[r] 
& \Zb \ar[r]
& 0
}
\]
and 
\[
\xymatrix{
0\ar[r]
& Q' \ar[r] 
& P'_n \ar[r] 
& (\cdots) \ar[r] 
& P'_0 \ar[r] 
& \Zb \ar[r]
& 0
}
\]
are exact sequences of $\Zb\M$-modules, with each~$P_k$ and~$P'_k$ finitely generated and projective, then the $\Zb\M$-module~$Q$ is finitely generated if, and only if, the $\Zb\M$-module~$Q'$ is finitely generated. 

\subsection{Monoids of homological type left-\pdf{\FP_2}}

\subsubsection{Presentations and partial resolutions of length~\pdf{2}}
\label{Section:ReidemesterFoxJacobian}

Let~$\M$ be a monoid and let~$\Sigma$ be a presentation of~$\M$. Let us define a partial resolution of length~$2$ of~$\Zb$ by free $\Zb\M$-modules
\[
\xymatrix{
\Zb\M[\Sigma_2]
	\ar [r] ^-{d_2}
& \Zb\M[\Sigma_1]
	\ar [r] ^-{d_1}
& \Zb\M
	\ar [r] ^-{\epsilon}
& \Zb
	\ar [r]
& 0. 
}
\]
The $\Zb\M$-modules $\Zb\M[\Sigma_1]$ and $\Zb\M[\Sigma_2]$ are the free $\Zb\M$-modules over~$\Sigma_1$ and~$\Sigma_2$, respectively: they contain the formal sums of elements denoted by~$u[x]$, where~$u$ is an element of~$\M$ and~$x$ is a $1$-cell of~$\Sigma$ or a $2$-cell of~$\Sigma$. Let us note that~$\Zb\M$ is isomorphic to the free $\Zb\M$-module over the singleton~$\Sigma_0$. The boundary maps are defined, on generators, by
\[
\begin{array}{r c l}
\Zb\M &\:\ofll{\epsilon}\: &\Zb \\
u &\:\longmapsto\: &1
\end{array}
\qquad\qquad
\begin{array}{r c l}
{\Zb\M[\Sigma_1]} &\:{\ofll{d_1}}\: &{\Zb\M} \\
{[x]} &\:{\longmapsto}\: &{\cl{x} - 1}
\end{array}
\qquad\qquad
\begin{array}{r c l}
{\Zb\M[\Sigma_2]} &\:{\ofll{d_2}}\: &{\Zb\M[\Sigma_1]} \\
{[\alpha]} &\:{\longmapsto}\: &{[s(\alpha)] - [t(\alpha)]}
\end{array}
\]
The maps~$\epsilon$ and~$d_2$ are respectively called the \emph{augmentation map} and the \emph{Reidemester-Fox Jacobian of~$\Sigma$}. In the definition of~$d_2$, the bracket~$[\cdot]$ is extended to the $1$-cells of~$\tck{\Sigma}$ thanks to the relation
\[
[1] \:=\: 0 
\qquad\text{and}\qquad
[uv] \:=\: [u] + \cl{u}[v],
\]
for all $1$-cells~$u$ and~$v$ of~$\Sigma$.

\begin{proposition}
\label{Proposition:LowSquierResolution}
Let~$\M$ be a monoid and let~$\Sigma$ be a presentation of~$\M$. The sequence of $\Zb\M$-modules
\[
\xymatrix{
\Zb\M[\Sigma_2]
	\ar [r] ^-{d_2}
& \Zb\M[\Sigma_1]
	\ar [r] ^-{d_1}
& \Zb\M
	\ar [r] ^-{\epsilon}
& \Zb
	\ar [r]
& 0
}
\]
is a partial free resolution of length~$2$ of~$\Zb$.
\end{proposition}

\begin{proof}
We first note that the sequence is a chain complex. Indeed, the augmentation map is surjective by definition. Moreover, we have 
\[
\epsilon d_1[x] \:=\: \epsilon (\cl{x}) - \epsilon(1) \:=\: 1 - 1 \:=\: 0,
\]
for every $1$-cell~$x$ of~$\Sigma$. In order to check that $d_1d_2 = 0$, we first prove, by induction on the length, that we have $d_1[u]=\cl{u}-1$ for every $1$-cell~$u$ of~$\tck{\Sigma}$. For the unit, we have $d_1[1]=d_1(0)=0$ and $\cl{1}-1=0$. Then, for a composite $1$-cell~$uv$ such that the result holds for both~$u$ and~$v$, we get
\[
d_1[uv] \:=\: d_1[u] + \cl{u} d_1[v] \:=\: \cl{u} - 1 + \cl{uv} - \cl{u} \:=\: \cl{uv} - 1.
\]
As a consequence, we have
\[
d_1d_2 [\alpha] \:=\: d_1[s(\alpha)] - d_1[t(\alpha)] \:=\: \cl{s(\alpha)} - \cl{t(\alpha)} \:=\: 0,
\]
for every $2$-cell~$\alpha$ of~$\Sigma$, where the last equality comes from  $\cl{s(\alpha)}=\cl{t(\alpha)}$, that holds since~$\Sigma$ is a presentation of~$\M$. 

The rest of the proof consists in defining contracting homotopies~$i_0$, $i_1$, $i_2$:
\[
\xymatrix{
\Zb\M[\Sigma_2]
	\ar@<+0.5ex> [r] ^-{d_2}
& \Zb\M[\Sigma_1]
	\ar@<+0.5ex> [r] ^-{d_1}
        \ar@<+0.5ex> [l] ^-{i_2}
& \Zb\M
	\ar@<+0.5ex> [r] ^-{\epsilon}
         \ar@<+0.5ex> [l] ^-{i_1}
& \Zb
         \ar@<+0.5ex> [l] ^-{i_0}
}
\]
We choose a representative~$\rep{u}$ in~$\Sigma_1^*$ for every element~$u$ of~$\M$, with $\rep{1}_x=1_x$ for every $0$-cell~$x$ of~$\Sigma$, and we fix a normalisation strategy~$\sigma$ for~$\Sigma$. Then we define the morphisms of $\Zb$-modules~$i_0$, $i_1$ and~$i_2$ by their values on generic elements
\[
i_0(1) \:=\: 1,
\qquad\qquad 
i_1(u) \:=\: [\rep{u}],
\qquad\qquad 
i_2(u[x]) \:=\: [\sigma(\rep{u}x)],
\]
where the bracket~$[\cdot]$ is extended to every $2$-cell of~$\tck{\Sigma}$ thanks to the relations
\[
[1_u] \:=\: 0, 
\qquad
[u f v] \:=\: \cl{u}[f]
\qquad\text{and}\qquad
[f\star_1 g] \:=\: [f] + [g],
\]
for all $1$-cells~$u$ and~$v$ and $2$-cells~$f$ and~$g$ of~$\tck{\Sigma}$ such that the composites~$ufv$ and~$f\star_1 g$ are defined. 

First, we have $\epsilon i_0 = \id_{\Zb}$. Next, for every~$u$ in~$\M$, we have $i_0\epsilon(u) = 1$ and
\[
d_1i_1 (u) \:=\: d_1[\rep{u}] = u - 1.
\]
Thus $d_1 i_1+i_0\epsilon = \id_{\Zb\M}$. Finally, we have, on the one hand,
\[
i_1 d_1(u[x]) \:=\: i_1(u\cl{x} - u) \:=\:  [\rep{ux}]  - [\rep{u}]
\]
and, on the other hand, 
\[
d_2 i_2 (u[x]) \:=\: d_2[\sigma(\rep{u}x)] \:=\: [\rep{u}x] - [\rep{ux}] \:=\: u[x] + [\rep{u}] - [\rep{ux}].
\]
For this equality, we check that $d_2[f]=[s(f)]-[t(f)]$ holds for every $2$-cell~$f$ of~$\tck{\Sigma}$ by induction on the size of~$f$. Hence we have $d_2 i_2 + i_1d_1=\id_{\Zb\M[\Sigma_1]}$, thus concluding the proof.
\end{proof}

From Proposition~\ref{Proposition:LowSquierResolution}, we deduce the following result:

\begin{theorem}
\label{Theorem:FP1FP2}
The following properties hold.
\begin{enumerate}[{\bf i)}]
\item Every monoid is of homological type left-$\FP_0$. 
\item Every finitely generated monoid is of homological type left-$\FP_1$.
\item Every finitely presented monoid is of homological type left-$\FP_2$. 
\end{enumerate}
\end{theorem}

\subsubsection{Examples}

Let us consider the monoid~$\M$ presented by the $2$-polygraph
\[
\Sigma \:=\: \bigpres {a,c,t} {at^{n+1} \odfll{\alpha_n}ct^n \,,\, n\in\Nb}.
\]
The monoid~$\M$ is finitely generated and, thus, it is of homological type left-$\FP_1$. However, for every natural number~$n$, we have 
\begin{align*}
d_2[\alpha_{n+1}] 
	\:&=\: [at^{n+2}] - [ct^{n+1}], \\
	\:&=\: [at^{n+1}] + \cl{at^{n+1}}[t] - [ct^n] - \cl{ct^n}[t], \\
	\:&=\: d_2[\alpha_n] + (\cl{at^{n+1}} - \cl{ct^n}) [t].
\end{align*}	
The equality $\cl{at^{n+1}} =  \cl{ct^n}$ holds in~$\M$ by definition, yielding $d_2[\alpha_{n+1}]  = d_2[\alpha_{n}] $. As a consequence, the $\Zb\M$-module~$\ker d_2$ is generated by the elements $[\alpha_n]-[\alpha_0]$. Since the $\Zb\M$-module~$\ker d_1$ is equal to~$\im d_2$, hence isomorphic to $\Zb\M[\Sigma_2]/\ker d_2$, it follows that~$\ker d_1$ is generated by~$[\alpha_0]$ only, so that, by Lemma~\ref{lemme_pl_n}, the monoid~$\M$ is of homological type left-$\FP_2$. This can also be obtained by simply observing that~$\M$ admits the finite presentation $\pres{a,c,t}{\alpha_0}$.

Now, let us consider the monoid~$\M$ presented by the $2$-polygraph
\[
\Sigma \:=\: \bigpres {a,b,t} {at^nb \odfll{\alpha_n} 1 \,,\, n\in\Nb}.
\]
The monoid~$\M$ is of homological type left-$\FP_1$, but not left-$\FP_2$. This is proved by showing that~$\ker d_1$ is not finitely generated as a $\Zb\M$-module, which is tedious by direct computation in this case. Another way to conclude is to extend the partial resolution of Proposition~\ref{Proposition:LowSquierResolution} by one dimension: it will then be sufficient to compute~$\im d_3$, which is trivial in this case because~$\Sigma$ has no critical branching, so that $\ker d_2=0$ and, as a consequence, $\ker d_1$ is isomorphic to $\Zb\M[\Sigma_2]$. Convergent presentations provide a method to obtain such a length-three partial resolution.

\subsection{Squier's homological theorem}
\label{SquierHomologicalTheorem}

\subsubsection{Coherent presentations and partial resolutions of length~\pdf{3}}

Let~$\M$ be a monoid and let~$\Sigma$ be a coherent presentation of~$\M$. Let us extend the partial resolution of \ref{Proposition:LowSquierResolution} into the resolution of length~$3$
\[
\xymatrix{
\Zb\M[\Sigma_3]
	\ar [r] ^-{d_3}
& \Zb\M[\Sigma_2]
	\ar [r] ^-{d_2}
& \Zb\M[\Sigma_1]
	\ar [r] ^-{d_1}
& \Zb\M
	\ar [r] ^-{\epsilon}
& \Zb
	\ar [r]
& 0,
}
\]
where the $\Zb\M$-module $\Zb\M[\Sigma_3]$ is the free $\Zb\M$-module over~$\Sigma_3$, formed by the linear combination of elements~$u[\gamma]$,  with~$u$ in~$\M$ and~$\gamma$ a $3$-cell of~$\Sigma$. The boundary map~$d_3$ is defined, for every $3$-cell~$\gamma$ of~$\Sigma$, by
\[
d_3[\gamma] \:=\: [s_2(\gamma)] - [t_2(\gamma)].
\]
The bracket~$[\cdot]$ is extended to $3$-cells of~$\tck{\Sigma}$ thanks to the relations
\[
[uAv] \:=\: \cl{u}[A],
\qquad
[A\star_1 B] \:=\: [A] + [B],
\qquad
[A\star_2 B] \:=\: [A] + [B],
\]
for all $1$-cells~$u$ and~$v$ and $3$-cells~$A$ and~$B$ of~$\tck{\Sigma}$ such that the composites are defined. In particular, the latter relation implies $[1_f]=0$ for every $2$-cell~$f$ of~$\tck{\Sigma}$.
We check, by induction on the size, that $d_3[A] = [s_2(A)] - [t_2(A)]$ holds for every $3$-cell~$A$ of~$\tck{\Sigma}$.

\begin{proposition}
\label{Proposition:LowSquierResolution2}
Let~$\M$ be a monoid and let~$\Sigma$ be a coherent presentation of~$\M$. The sequence of $\Zb\M$-modules 
\[
\xymatrix{
\Zb\M[\Sigma_3]
	\ar [r] ^-{d_3}
& \Zb\M[\Sigma_2]
	\ar [r] ^-{d_2}
& \Zb\M[\Sigma_1]
	\ar [r] ^-{d_1}
& \Zb\M
	\ar [r] ^-{\epsilon}
& \Zb
	\ar [r]
& 0
}
\]
is a partial free resolution of length~$3$ of~$\Zb$.
\end{proposition}

\begin{proof}
We proceed with the same notations as the ones of the proof of Proposition~\ref{Proposition:LowSquierResolution}, with the extra hypothesis that~$\sigma$ is a left normalisation strategy for~$\Sigma$. This implies that $i_2(u[v])=[\sigma(\rep{u}v)]$ holds for all~$u$ in~$\M$ and~$v$ in~$\Sigma_1^*$, by induction on the length of~$v$.

We have $d_2d_3=0$ because $s_1s_2=s_1t_2$ and $t_1s_2=t_1t_2$. Then, we define the following morphism of $\Zb$-modules
\[
\begin{array}{r c l}
{\Zb\M[\Sigma_2]} &\:{\ofll{i_3}}\: &{\Zb\M[\Sigma_3]} \\
{u[\alpha]} &\:\longmapsto\: &{[\sigma(\rep{u}\alpha)]}
\end{array}
\]
where $\sigma(\rep{u}\alpha)$ is a $3$-cell of~$\tck{\Sigma}$ with the following shape, with $v=s(\alpha)$ and $w=t(\alpha)$:
\[
\xymatrix @!C @C=2em @R=2em {
& {\rep{u}w}
	\ar@2@/^/ [dr] ^-{\sigma(\rep{u}w)}
\\
{\rep{u}v}
	\ar@2@/^/ [ur] ^-{\rep{u}\alpha}
	\ar@2@/_/ [rr] _-{\sigma(\rep{u}v)} ^-{}="tgt"
&& {\rep{uv}}
\ar@3 "1,2"!<-10pt,-20pt>;"tgt"!<-10pt,15pt> ^-{\sigma(\rep{u}\alpha)}
}
\]
Let us note that such a $3$-cell necessarily exists in~$\tck{\Sigma}$ because~$\Sigma_3$ is a homotopy basis of~$\tck{\Sigma}$. Then we have, on the one hand, 
\[
i_2d_2(u[\alpha]) \:=\: i_2(u[v]-u[w]) \:=\: [\sigma(\rep{u}v)] - [\sigma(\rep{u}w)]
\]
and, on the other hand,
\begin{align*}
d_3i_3(u[\alpha]) 
	&\:=\: [\rep{u}\alpha \star_1 \sigma(\rep{u}w)] - [\sigma(\rep{u}v)], \\
	&\:=\: u[\alpha] + [\sigma(\rep{u}w)] - [\sigma(\rep{u}v)]. \\
\end{align*}
Hence $d_3 i_3 + i_2 d_2=\id_{\Zb\M[\Sigma_2]}$, concluding the proof.
\end{proof}

\subsubsection{Remark}

The proof of Proposition~\ref{Proposition:LowSquierResolution2} uses the fact that~$\Sigma_3$ is a homotopy basis to produce, for every $2$-cell~$\alpha$ of~$\Sigma$ and every~$u$ in~$\M$, a $3$-cell $\sigma(\rep{u}\alpha)$ with the required shape. The hypothesis on~$\Sigma_3$ could thus be modified to only require the existence of such a $3$-cell in~$\tck{\Sigma}$: however, it is proved in~\cite{GuiraudMalbos12advances} that this implies that~$\Sigma_3$ is a homotopy basis.

\medskip
From Proposition~\ref{Proposition:LowSquierResolution2}, we deduce

\begin{theorem}[{\cite[Theorem~3.2]{CremannsOtto94}}, {\cite[Theorem~3]{Lafont95}}, \cite{Pride95}]
\label{Theorem:TDFimpliesFP3}
Let~$\M$ be a finitely presented monoid. If~$\M$ is of finite derivation type, then it is of homological type left-$\FP_3$.
\end{theorem}

By Theorem~\ref{Theorem:Squier}, this implies

\begin{theorem}[{\cite[Theorem~4.1]{Squier87}}]
\label{Theorem:SquierAb1}
If a monoid admits a finite convergent presentation, then it is of homological type left-$\FP_3$.
\end{theorem}

\subsubsection{Example}
\label{Example:Monoidaa}

Let us consider the monoid~$\M$ with the convergent presentation
\[
\pres {a} {aa\odfll{\mu}a}.
\]
With the leftmost normalisation strategy~$\sigma$, we get, writing the $2$-cell~$\mu$ as a string diagram~\twocell{mu}:
\[
\sigma(a) \:=\: 1_a
\qquad
\sigma(aa) \:=\: \twocell{mu}
\qquad
\sigma(aaa) \:=\: \mu a \star_1 \mu \:=\: \twocell{(mu *0 1) *1 mu} \,.
\]
The presentation has exactly one critical branching, whose corresponding generating confluence can be written in the  two equivalent ways
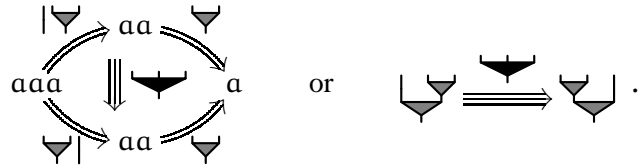
\[
\vcenter{\xymatrix @!C @R=1em @C=1em {
& aa
  \ar@2@/^/ [dr]^{\twocell{mu}}
  \ar@3 []!<-7.5pt,-12.5pt>;[dd]!<-7.5pt,12.5pt> ^-*+{\twocell{asso}}
\\
aaa 
  \ar@2@/^/ [ur]^{\twocell{1}\:\twocell{mu}} 
  \ar@2@/_/ [dr]_{\twocell{mu}\:\twocell{1}}
&& a 
\\
& aa
  \ar@2@/_/ [ur]_{\twocell{mu}}
}}
\qquad\text{or}\qquad
\vcenter{\xymatrix@C=3em{
{\twocell{(1 *0 mu) *1 mu}}  
	\ar@3 [r] ^-*+{\twocell{asso}}
& {\twocell{(mu *0 1) *1 mu}}
}}\;.
\]
The $\Zb\M$-module~$\ker d_2$ is generated by
\begin{align*}
d_3\big[\twocell{asso}\big] &=
\big[\;\twocell{(1 *0 mu) *1 mu}\;\big]
- 
\big[\;\twocell{(mu *0 1) *1 mu}\;] \\
&=
\big[\twocell{1}\:\twocell{mu} \big] + \big[\twocell{mu}\big] 
- 
\big[\twocell{mu}\:\twocell{1}\big] - \big[\twocell{mu}\big]\\
&= a \big[\twocell{mu}\big] - \big[\twocell{mu}\big].
\end{align*}

\subsubsection{Resolutions from convergent presentations}

In~\cite{GuiraudMalbos12advances}, the results presented here are generalised to produce a free resolution 
\[
\xymatrix{
(\cdots)
	\ar [r] ^-{d_{n+1}}
& \Zb\M[\Sigma_n]
	\ar [r] ^-{d_n}
& \Zb\M[\Sigma_{n-1}]
	\ar [r] ^-{d_{n-1}}
& (\cdots)
	\ar [r] ^-{d_2}
& \Zb\M[\Sigma_1]
	\ar [r] ^-{d_1}
& \Zb\M[\Sigma_0]
	\ar [r] ^-{\epsilon}
& \Zb
	\ar [r]
& 0
}
\]
of the trivial $\Zb\M$-module~$\Zb$ from a convergent presentation of a monoid~$\M$. For~$k\geq 4$, the $\Zb\M$-module $\Zb\M[\Sigma_k]$ is defined as the free $\Zb\M$-module over a family~$\Sigma_k$ of $k$-cells obtained from the $(k-1)$-fold critical branchings.  For example, the $4$-cell
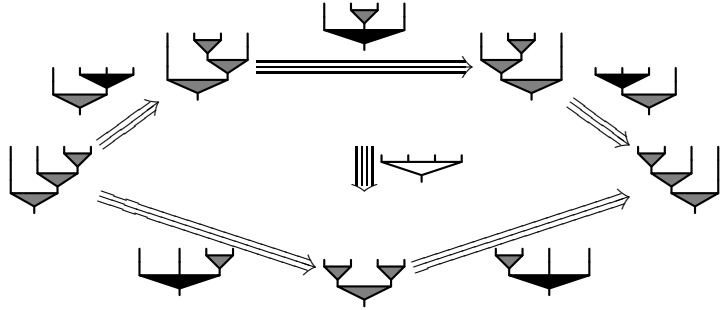
\[
\xymatrix @C=2em @R=1em{
& { \twocell{(1 *0 mu *0 1) *1 (1 *0 mu) *1 mu}}
	\ar@3 [rr] 
		^-*+{\twocell{(1 *0 mu *0 1) *1 asso}} _-{}="src"
&& { \twocell{(1 *0 mu *0 1) *1 (mu *0 1) *1 mu}}
	\ar@3 [dr] ^-{\twocell{(asso *0 1) *1 mu}} 
\\
{ \twocell{(2 *0 mu) *1 (1 *0 mu) *1 mu}}
	\ar@3 [ur] ^-{\twocell{(1 *0 asso) *1 mu}}
	\ar@3 [drr] _-*+{\twocell{(2 *0 mu)*1 asso}}
&&&& { \twocell{(mu *0 2) *1 (mu *0 1) *1 mu}}
\\
&& { \twocell{(mu *0 mu) *1 mu}}
	\ar@3 [urr] _-*+{\twocell{(mu *0 2)*1 asso}}
\ar@4{ ->} "src"!<0pt,-30pt>; "3,3"!<0pt,35pt>  ^-*+{\twocell{penta}}
}
\]
is the only element of~$\Sigma_4$ in the case of the monoid of Example~\ref{Example:Monoidaa}.
A related resolution is obtained in~\cite{Kobayashi90} using antichains instead of $k$-fold critical branchings.

\subsubsection{Other homological finiteness conditions}
\label{AnotherHomologicalType}

In the definition \ref{Definition:LeftFPn} of homological type left-$\FP_n$ for a monoid~$\M$, the replacement of left modules by right modules, bimodules or natural systems gives the definitions of the homological types right-$\FP_n$, bi-$\FP_n$ and~$\FP_n$, for every~$0\leq n\leq\infty$. We refer the reader to~\cite[Section 5.2]{GuiraudMalbos12advances} for the relations between these different finiteness conditions. In particular, for~$n=3$, all of these homotopical conditions are consequences of the finite derivation type property. The proof is similar to the one for the left-$\FP_3$ property in Section~\ref{SquierHomologicalTheorem}: for example, in the case of the right-$\FP_3$ property, we consider right modules and, to get the contracting homotopy, we construct a right normalisation strategy~$\sigma$ by defining a $3$-cell $\sigma(\alpha\rep{u})$ with shape
\[
\xymatrix @!C @C=2em @R=2em {
& {w\rep{u}}
	\ar@2@/^/ [dr] ^-{\sigma(w\rep{u})}
\\
{v\rep{u}}
	\ar@2@/^/ [ur] ^-{\alpha\rep{u}}
	\ar@2@/_/ [rr] _-{\sigma(v\rep{u})} ^-{}="tgt"
&& {\rep{vu}}
\ar@3 "1,2"!<-10pt,-17.5pt>;"tgt"!<-10pt,12.5pt> ^-{\sigma(\alpha\rep{u})}
}
\]
for any generating $2$-cell $\alpha:v\dfl w$ and~$u$ in the monoid.

\section{Squier's example and variant}

\subsection{Squier's example}
\label{Example:MonoidSk}

In \cite{Squier87}, Squier defines, for every~$k\geq 1$, the monoid~$\S_k$ presented by
\[
\bigpres
	{a, b ,t , x_1, \dots, x_k, y_1, \dots, y_k}
	{(\alpha_n)_{n\in\Nb},\; (\beta_i)_{1\leq i\leq k},\; (\gamma_i)_{1\leq i\leq k},\; (\delta_i)_{1\leq i\leq k},\; (\epsilon_i)_{1\leq i\leq k}}
\]
with
\[
at^nb \odfll{\alpha_n} 1,\quad
x_ia \odfll{\beta_i} atx_i,\quad
x_it \odfll{\gamma_i} tx_i,\quad
x_ib \odfll{\delta_i} bx_i,\quad
x_iy_i \odfll{\epsilon_i} 1.
\]
In~\cite{Squier94}, Squier proves the following properties for~$\S_1$. With similar arguments, the result extends to every monoid~$\S_k$, for~$k\geq 1$. 

\begin{theorem}[{\cite[Theorem~6.7, Corollary~6.8]{Squier94}}]
\label{Theorem:MonoidS1}
For every~$k\geq 1$, the monoid~$\S_k$ satisfies the following properties:
\begin{enumerate}[{\bf i)}]
\item it is finitely presented,
\item it has a decidable word problem,
\item it is not of finite derivation type,
\item it admits no finite convergent presentation.
\end{enumerate}
\end{theorem}

In~\cite{Squier87}, Squier had already proved 

\begin{proposition}[{\cite[Example 4.5.]{Squier87}}]
\label{Proposition:MonoidSk}
For~$k\geq 2$, the monoid~$\S_k$ is not of finite homological type left-$\FP_3$ and, as a consequence, it does not admit a finite convergent presentation.
\end{proposition}

Proposition~\ref{Proposition:MonoidSk} does not hold for~$\S_1$. Indeed, this monoid is of homological type left-$\FP_{\infty}$, {\cite[Example 4.5]{Squier87}}. This proves that, if finite derivation type implies left-$\FP_3$, the reverse implication does not hold for general monoids. However, in the special case of groups, the property of having finite derivation type is equivalent to the homological finiteness condition left-$\FP_3$~\cite{CremannsOtto96}. The latter result is based on the Brown-Huebschmann isomorphism between homotopical and homological syzygies~\cite{BrownHuebschmann82}.

\subsection{Proof of Theorem~\ref{Theorem:MonoidS1}}

Let us prove the result in the case of the monoid~$\S_1$, with the following infinite presentation:
\[
\Sq \:=\: \pres {a,b,t,x,y} {(\alpha_n)_{n\in \Nb},\; \beta,\; \gamma,\; \delta, \;\epsilon}
\]
with
\[
at^nb \odfll{\alpha_n} 1,\quad
xa\odfll{\beta} atx,\quad 
xt\odfll{\gamma} tx,\quad
xb\odfll{\delta} bx,\quad
xy\odfll{\epsilon} 1.
\]
In what follows, we denote by $\gamma_n:xt^n\dfl t^n x$ the $2$-cell of $\Sq^*$ defined by induction on~$n$ as follows:
\[
\gamma_0 \:=\: 1_x
\qquad
\text{and}
\qquad
\gamma_{n+1} \:=\: \gamma t^n \star_1 t \gamma_n.
\]
For every~$n$, we write $f_n:xat^b\dfl at^{n+1}bx$ the $2$-cell of $\Sq^*$ defined as the following composite: 
\[
\xymatrix @C=3em {
xat^nb
	\ar@2 [r] ^-{\beta t^n b} 
& atxt^n b
	\ar@2 [r] ^-{at\gamma_n b}
& at^{n+1}xb 
	\ar@2 [r] ^-{at^{n+1}\delta}
& at^{n+1}bx.
}
\]
We note that~$f_n$ contains no $2$-cell~$\alpha_k$.

\begin{proposition}
\label{propS1finitetype}
The monoid~$\S_1$ admits the finite presentation $\tilde{\Sq} = \pres {a, b, t, x, y} {\alpha_0, \beta, \gamma, \delta, \epsilon}$.
\end{proposition}

\begin{proof}
For every natural number~$n$, we consider the following $2$-sphere of~$\tck{\Sq}$:
\begin{equation}
\label{An}
\vcenter{\xymatrix @!C @C=0em @R=2em {
& at^{n+1}bxy
	\ar@2 [rr] ^-{at^{n+1}b\epsilon} 
&& at^{n+1}b
	\ar@2@/^/ [dr] ^-{\alpha_{n+1}}
\\
xat^nby 
	\ar@2@/^/ [ur] ^-{f_n y}
	\ar@2@/_/ [drr] _-{x\alpha_n y} 
&&&& 1
\\
&& xy
	\ar@2@/_/ [urr] _-{\epsilon}
}}
\end{equation}
Thus, the $2$-cell~$\alpha_{n+1}$ is parallel to the composite $2$-cell
\begin{equation}
\label{alphan}
at^{n+1}b\epsilon^- \star_1 f_n^- y \star_1 x\alpha_n y \star_1 \epsilon.
\end{equation}
Since~$f_n$ contains no~$\alpha_k$, this proves the result by induction on~$n$.
\end{proof}

\begin{proposition}
\label{propS1convergent}
The $2$-polygraph~$\Sq$ is convergent and Squier completion of~$\Sq$ contains a $3$-cell~$A_n$ with shape
\[
\xymatrix @!C @C=2em {
& at^{n+1}bx
	\ar@2@/^/ [dr] ^-{\alpha_{n+1}x}
\\
xat^nb
	\ar@2@/^/ [ur] ^-{f_n}
	\ar@2@/_/ [rr] _-{x\alpha_n} ^{}="tgt"
\ar@3 "1,2"!<0pt,-15pt>;"tgt"!<0pt,10pt> ^-*+{A_n}
&& x
}
\]
for every natural number~$n$.
\end{proposition}

\begin{proof}
Let us prove that~$\Sq$ terminates. For that, we build a termination order based on derivations, similar to the method of~\cite[Theorem~4.2.1]{GuiraudMalbos09} for $3$-polygraphs. We associate, to every $1$-cell~$u$ of~$\Sq^*$, two maps
\[
u_* \::\: \Nb \:\fl\: \Nb
\qquad\text{and}\qquad
\partial(u) \::\: \Nb \:\fl\: \Nb
\]
as follows. First, we define them on the $1$-cells of~$\Sq$:
\[
x_*(n) = n+1,
\qquad a_*(n) = b_*(n) = t_*(n) = y_*(n) = n,
\]
\[
\partial(a)(n) = 3^n,
\qquad \partial(b)(n) = \partial(y)(n) = 2^n,
\qquad \partial(t)(n) = \partial(x)(n) = 0.
\]
Then, we extend the mappings to every $1$-cell of~$\Sq^*$ thanks to the following relations:
\[
1_*(n) = n,
\qquad (uv)_*(n) = v_*(u_*(n)),
\qquad \partial(1)(n) = 0,
\qquad \partial(uv)(n) = \partial(u)(n) + \partial(v)(u_*(n)).
\]
We compare parallel $1$-cells of~$\Sq^*$ by the order generated by~$u<v$ if~$u_*\leq v_*$ and~$\partial(u)<\partial(v)$. The defining relations of~$(\cdot)_*$ and~$\partial$ imply that the composition of $1$-cells of~$\Sq^*$ is strictly monotone in both arguments. The natural order on~$\Nb$ implies that every decreasing family of parallel $1$-cells of~$\Sq^*$ is stationary. To get a termination order, hence the termination of~$\Sq$, there remains to check that~$u>v$ for every $2$-cell~$u\dfl v$ of~$\Sq$. Indeed, we check that the following (in)equalities are satisfied:
\[
(at^k b)_*(n) = n = 1_*(n),
\qquad
(xa)_*(n) = n+1 = (atx)_*(n), 
\qquad
(xt)_*(n) = n+1 = (tx)_*(n),
\]
\[
(xb)_*(n) = n+1 = (bx)_*(n),
\qquad
(xy)_*(n) = n+1 > n = 1_*(n),
\]
and
\[
\partial(at^k b)(n) = 3^n + 2^n > 0 = \partial(1)(n),
\qquad
\partial(xa)(n) = 3^{n+1} > 2^n + 3^n = \partial(atx)(n),
\]
\[
\partial(xt)(n) = 2^{n+1} > 2^n = \partial(tx)(n),
\qquad
\partial(xb)(n) = 2^{n+1} > 2^n = \partial(bx)(n),
\]
\[
\partial(xy)(n) = 2^{n+1} > 0 = \partial(1)(n).
\]
Let us prove that~$\Sq$ is confluent and compute Squier completion of~$\Sq$. The $2$-polygraph~$\Sq$ has exactly one critical branching $(\beta t^n b,x\alpha_n)$ for every natural number~$n$, and each of those critical branchings is confluent, yielding the $3$-cell~$A_n$. We conclude thanks to Theorem~\ref{Theorem:Squier1}.
\end{proof}

\begin{proposition}
\label{propS1decidable}
The monoid~$\S_1$ has a decidable word problem.
\end{proposition}

\begin{proof}
The convergent presentation~$\Sq$ of~$\S_1$ is infinite, so that the normal-form procedure cannot be applied. However, the sources of the $2$-cells~$\alpha_n$ are exactly the elements of the regular language~$at^*b$. This implies that the sources of the $2$-cells of~$\Sq$ form a regular language over the finite set~$\ens{a,b,t,x,y}$: by~\cite[Proposition~3.6]{OttoKatsuraKobayashi98}, this implies that the word problem of~$\S_1$ is decidable.
\end{proof}

In order to show that~$\S_1$ is not of finite derivation type, by Theorem~\ref{Theorem:Squier2}, it is sufficient to check that the finite presentation~$\tilde{\Sq}$ of~$\S_1$ given in Proposition~\ref{propS1finitetype} admits no finite homotopy basis. We denote by 
\[
\pi : \tck{\Sq} \longrightarrow \tck{\tilde{\Sq}}
\]
the projection that sends the $2$-cells~$\beta$, $\gamma$, $\delta$ and~$\epsilon$ to themselves and whose value on~$\alpha_n$ is given by induction on~$n$, thanks to~\eqref{alphan}, i.e.\
\[
\pi(\alpha_0) \:=\: \alpha
\qquad\text{and}\qquad
\pi(\alpha_{n+1}) \:=\: g_n^- \star_1 x\pi(\alpha_n)y \star_1 \epsilon
\]
where 
\[
g_n \:=\: f_n y \star_1 at^{n+1}b\epsilon.
\]
By application of the homotopy basis transfer theorem~\ref{Theorem:HomotopyBasesTransfer} to~$\Sq$ and~$\tilde{\Sq}$, with~$F$ the canonical inclusion of~$\tck{\tilde{\Sq}}$ into~$\tck{\Sq}$, with~$G=\pi$ and with~$\tau$ mapping each $1$-cell~$u$ of~$\tck{\tilde{\Sq}}$ to~$1_u$, we obtain

\begin{lemma}
The monoid~$\S_1$ admits the coherent presentation $\cohpres {a, b, t, x, y}{\alpha_0, \beta, \gamma, \delta, \epsilon}{(\tilde{A}_n)_{n\in\Nb}}$ where~$\tilde{A}_n$ is the $3$-cell
\[
\xymatrix @!C @C=2em {
& at^{n+1}bx
	\ar@2@/^/ [dr] ^-{\pi(\alpha_{n+1})x}
\\
xat^nb
	\ar@2@/^/ [ur] ^-{f_n}
	\ar@2@/_/ [rr] _-{x\pi(\alpha_n)} ^{}="tgt"
\ar@3 "1,2"!<0pt,-15pt>;"tgt"!<0pt,10pt> ^-*+{\tilde{A}_n}
&& x
}
\]
\end{lemma}

Let us now deduce that~$\S_1$ is not of finite derivation type. A direct proof is given in~\cite[Theorem~6.7]{Squier87}, see also~\cite[§6]{Lafont95}. However, we can use an argument coming from the dimension above, to explicit as $4$-cells the relations between the $3$-cells. We choose this proof to incite the reader to continue the exploration of the higher dimensions of rewriting, e.g.\ with~\cite{Metayer03,GuiraudMalbos12advances}. Precisely, we use the fact that the triple critical branchings (the minimum overlaps of three rewriting steps) of a reduced convergent $2$-polygraph~$\Sigma$ induce a homotopy basis of Squier completion~$\Sr(\Sigma)$: all the parallel $3$-cells of~$\tck{\Sr(\Sigma)}$ are equal up to the $4$-cells generated by the triple critical branchings~\cite[Proposition~4.4.4]{GuiraudMalbos12advances}. This implies

\begin{proposition}
\label{Proposition:CriticalTriples}
Let~$\Sigma$ be a reduced convergent $2$-polygraph with no triple critical branching. Then all the parallel $3$-cells of the free $(3,1)$-category~$\tck{\Sr(\Sigma)}$ are equal.
\end{proposition}

We can now conclude with

\begin{proposition}
\label{propS1nottdf}
The monoid~$\S_1$ is not of finite derivation type.
\end{proposition}

\begin{proof}
We note that the $3$-cell $A_n y\star_1 \epsilon$ of~$\tck{\Sr(\Sq)}$ has the $2$-sphere~\eqref{An} as boundary:
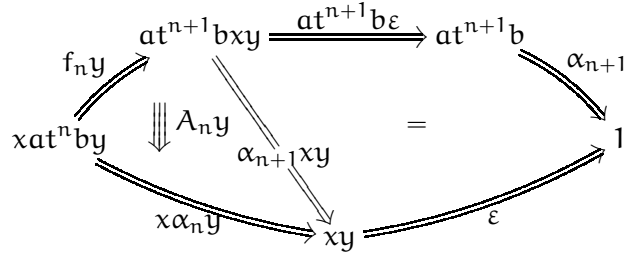
\[
\xymatrix @!C @C=0em @R=2em {
& at^{n+1}bxy
	\ar@2 [rr] ^-{at^{n+1}b\epsilon} 
	\ar@2 [ddr] |(0.6){\alpha_{n+1}xy}
&& at^{n+1}b
	\ar@2@/^/ [dr] ^-{\alpha_{n+1}}
\\
xat^nby 
	\ar@2@/^/ [ur] ^-{f_n y}
	\ar@2@/_/ [drr] _-{x\alpha_n y} ^{}="tgt"
\ar@3 "1,2"!<-15pt,-25pt>;"tgt"!<-13.18pt,20pt>  ^(0.4)*+{A_n y}
&&&& 1
\\
&& xy
	\ar@2@/_/ [urr] _-{\epsilon}
	\ar@{} [uur] |(0.55){=}
}
\]
As a consequence, for every natural number~$n$, the $3$-cell $g_n^- \star_1 A_n y \star_1 \epsilon$ has source~$\alpha_{n+1}$ and target $g_n^-\star_1 x\alpha_n y \star_1 \epsilon$. We define the $3$-cell~$B_n$ of~$\tck{\Sr(\Sq)}$ by induction on~$n$ as
\[
B_0 \:=\: 1_{\alpha_0}
\qquad\text{and}\qquad 
B_{n+1} \:=\: \big( g_n ^- \star_1 A_n y \star_1 \epsilon \big) \star_2 B_n
\]
so that~$B_n$ has source~$\alpha_n$ and target~$\pi(\alpha_n)$, by definition of~$\pi$. As a consequence, the $3$-cell~$\tilde{A}_n$, when seen as a $3$-cell of~$\tck{\Sr(\Sq)}$ through the canonical inclusion, is parallel to the following composite $3$-cell:
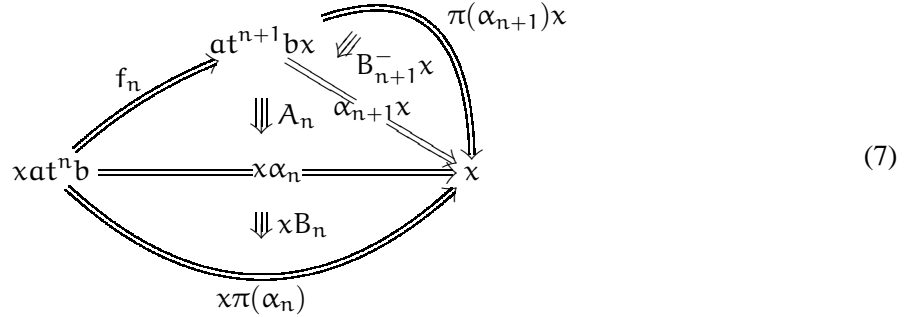
\begin{equation}
\label{composite}
\vcenter{\xymatrix @!C @C=3em @R=3em {
& at^{n+1}bx
	\ar@2@/^8ex/ [dr] ^(0.6){\pi(\alpha_{n+1})x} _(0.3){}="src1"
	\ar@2 [dr] |-{\alpha_{n+1}x} ^(0.3){}="tgt1"
\ar@3 "src1"!<-5pt,-10pt>;"tgt1"!<5pt,10pt> ^-{B_{n+1}^-x}
\\
xat^nb
	\ar@2@/^/ [ur] ^-{f_n}
	\ar@2 [rr] |-{x\alpha_n} ^{}="tgt" _{}="src2"
	\ar@2@/_8ex/ [rr] _-{x\pi(\alpha_n)} ^{}="tgt2"
\ar@3 "1,2"!<0pt,-20pt>;"tgt"!<0pt,15pt> ^-*+{A_n}
\ar@3 "src2"!<0pt,-15pt>;"tgt2"!<0pt,15pt> ^-*+{xB_n}
&& x
}}
\end{equation}
We observe that the convergent $3$-polygraph is reduced and has no critical triple branching. Thus, as a consequence of Proposition~\ref{Proposition:CriticalTriples}, all the parallel $3$-cells of~$\tck{\Sr(\Sq)}$ are equal. This implies that~$\tilde{A}_n$ is equal to the composite~\eqref{composite}. Expanding the definition of~$B_{n+1}$, we get:
\begin{equation}
\label{piAn}
\tilde{A}_n \:=\: \big( f_n \star_1 B_n^- x \big) \star_2 \big( f_n \star_1 g_n^-x \star_1 A_n^- yx \star_1 \epsilon x \big) \star_2 A_n \star_2 xB_n.
\end{equation}

Now, let us assume that~$\tilde{\Sq}$ admits a finite homotopy basis. By Proposition~\ref{Proposition:FiniteBasisExtracted}, there exists a natural number~$n$ such that the $3$-cells~$\tilde{A}_0$, \dots, $\tilde{A}_n$ form a homotopy basis of~$\tck{\tilde{\Sq}}$. In particular, the $3$-cell~$\tilde{A}_{n+1}$ is parallel to a composite~$W$ of the $3$-cells~$\tilde{A}_0$, \dots, $\tilde{A}_n$, hence it is equal to~$W$ in~$\tck{\Sq}$. 

Thus, on the one hand, by application of~\eqref{piAn} to each of~$\tilde{A}_0$, \dots, $\tilde{A}_n$, and by definition of $B_0$, \dots, $B_n$, we get that~$\tilde{A}_{n+1}$ is a composite of the $3$-cells~$A_0$, \dots, $A_n$. But, on the other hand, the relation~\eqref{piAn} tells us that~$\tilde{A}_{n+1}$ is equal to a composite 
\begin{equation}
\label{Equation:sansLesCrochets}
\tilde{A}_{n+1} \:=\: C \star_2 (h \star_1 A_{n+1}^-yx \star_1 k) \star_2 A_{n+1} \star_2 D
\end{equation}
where the $3$-cells~$C$ and~$D$ contain~$A_0$, \dots, $A_n$ only. 

To prove that this leads to a contradiction, let us consider the free right $\Zb\S_1$-module $\Zb\S_1[\Gamma]$ over the homotopy basis $\Gamma=\ens{A_n, n\in \Nb}$. We define a map
\[
[\cdot] : \tck{\Sr(\Sq)} \longrightarrow \Zb\S_1[\Gamma]
\]
thanks to the relations
\[
[uAv] \:=\: [A]\cl{v},
\qquad
[A\star_1 B] \:=\: [A] + [B],
\qquad
[A\star_2 B] \:=\: [A] + [B],
\]
for all $1$-cells~$u$ and~$v$ and $3$-cells~$A$ and~$B$ of~$\tck{\Sr(\Sq)}$ such that the composites are defined.
From~\eqref{Equation:sansLesCrochets}, we deduce that 
\begin{equation}
\label{Equation:avecLesCrochets}
[\tilde{A}_{n+1}] = [C] - [A_{n+1}]\cl{yx} + [A_{n+1}] + [D]
\end{equation}
holds in $\Zb\S_1[\Gamma]$.
Since the $3$-cell~$\tilde{A}_{n+1}$ is a composite of the $3$-cells~$A_0$, \dots, $A_n$, we have that $[\tilde{A}_{n+1}]$ is a linear combination of~$[A_0]$, \dots, $[A_n]$. Since the right $\Zb\S_1$-module $\Zb\S_1[\Gamma]$ is free, it follows from~\eqref{Equation:avecLesCrochets} that
\[
\cl{yx} = 1
\]
holds in~$\S_1$. However, the $1$-cells~$yx$ and~$1$ are distinct normal forms of the convergent presentation~$\Sq$ of~$\S_1$. This means that~$\cl{yx}$ is distinct from~$1$ in~$\S_1$ and, thus, leads to a contradiction, so that we conclude that~$\tck{\tilde{\Sq}}$ does not admit a finite homotopy basis.
\end{proof}

Finally, by Theorem~\ref{Theorem:Squier}, we get:

\begin{corollary}
The monoid~$\S_1$ admits no finite convergent presentation.
\end{corollary}

\subsection{A variant of Squier's example}

Let us consider the monoid~$\M$ presented by the following $2$-polygraph from~\cite{LafontProute91,Lafont95}:
\[
\LP \:=\: \bigpres {a,b,c,d,d'} {ab\odfll{\alpha_0} a,\; da\odfll{\beta} ac,\; d'a\odfll{\gamma} ac}.
\]
The monoid~$\M$ has similar properties to Squier's example: it admits a finite presentation, it has a decidable word problem, yet it is not of finite derivation type and, as a consequence, it does not admit a finite convergent presentation. 

To prove these facts, the $2$-polygraph~$\LP$ is completed, by Knuth-Bendix procedure, into the infinite convergent $2$-polygraph 
\[
\widecheck{\LP} \:=\: 
	\bigpres {a,b,c,d,d'} {(ac^nb\odfll{\alpha_n} ac^n)_{n\in \Nb},\; da\odfll{\beta} ac,\; d'a\odfll{\beta'} ac}.
\]
Squier completion of~$\widecheck{\LP}$ has two infinite families of $3$-cells:
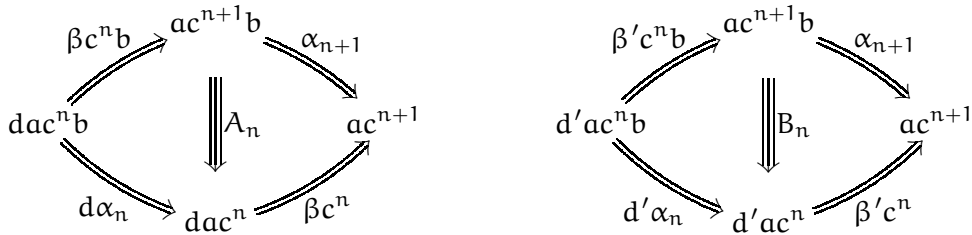
\[
\xymatrix@!C @R=2em@C=2em{
& ac^{n+1}b
	\ar@2@/^/ [dr] ^{\alpha_{n+1}} 
	\ar@3 []!<0pt,-20pt>;[dd]!<0pt,20pt> ^{A_n}
\\
dac^nb
	\ar@2@/^/ [ur] ^{\beta c^nb}
	\ar@2@/_/ [dr] _{d\alpha_n}
&& ac^{n+1}
\\
& dac^n 
	\ar@2@/_/ [ur] _{\beta c^n}
} 
\qquad\qquad
\xymatrix@!C @R=2em@C=2em{
& ac^{n+1}b
	\ar@2@/^/ [dr] ^{\alpha_{n+1}} 
	\ar@3 []!<0pt,-20pt>;[dd]!<0pt,20pt> ^{B_n}
\\
d'ac^nb
	\ar@2@/^/ [ur] ^{\beta' c^nb}
	\ar@2@/_/ [dr] _{d'\alpha_n}
&& ac^{n+1}
\\
& d'ac^n 
	\ar@2@/_/ [ur] _{\beta' c^n}
} 
\]
Moreover, the $2$-polygraph~$\widecheck{\LP}$ has no triple critical branching. In a similar way to the case of Squier's monoid~$\S_1$, we get that the (finitely generated, with a decidable word problem) monoid~$\M$ is not of finite derivation type: we prove that the $3$-cells~$B_n$ induce a projection~$\pi$ of~$\tck{\widecheck{\LP}}$ onto~$\tck{\LP}$, so that the family $(\pi(A_n))_{n\in\Nb}$ is an infinite homotopy basis of~$\tck{\LP}$. Then we prove that no finite subfamily of $(\pi(A_n))_{n\in\Nb}$ can be a homotopy basis of~$\tck{\LP}$.


\begin{small}
\bibliographystyle{amsalpha}
\bibliography{biblioihdr}
\end{small}

\vfill

\noindent \textsc{Yves Guiraud} -- {\small \textsf{yves.guiraud@pps.univ-paris-diderot.fr}} \\
\begin{small}
INRIA Paris \\ 
IRIF, CNRS UMR 8243 \\
Université Paris 7, Case 7014 \\
75205 Paris Cedex 13, France
\end{small}

\bigskip
\noindent \textsc{Philippe Malbos} -- {\small \textsf{malbos@math.univ-lyon1.fr}} \\
\begin{small}
Univ Lyon, Université Claude Bernard Lyon 1 \\
CNRS UMR 5208, Institut Camille Jordan \\
43 blvd. du 11 novembre 1918 \\
F-69622 Villeurbanne cedex, France
\end{small}

\end{document}